\documentclass[a4paper,leqno,11pt]{amsart}

\usepackage{amsfonts,amssymb,verbatim,amsmath,amsthm,latexsym,textcomp,amscd}
\usepackage{latexsym,amsfonts,amssymb,epsfig,verbatim}
\usepackage{amsmath,amsthm,amssymb,latexsym,graphics,textcomp}
\usepackage[color,matrix,arrow,all]{xy}
\usepackage{graphicx}
\usepackage[usenames,dvipsnames]{color}
\usepackage{url}
\usepackage{enumerate}
\usepackage[mathscr]{euscript}

\usepackage{tikz}

\usepackage{dsfont}

\usetikzlibrary{matrix}

\usepackage{sseq}

\definecolor{frenchblue}{rgb}{0.0, 0.45, 0.73}

\usepackage[colorlinks,citecolor=frenchblue, linkcolor=frenchblue]{hyperref}

\theoremstyle{definition}

\newtheorem{theorem}{Theorem}[section]

\newtheorem{prop}[theorem]{Proposition}

\newtheorem{lemma}[theorem]{Lemma}

\newtheorem{defn}[theorem]{Definition}

\newtheorem{rmk}[theorem]{Remark}

\newtheorem{exam}[theorem]{Example}

\newtheorem{thm}[theorem]{Theorem}

\newcommand{\C}{{\mathbb C}}

\newcommand{\Z}{{\mathbb{Z}}}

\newcommand{\upi}{\underline{\pi}}

\newcommand{\uH}{\underline{H}}

\newcommand{\uZp}{\underline{\Z/p}}

\newcommand{\Pic}{\mbox{Pic}}

\newcommand{\Ho}{\mbox{Ho}}

\newcommand{\SP}{\mbox{Sp}}

\newcommand{\F}{\mathbb F}

\newcommand\FF{{\mathcal F}}

\newcommand\GG{{\mathcal G}}

\newcommand\HH{{\mathcal H}}

\newcommand\II{{\mathcal I}}

\newcommand\LL{{\mathcal L}}

\newcommand\MM{{\mathcal M}}

\newcommand\PP{{\mathcal P}}

\newcommand\RR{{\mathcal R}}

\newcommand\PMF{{\PP\kern-2pt\MM\FF}}

\newcommand\PML{{\PP\kern-2pt\MM\LL}}

\newcommand\Mod{\operatorname{Mod}}

\newcommand\tr{\operatorname{tr}}

\newcommand{\pair}[1]{\langle #1\rangle}

\newcommand{\fsubd}{\mathrel{{\scriptstyle\searrow}\kern-1ex^d\kern0.5ex}}

\newcommand{\bsubd}{\mathrel{{\scriptstyle\swarrow}\kern-1.6ex^d\kern0.8ex}}

\newcommand{\fsubeq}{\mathrel{\raise-.7ex\hbox{$\overset{\searrow}{=}$}}}

\newcommand{\bsubeq}{\mathrel{\raise-.7ex\hbox{$\overset{\swarrow}{=}$}}}

\newcommand{\tsh}[1]{\left\{\kern-.9ex\left\{#1\right\}\kern-.9ex\right\}}

\renewcommand{\Im}{\operatorname{Im}}

\numberwithin{equation}{section}
\newenvironment{myeq}[1][]
{\stepcounter{theorem}\begin{equation}\tag{\thetheorem}{#1}}
	{\end{equation}}
\newenvironment{meq*}[1][]
{\stepcounter{theorem}\begin{equation*}\tag{\thetheorem}{#1}}
	{\end{equation*}}
\newtheorem{subsec}[theorem]{}
\newenvironment{mysubsection}[2][]
{\begin{subsec}\begin{upshape}\begin{bfseries}{#2.}
			\end{bfseries}{#1}}
		{\end{upshape}\end{subsec}}

\usepackage{tikz-cd}
\newcommand{\bZ}{\langle \Z \rangle}
\newcommand{\bZpboxM}{\bZ_p\boxtimes \uM}

\newcommand{\id}{\textup{id}}
\newcommand{\uA}{\underline{A}}

\newcommand{\uM}{\underline{M}}
\newcommand{\uN}{\underline{N}}
\newcommand{\hm}{\hat{m}}
\newtheorem{notation}[theorem]{Notation}

\tikzcdset{scale cd/.style={every label/.append style={scale=#1},
		cells={nodes={scale=#1}}}}
\newcommand{\msp}{\mspace{2mu}}
\newcommand{\mcal}[1]{\mathcal{#1}}
\usepackage{ushort}
\newcommand{\uZ}{\underline{\mathbb{Z}}}
\newcommand{\uk}{\underline{k}}
\usepackage[normalem]{ulem}
\usepackage{soulutf8}
\usepackage{bookmark}
\usepackage{mathrsfs}
\newcommand{\G}{{\mathbb G}}
\usepackage{lipsum}

\newtheorem*{theorem*}{Theorem}
\newtheorem*{thma}{Theorem A}
\newtheorem*{thmb}{Theorem B}
\newtheorem*{thmc}{Theorem C}
\newtheorem*{thmd}{Theorem D}

\theoremstyle{remark}
\swapnumbers
\newtheorem{ack}[theorem]{Acknowledgements}

\newcommand{\smas}{\wedge}

\newcommand{\res}{\textup{res}}
\newcommand{\Hom}{\textup{Hom}}
\renewcommand{\tr}{\textup{tr}}	
\renewcommand{\div}{\textup{div}}	
\newcommand{\lcm}{\textup{lcm}}
\usepackage{enumitem}
\makeatletter
\newcommand{\namelabel}[1]{%
	\phantomsection
	\renewcommand{\@currentlabel}{#1}
	\label{#1}
}
\makeatother
\usetikzlibrary{matrix}
\input xy
\xyoption{all}
\newcommand{\mbn}{M\square N}
\newcommand{\otm}{\otimes}
\usepackage[disable]{todonotes} 
\tikzset{notestyleraw/.append style={align=justify}}
\title[]{Equivariant cohomology  for cyclic groups}
\author{Samik Basu, Pinka Dey}

\makeatletter
\@namedef{subjclassname@2020}{%
  \textup{2020} Mathematics Subject Classification}
\makeatother

\address{Statistics and Mathematics Unit,
Indian Statistical Institute,
Kolkata-700108, West Bengal, India.}
\email{samik.basu2@gmail.com; samikbasu@isical.ac.in}

\address{Stat-Math Unit,
	Indian Statistical Institute,
	B. T. Road, Kolkata-700108, India.}
\email{pinkadey11@gmail.com}


\begin{document}
	\begin{abstract}
In this paper, we compute the $RO(C_n)$-graded coefficient ring of equivariant cohomology for cyclic groups $C_n$, in the case of Burnside ring coefficients, and in the case of constant coefficients. We use the invertible Mackey functors under the box product to reduce the gradings in the computation from $RO(C_n)$ to those expressable as combinations of $\lambda^d$ for divisors $d$ of $n$, where $\lambda$ is the inclusion of $C_n$ in $S^1$ as the roots of unity. We make explicit computations for the geometric fixed points for Burnside ring coefficients, and in the positive cone for constant coefficients. The positive cone is also computed for the Burnside ring in the case of prime power order, and in the case of square free order. Finally, we also make computations at non-negative gradings for the constant coefficients.
	\end{abstract}
	\maketitle
	
	\section{Introduction}

In the classical setting, equivariant homotopy theory was extremely useful in interpreting computations for classifying spaces of groups. The Atiyah-Segal completion theorem \cite{AS69} about the complex $K$-theory of $BG$, the classifying space of $G$, yields a relation between the equivariant $K$-theory of a $G$-space $X$ and the non-equivariant $K$-theory of it's Borel construction $X_{hG}$. The Borel construction is defined as 
\[ X_{hG} := X\times_G EG,\] 
where $EG\to BG$ is the universal $G$-bundle. A similar story can be weaved around the Sullivan's conjecture comparing $p$-completions of homotopy fixed points to the actual fixed points, and maps out of classifying spaces \cite[Chapter VIII]{May96}. In fact, many techniques used in equivariant stable homotopy theory were developed to understand the Segal conjecture \cite{Ada84}. There is a renewed interest in the subject arising from the use of equivariant homotopy theory to resolve the Kervaire invariant one problem \cite{HHR16}. 

In terms of computations of homotopy classes $[X,Y]$, the simplest case that may arise is when $Y$ is an Eilenberg-MacLane spectrum, and this gives rise to ``ordinary cohomology theories". It takes a little work to define these in the equivariant context, as the homotopy groups carry a Mackey functor \cite{TW95} structure coming from the restriction, transfer, and conjugation maps. It is fair then to consider Eilenberg MacLane spectra $H\uM$ for Mackey functors $\uM$,  as those whose homotopy groups are concentrated in degree $0$, where it equals $\uM$ \cite{GM95a}. The ordinary cohomology theory thus defined by a Mackey functor, carries a Mackey functor structure, as well as an extension of grading from the integers to the real representation ring $RO(G)$. 

For a finite group $G$, and a real representation $V$ of $G$, one may endow $V$ with a $G$-invariant inner product. This allows us to define 
\[ D(V)= \{ v\in V \mid \pair{v,v} \leq 1 \}, ~\mbox{  }~  S(V)= \{ v\in V \mid \pair{v,v} = 1 \}, ~\mbox{  }~ S^V = D(V)/S(V),\] 
the unit disk in $V$, the unit sphere in $V$, and the one-point compactification of $V$, respectively. In the equivariant stable homotopy category, the $V$-fold suspension functors $X \mapsto \Sigma^V X = X\wedge S^V$ are invertible. We refer to \cite{MM02,LMS86} for the construction of the equivariant stable homotopy category, as the homotopy category of $G$-spectra denoted by $\SP^G$. As the $V$-fold suspension functors are invertible, one has a spectrum $S^\alpha$ for $\alpha \in RO(G)$, which may be used to equip the homotopy classes of maps with a $RO(G)$-grading. 

Given a $G$-space $X$, the ordinary cohomology with coefficients in the Mackey functor $\uM$ at $\alpha\in RO(G)$ is denoted by $H^\alpha_G(X;\uM):= \pi_{-\alpha}^G F(X,H\uM)$ for the underlying group, and $\uH^\alpha_G(X;\uM):=  \upi_{-\alpha}^G F(X,H\uM)$ for the underlying Mackey functor. The notation $F(-,-)$ refers to the function spectrum. The first step towards computing such a cohomology theory is to compute it in the case $X=S^0$ the sphere spectrum. This turns out to be a difficult problem to describe in general, and has been carried out in very restricted cases when the groups possess a simple structure for the poset of subgroups \cite{Lew88, BG19, BG20, Zen18, KL20}. The interesting coefficients in this context are $\uA$ (the Burnside ring Mackey functor), $\uZ$ (the constant Mackey functor) or $\uZp$ (the constant Mackey functor associated to $\Z/p$).  In this paper, we are interested in computations of $\pi_\alpha^G H\uA$, and $\pi_\alpha^G H\uZ$, in the case $G=C_n$, a cyclic group of order $n$. Even though the poset of subgroups for cyclic groups are relatively simple, earlier attempts \cite{BG21b} warn us that computations may not necessarily yield easy formulas. Here, we follow an opportunistic approach, and identify sectors of $RO(G)$ where we can make complete calculations. 

The Mackey functors $\uA$ and $\uZ$ have a commutative multiplicative structure which makes the spectra $H\uA$ and $H\uZ$ commutative ring objects in the symmetric monoidal category of $G$-spectra. This implies a rich multiplicative structure on homotopy groups in terms of multiplicative norms. In particular, $\pi_\bigstar (H\uA)$ and $\pi_\bigstar(H\uZ)$ are graded commutative rings. 

\begin{mysubsection}{Computational techniques} 
Before we describe our results, we point out the computational techniques available to us. An advantage for cyclic groups is that the irreducible representations are described using powers of $\lambda$, the one dimensional complex representation associated to the inclusion $C_n \subset S^1$ as the $n^{th}$-roots of unity. A $G$-representation $V$ defines an Euler class $a_V \in \pi_{-V}(S^0)$ given by the inclusion $S^0 \to S^V$. One has the equation $a_{V\oplus W} = a_V a_W$, which implies that for the cyclic group $C_n$, these are monomials in $a_{\lambda^s}$ for $0\leq s <n$. The Euler class is part of a Gysin cofibre sequence
\[ S(\lambda^s)_+ \to S^0 \stackrel{a_{\lambda^s}}{\to} S^{\lambda^s}.\]
Smashing with a spectrum $X$ and computing $RO(G)$-graded homotopy groups gives a long exact sequence in which two terms are $\pi_\bigstar X$ and a shift of $\pi_\bigstar X$, while the third $\pi_\bigstar (S(\lambda^s)_+ \wedge X)$ can be computed from a proper subgroup of $C_n$. We call this the {\it cofibre sequence method}. 

The second technique we use is the {\it Tate diagram method}. For $1\leq s < n$, write 
\[E_s = S(\infty \lambda^s)= \varinjlim_n S(n\lambda^s),	~~~ E^s = S^{\infty \lambda^s} = \varinjlim_n S^{n\lambda^s}.\]
We have the following commutative diagram \cite{GM95} in which the rows are cofibre sequences. 
\begin{equation*}
			\begin{tikzcd}[scale cd=.92]
			{E_s}_+\wedge X \arrow[r,""]\arrow["\simeq"]{d} & X  \arrow[r,""]\arrow[""]{d} & E^s \wedge X\arrow[""]{d}
				\\
				{E_s}_+\wedge F({E_s}_+,X)\arrow[r,""] &  F({E_s}_+,X) \arrow[r,""] & E^s\wedge F({E_s}_+,X) 
			\end{tikzcd}
		\end{equation*}
The idea is that the bottom row is computable using subgroups of the cyclic group, and in this way, we know the top left corner. The map $X \to E^s \wedge X$ is basically obtained by inverting the Euler class $a_{\lambda^s}$ in case $X$ is a ring spectrum. Applying both of these together often leads to explicit calculations. 

The third technique is called the {\it Bredon cohomology method}. This is given by the formulas for a representation $V$
\[ \pi_{n-V} (H\uM) \cong \pi_n(S^V \wedge H\uM) \cong H_n^G(S^V;\uM), \]
\[\pi_{V-n}(H\uM) \cong \pi_{-n} (S^{-V}\wedge H\uM) \cong H^n_{G} (S^V;\uM).\]
The groups $H_n^G(S^V;\uM)$ and $H^n_G(S^V;\uM)$ can be computed by writing down an equivariant CW complex structure on the space $S^V$, and then via cellular homology and cohomology. We use these three techniques judiciously to uncover computations of $\pi_\bigstar^{C_n} H\uA$ and $\pi_\bigstar^{C_n} H\uZ$ for $\bigstar$ lying in various sectors of $RO(C_n)$. 
\end{mysubsection}

\begin{mysubsection}{Reduction to divisor gradings}
Let $\Pic(\Ho(\SP^G))$ denote the group of isomorphism classes of invertible $G$-spectra \cite{FLM01}. It is proved therein that there is an exact sequence
\[ 0 \to \Pic(A(G)) \to \Pic(\Ho(\SP^G)) \stackrel{d}{\to} C(G),\]
where $C(G)$ denotes the additive group of functions from the conjugacy classes of subgroups to $\Z$. For $G=C_n$, the Picard group $\Pic(\Ho(\SP^{C_n}))$ has been computed in \cite{Ang21}, where he shows that the map $RO(C_n) \to \Pic(\Ho(\SP^{C_n}))$ given by $\alpha \mapsto S^\alpha$ is surjective, and constructs the following commutative diagram.
\[\xymatrix{ 0\ar[r] & RO_0(C_n) \ar[r] \ar@{->>}[d]  & RO(C_n) \ar[r]^{d} \ar@{->>}[d] & C(C_n) \ar@{=}[d] \\
0 \ar[r] & \Pic(A(C_n)) \ar[r] & \Pic(\Ho(\SP^{C_n})) \ar[r]^-{d} & C(C_n) .}
\] 
The map $RO(C_n) \to C(C_n)$ sends $\alpha \in RO(C_n)$ to the function which maps a subgroup $H$ to the fixed point dimension $\dim(\alpha^H)$. We denote this by $|\alpha^H|$. The kernel of this map is denoted by $RO_0(C_n)$ which comprises $\alpha\in RO(C_n)$ such that $|\alpha^H|=0$ for all subgroups $H$. This is additively generated by elements of the form $\lambda^s - \lambda^{ks}$ for $\gcd(k,n)=1$ and $1\leq s < n$. The group \cite{Ang21}
\[ \Pic(A(C_n)) \cong \prod_{d \mid n} (\Z/d)^\times /\{\pm 1\}.\]  
We rewrite this in a form convenient for equivariant cohomology computations. Let $\MM_G$ be the category of $G$-Mackey functors. This is a symmetric monoidal category with $\Box$ as the monoidal structure and $\uA$ as the unit object. We identify $\Pic(A(C_n))$ with $\Pic(\MM_G)$ using a slightly different isomorphism (Theorem \ref{isocls}) than \cite{Ang21}. The elements of $\Pic(\MM_G)$ are denoted $\uA[\tau]$ and enumerated using functions 
\[ \tau \in \{\mbox{Divisors of } n\} \to \Z,~~ d \mid n \mapsto \tau_d, \mbox{ with } \gcd(\tau_d, \frac{n}{d})=1,\]
and we have an explicit function out of $RO_0(C_n)$ which sends $\alpha \mapsto \tau(\alpha)$ explicitly describing the left vertical arrow of the above diagram. In these terms, Theorem \ref{mfallze} reproves \cite[Theorem 5.1]{Ang21} and shows 
\begin{myeq}\label{Azero}
\upi_\alpha(H\uA) \cong \uA[\tau(\alpha)].
\end{myeq}

We now have a sequence of symmetric monoidal functors 
\[ \SP^{C_n} \stackrel{-\wedge H\uA}{\longrightarrow} H\uA-\Mod \stackrel{-\wedge_{H\uA} H\uZ}{\longrightarrow} H\uZ-\Mod.\] 
These preserve the invertible objects. The above computation shows that the image of $\Pic(A(C_n))$ is easily computable, with the image in $H\uA-\Mod$ given by an actual Mackey functor $\uA[\tau]$. The composition onto $H\uZ-\Mod$ is even easier, they all map to the unit object $H\uZ$ \cite[Theorem B]{Ang21}. The group $RO(C_n)/RO_0(C_n)$ is spanned by representations of the form $\lambda^d$ for $d\mid n$ together with the sign representation $\sigma$ if $n$ is even. This allows us to define {\it divisor gradings} $\bigstar_{\pm \div}$.  
\begin{defn}
The sector $\bigstar_{\pm \div} \subset RO(C_n)$ is defined to be  linear combinations of the form $ \ell +\epsilon \sigma+\sum_{d_i\mid n}~~b_i\mspace{2mu} \lambda^{d_i} $  with $ b_i, \ell, \in \Z, $ and $ \epsilon\in \{0,1\} $.
\end{defn}
We then have $RO(C_n)$ is a direct sum of $RO_0(C_n)$ and the subgroup corresponding to $\bigstar_{\pm \div}$. Any $\alpha\in RO(C_n)$ can be uniquely represented as $\alpha_0 + \alpha_\div$ in this decomposition, and we have the formula 
\[ \upi_\alpha(H\uA) \cong \upi_{\alpha_\div}H\uA \Box \uA[\tau(\alpha_0)], \mbox{ and } \upi_\alpha(H\uZ) \cong \upi_{\alpha_\div}(H\uZ).\] 
For $\uZ$, this means that there are units in $\pi_\alpha (H\uZ)$ for $\alpha \in RO_0(C_n)$ which multiply coherently, and we have an isomorphism of rings 
\[ \pi_\bigstar(H\uZ) \cong \pi_{\bigstar_{\pm \div}}(H\uZ) \otimes \pi_{\bigstar_0}H\uZ,\]
where $\bigstar_0$ is the sector corresponding to $RO_0(C_n)$. The second factor is easily described multiplicatively by writing down a basis of $RO_0(C_n)$ and adjoining $ \mu_b^{\pm 1}$ for each basis element $b$. 

For $\uA$, $\pi_{\bigstar_0}(H\uA)$ is describable using the multiplicative isomorphism given in \eqref{Azero}. However, we do not venture into this exercise, as writing this in terms of generators and relations leads to unnecessary notational complexity. Instead, we write down generators $\mu(\alpha)$ for each $\alpha \in \bigstar_0$, and the explicit formula for them may be used for calculating products. Now, we point out that if $\upi_{\bigstar_{\pm \div}}(H\uA)$  is known one may use \eqref{Azero} to go to $\upi_\bigstar(H\uA)$. We demonstrate this explicitly in two computations. (see \S \ref{ringstrp}) \\
(1) We compute $\pi_{\bigstar_{\pm \div}}H\uA$ for $G=C_p$ using the Tate diagram method. This gives an alternative way of understanding the calculations in \cite{Lew88}, and gives a much simpler formula. \\
(2) We show how to perform the calculations for $G=C_{pq}$, which completes the additive decomposition that was started in \cite{BG19}. 

\end{mysubsection}

\begin{mysubsection}{Constructing cohomology classes} 
The Euler classes $a_V\in \pi_{-V}(S^0)$ induce elements in the homotopy groups of $H\uA$ and $H\uZ$ via the unit map. For an orientable representation $V$, the orientation class in $H_{\dim(V)}(S^V)$ is $G$-equivariant, which means that it gives a class $u_V \in H_{\dim(V)}^G(S^V;\uZ)$. Thus, $u_V \in \pi_{\dim(V)-V}(H\uZ)$. We show that for a one dimensional complex representation $\xi$ (which is automatically orientable) $\pi_{2-\xi}(H\uA)\cong A(\ker(\xi))$ \eqref{eq:defining u-class} where we identify $\xi$ with the associated homomorphism $G \to S^1$. This is then used to lift the orientation class $u_\xi\in \pi_{2-\xi}(H\uZ)$ to $\pi_{2-\xi}(H\uA)$, which is then denoted by $u_\xi^A$ (Definition \ref{def:u-class}).  

The Euler classes bear the relation \eqref{eq:au general} (we call these {\it a-relations})
\[		I_{\tr^G_{\ker(\xi)}} a_{\xi} =0, \]
where $I_{\tr^G_{\ker(\xi)}}$ denotes the image of the transfer from $A(\ker(\xi))\to A(G)$. The relations for the classes $u_\xi$ are more involved. 	Firstly, the group in which $u_\xi$ lives is $A(\ker(\xi))$ which is not usually a cyclic $A(G)$-module, so it will have a set of generators which we denote by $u_\xi^{(\ker(\xi))}$ \eqref{genH}. We then write down the relations for these classes (which we call {\it u-relations}) using the kernel of the restriction map from $A(G)$ to $A(\ker(\xi))$ (Proposition \ref{prop:u-rel}). The Euler classes and the orientation classes satisfy the {\it gold relations} (or {\it au-relations}) described in Theorem \ref{prop:au-rel}. 
  
\end{mysubsection}

\begin{mysubsection}{Geometric fixed points} Let $\PP$ be the family of proper subgroups of $G$. The geometric fixed points of a $G$-spectrum $X$ is defined as 
\[\Phi^G(X) = (\widetilde{E\PP} \wedge X)^G.\]
The $G$-space $E\PP$ is characterized by the property 
\[ E\PP^H = \begin{cases} \ast &\mbox{if } H \in \PP \\ \varnothing & \mbox{if } H \not\in \PP, \end{cases}
\]
and $\widetilde{E\PP}$ is defined using the cofibre sequence 
\[ E\PP_+ \to S^0 \to \widetilde{E\PP}.\]
For the cyclic group $C_n$, we may write $n={p_1^{n_1}\dots p_k^{n_k}} $, and $s_i = \frac{n}{p_i}$. Then, 
\[ \widetilde{E\PP} \simeq E^{s_1}\wedge \cdots \wedge E^{s_k}.\]
This points us to two consequences. For a commutative ring spectrum $X$, $\pi_\ast(\Phi^{C_n}(X))$ is obtained from $\pi_\bigstar(X)$ by inverting the Euler classes and restricting to integer gradings. Further, iterating the Tate diagram method is useful in computing the homotopy groups of the geometric fixed point spectrum. These have been computed for $\Phi^{C_{p^n}}H\uZ$ (follows from \cite{HHR16}), and for  $\Phi^{(C_p)^n}H\uZ$ \cite{HK17,Kri20}. If $n$ has more than one prime factor (that is, $k\geq 2$), $\pi_\ast\Phi^{C_n}H\uZ = 0$. This is however, not true for $\uA$, and we prove (see Theorem \ref{thm:geo})
\begin{thma}
Let  $n={p_1^{n_1}\dots p_k^{n_k}} $.  For $ i=1 $ to $  k $, let $ H_i $ be  the subgroup of $ G $ of order $ |H_i|=n/p_i $. Then we have
		\begin{equation*}
			\pi_*(\Phi^{C_n}(H\uA))\cong \Z\oplus \bigoplus^k_{\ell=1} \bigoplus_{t>0}\frac{A(H_\ell)}{\langle   I_{H_\ell}, p_\ell \rangle}\{(u_{\lambda^{n/p_\ell}}/a_{\lambda^{n/p_\ell}})^t\}
		\end{equation*}
		where $p_\ell=[G:H_\ell]$ and  the ideal $ I_{H_\ell} $  consists  the transfers
		\begin{equation*}
			\big \langle  \tr_{H_i\cap H_\ell}^{H_\ell} A(H_i\cap H_\ell) \mid 1\le i\le k, i\ne \ell\big\rangle. 
		\end{equation*}
\end{thma} 
\end{mysubsection}

\begin{mysubsection}{Computations over the positive cone} We write {\it the positive cone} for the sector of $RO(G)$ consisting of elements of the form $n-V$ where  $V$ is an actual representation of $G$. This is the sector where computations can be carried out using the Bredon cohomology method. The positive cone is also additively closed, so the part of $\pi_\bigstar^G(H\uA)$ or $\pi_\bigstar^G (H\uZ)$ restricted to the positive cone is a graded commutative ring. 
Observe that both the Euler classes and the orientation classes lie inside the positive cone. Our computations in this sector are inspired from the following theme for a cyclic group.
\[\mbox{\it The positive cone is generated by the Euler classes and the orientation classes.}\]

We write $\bigstar_\div$ for $\bigstar_{\pm \div} \cap \big(\mbox{\it positive cone}\big)$, which is a linear combination $ \ell -\epsilon \sigma-\sum_{d_i\mid n}~~b_i\mspace{2mu} \lambda^{d_i} $  with $ b_i \geq 0$ and $\epsilon \geq 0$. It suffices to verify the theme above by restricting the attention to $\bigstar_\div$. For the constant coefficients $\uZ$, the a-relations, and au-relations are given by 
\[\frac{n}{d}a_{\lambda^d} =0 \quad \textup{~and~} \quad 
		\frac{d}{\gcd(d,s)}a_{\lambda^s} u_{\lambda^d}=\frac{s}{\gcd(d,s)} a_{\lambda^d}u_{\lambda^s}.\]
There are no u-relations. If $n$ is even, $\sigma$ is a non-orientable representation which has the Euler class $a_\sigma$ satisfying $2 a_\sigma=0$, but no orientation class. For $\uZ$-coefficients we prove (see Theorem \ref{thm ring z ev} and Theorem \ref{hzthm})
\begin{thmb}
For a cyclic group $C_n$, $\pi_{\bigstar_\div}^{C_n}(H\uZ)$ is the polynomial ring on the Euler classes $a_{\lambda^d}$ (and $a_\sigma$ if $n$ is even), the orientation classes $u_{\lambda^d}$ modulo the ideal generated by the a-relations and the au-relations. 
\end{thmb}
Theorem B is used in \cite{BDK23} in the computation of equivariant cohomology of complex projective spaces for cyclic groups. This result confirms for us that there are no further relations among Euler classes and orientation classes in this case. We also point out that this is special for cyclic groups, as even for elementary Abelian groups with $\F_p$ coefficients we have additional relations (Remark \ref{elemabrel}). We also compute the positive cone for $\uA$-coefficients but have to assume that either $n=$ a prime power or $n$ is square free (see Theorem \ref{thm A cpm}, Theorem \ref{thm A cpm div} and Theorem \ref{Thm: Ring structure main Thm})
\begin{thmc}
Suppose that $n$ is either  a prime power or square free. The ring $\pi_{\bigstar_{\div}}^{C_n}H\uA$ is generated by Euler classes and orientation classes, modulo the ideal generated by the a-relations, u-relations, and the au-relations.
\end{thmc}
The proof of Theorem C goes differently for the prime power case, and the square free case. In the latter case, we use the trick employed in \cite{BG20} which allows us to move from $\uZ$-coefficients to $\uA$-coefficients by a string of short exact sequences of Mackey functors.
\end{mysubsection}

\begin{mysubsection}{Computations in non-negative grading}
The final part of the computation is carried out for non-negative gradings. This is the part of $RO(G)$ consisting of $\alpha$ which are $\leq 0$ which means that $\dim(\alpha^H)\leq 0$ for all subgroups $H$. In this case, one may use the fact that the maps $S^0 \to H\uA \to H\uZ$ are all $0$-equivalences to note that 
\[ \upi_\alpha(S^0) \cong \upi_\alpha H\uA \mbox{ for } \alpha \leq 0,\]
\[ \upi_\alpha(S^0) \to \upi_\alpha H\uZ \mbox{ is surjective for } \alpha \leq 0.\]
We also have that 
the fixed point map 
	$$
	\mathcal{F}^G: \pi_\alpha^G(S^0)\to C(G), \quad  \mathcal{F}^G(\nu)(H)= \deg (\nu^H).
	$$
	is injective for $\alpha\in \bigstar_{\le 0}$ \cite[Proposition 1]{tDP78}. This gives us enough ingredients to proceed towards computations for  $\bigstar_{\pm \textup{div},\le 0}=\bigstar_{ \pm \textup{div}}\cap \bigstar_{\le 0}$. However, we soon discover that this sector becomes notational very complicated for $\uA$ coefficients, if we try to express the answer in terms of generators and relations. The reason is that the Burnside ring for subgroups, and invertible modules over them enter into the notation when we try to write down multiplicative formulas. So we restrict our attention to $\uZ$ coefficients in this sector. 

The Euler classes belong to non-negative gradings, and for $\bigstar_{\div,\leq 0}= \bigstar_{ \textup{div}}\cap \bigstar_{\le 0}$, from the previous computations, we deduce that
\[\pi_{\bigstar_{\div,\leq 0}}^{C_n}(H\uZ) = \begin{cases} \Z[a_{\lambda^d}, d\mid n]/(\frac{n}{d}a_{\lambda^d}) &\mbox{if } n \mbox{ is odd} \\ \Z[a_{\lambda^d}, d\mid n, a_\sigma]/(\frac{n}{d}a_{\lambda^d}, 2 a_\sigma) & \mbox{if } n \mbox{ is even.} \end{cases}\]
In the square free case, we have a full calculation of $\pi_{\bigstar_{\pm \div,\leq 0}}^{C_n}(H\uZ)$ in \cite{BG20},  the group is calculated for $\alpha\in \bigstar_{\leq 0}$ as
\[ \pi_\alpha (H\uZ) = \begin{cases} \Z & \mbox{if } |\alpha|=0\\ \Z/(m(\alpha)) &\mbox{if } |\alpha|<0, \end{cases}\]
for a function $m(\alpha)$. The $|\alpha|<0$ part of the expression above lies in the image of $a_\lambda$-multiplication. The torsion-free part is expressed as a subring of $\Z[u_{\lambda^d}^{\pm 1} \mid  d\mid n]$ which in degree $\alpha$ is $\frac{n}{m(\alpha)}$ times the generator. This demonstrates the difficulty in expressing this sector in terms of generators and relations. Nevertheless, restricting our attention to $n=$ prime power, we have a description of the ring structure in Theorem \ref{neodd} and Theorem \ref{negen}. We have classes 
\[ a_{\lambda_i/\lambda_{i-1}} : S^{\lambda_{i-1}}\to S^{\lambda_i}, z\mapsto z^p,\]
which via products and  transfers produce classes $\chi_{\alpha,s}\in \pi_\alpha^{C_{p^m}}H\uZ$ for $0\leq s \leq m$ and $|\alpha|=0$. These classes satisfy a bunch of relations (see \S \ref{secneg} for more details) denoted by the collection $\rho$. The ring structure is described using all these generators and relations.
\begin{thmd}
The ring $\pi_{\bigstar_{\leq 0}}^{C_{p^m}}$ for $p$ odd is the ring generated by $a_{\lambda_0}$, $\chi_{\alpha,s}$ for $|\alpha|=0$, $0\leq s\leq m$, modulo the relations $\rho$ defined between them. For $p=2$, there is an additional generator $a_\sigma$, with a similar collection of relations.  
\end{thmd}

\end{mysubsection}

	\begin{notation} We use the following notations
		\begin{itemize}
			\item Throughout this paper,  a general finite group is denoted by $ \GG$. We write  $ G=C_n $ for the cyclic group of order  $ n $, and $ g\in G $ denotes a fixed generator of $ G $. We use $ \G $ when $ n=p^k $ for some prime $ p $.
			\item  We often omit the superscript in $ \pi^G_V(H\uZ) $ and simply write $ \pi_V(H\uZ) $ when the group $ G $ is understood.
			\item  The irreducible complex $ 1 $-dimensional  representations are $1,\lambda, \lambda^2,\cdots, \lambda^{n-1}  $, where $ \lambda  $ acts via $ n ^{th}$ root of unity.  We use the same notation for its realization. Note as a real representation $\lambda^i$ and $\lambda^{n-i}$ are isomorphic. For $ d\mid n $, the kernel of the representation $ \lambda^d $ is $ C_d \le G$. If $n$ is even, the sign representation $\sigma$ satisfies $2\sigma = \lambda^{n/2}$. 
			\item  A divisor  $ d $ of $ n $ is said to be  proper  if $ d\ne n $.  
			\item 	The linear combination $ \ell -\epsilon \sigma-\sum_{d_i\mid n}~~b_i\mspace{2mu} \lambda^{d_i} $  with $ b_i, \ell, \in \Z, $ and $ \epsilon\in \{0,1\} $ is denoted by $ \bigstar_{\pm\textup{div}} $.  When  all the $ b_i $'s are  positive, the resultant is denoted by $ \bigstar_{\textup{div}} $.
			The gradings of  $ \bigstar_{\textup{div}} $,  which does not contain  $ \sigma $, i.e., $ \epsilon=0 $ is called $ \bigstar^e $. We also use the same notation $ \bigstar_\div  $ for cohomological grading, that is when graded cohomologically, $ \bigstar_\div  $ consists of linear combinations of the form $ \sum_{d_i\mid n}~~b_i\mspace{2mu} \lambda^{d_i}+\epsilon \sigma-\ell $  with $ b_i{\ge 0}, \ell \in \Z, $ and $ \epsilon\in \{0,1\} $. 
			We use $ \bigstar_{\le 0} $ when all the fixed point dimensions are $ \le 0 $. Let $ \bigstar^e_{\le 0}$ consist of those elements of  ${\bigstar_{\le 0}}$ which does not involve $\sigma$, i.e, of the form $\sum_{d_i\mid n}b_i\lambda^{d_i}+r $.
			Further, $\bigstar_{\pm\textup{div},\le 0}  $ means the intersection of $ \bigstar_{\pm \textup{div}}$ and $ \bigstar_{\le 0}$.  Similarly, $\bigstar_{\textup{div},\le 0}=\bigstar_{ \textup{div}}\cap \bigstar_{\le 0}$.  
		\end{itemize}
	\end{notation}

\begin{mysubsection}{Organization}
In \S\ref{symmono} and \S \ref{invertiblemod}, we write down a formula for the box product of Mackey functors for cyclic groups, and invertible modules over them. This leads to a computation over the gradings in $RO_0(C_n)$. In \S \ref{lifting}, we lift the orientation classes to $\uA$ coefficients and discuss computations for the geometric fixed points. Some computations in $C_p$ and $C_{pq}$ cases with $\uA$-coefficeints are carried out in \S \ref{ringstrp}. \S \ref{secrepA} and \S \ref{secpositiveHZ} are devoted to computing the positive cone for the Burnside ring and the constant coefficients respectively. In \S \ref{secneg}, we discuss the computations in non-negative gradings.  
\end{mysubsection}

\begin{ack}
The first author would like to thank Surojit Ghosh for some helpful conversations in various phases of this paper. The research of the second author was supported by the NBHM grant no. 16(21)/2020/11. 
\end{ack}	

	\section{The symmetric monoidal category of  $ G $-Mackey functors}\label{symmono}
	The category of Mackey functors \cite{Dre73} has a symmetric monoidal structure denoted by the box product $\Box$ \cite{Lew81}. For $(-1)$-connected $G$-spectra $E$ and $F$, we have the formula  from \cite[1.3]{LM06}
\begin{myeq}\label{smasbox}
	\underline{\pi}_0^G(E\smas F)\cong \underline{\pi}_0^G(E) \Box \underline{\pi}_0^G( F).
\end{myeq}
We recall the definition of the box product and write down an explicit formula in the case of cyclic groups. If the order is square free, there is a further simplification using a tensor product construction.\par
For finite $G$-sets $X, Y$, let $C(X,Y)$ be the isomorphism classes of diagrams of the form 
 $$\xymatrix@R-=.25cm@C-=.25cm{ &  &L \ar[dl] \ar[dr]  \\ &X  &  &Y}$$ 
 where  $X\leftarrow L\rightarrow Y $ and $X
\leftarrow L' \rightarrow Y$ are said to be isomorphic if the following commutes
 $$\xymatrix@R-=.25cm@C-=.25cm{&  &L\ar[dl] \ar[dd]^{\cong} \ar[dr]  \\ &X  &  &Y \\& &L' \ar[ul] \ar[ur]}$$
Under disjoint union $C(X,Y)$ becomes a commutative monoid.
The \textit{Burnside category} $B_G$ is defined as the category in which objects are finite $G$-sets and the morphisms between two $G$-sets,  $X$ and $Y$ is the group completion of $C(X,Y)$.
 A $G$-Mackey functor is a contravariant additive functor from the Burnside category $\mathcal{B}_G$ to the category of Abelian groups, $\mathcal{A}b$. Following \cite{Dre73} we write an explicit definition of Mackey functors for $G=C_n$.
 \begin{defn}
 	A $G$-Mackey functor $\uM$ assigns  a commutative $G/H$-group $\uM(G/H)$ for each subgroup $H\le G$, and constitutes maps called  transfer $\tr^H_K : \uM(G/K) \to  \uM(G/H)$ and restriction $\res^H_K : \uM(G/H) \to  \uM(G/K)$ for $K\le H\le G$ so that
 	\begin{enumerate}
 		\item  $\tr^H_L= \tr^H_K \tr^K_L $ and $\res^H_L= \res^K_L\res^H_K$ for all $L\le K \le H$.
 		\item  $\tr^H_K(\gamma\cdot x) = \tr^H_K(x)$ for all $x \in  \uM(G/K)$ and  $\gamma\in  H/K$.
 		\item $\gamma\cdot  \res^H_K(x) = \res^H_K(x)$ for all $x \in \uM(G/H)$ and  $\gamma\in H/K$.
 		\item $\res^H_K \tr^H_L(x)=\sum_{\gamma\in H/LK}\gamma\cdot \tr^K_{L\cap K} \res^L_{L\cap K}(x)$ for all subgroups $L,K\le H$.
 	\end{enumerate}
 \end{defn}
\begin{exam}
	The \textit{Burnside ring Mackey functor} $ \uA $ is defined by $ \uA(G/H)=A(H) $, the Burnside ring of $ H $ which is the group completion of the free Abelian monoid of isomorphism classes of finite $ H $-sets under disjoint union. Thus the classes $ [H/K] $,  $ K\le H $ form a basis of $ A(H) $. The action of $ G/H $ on $ A(H) $ is trivial.  The transfer maps, $ \tr_K^H $ are defined by inducing up a $ K $-set $ S $ to an $ H $-set : $S\mapsto H\times_K S$ for $ K\le H $. This is the representable functor defined as $B_G(-,G/G)$. We write $ \mu_n=[G/G] $ and $ \mu_d=[C_d/C_d] $ for $ d|n $. All  restriction  maps are determined from $\res^G_{C_d}(\mu_n)=\mu_d  $.
	Let $n=p_1^{n_1}\cdots p_k^{n_k}$. Then 
	\begin{equation*}\label{rankah}
		\textup{~rank}(A(G))=\Z^\ell \quad \textup{~where~} \ell=(n_1+1)\cdots(n_k+1).
	\end{equation*}
\end{exam}
\begin{exam}
	The constant Mackey functor $  \underline{\mathbb{Z}} $ is described by 
	\[
	\uZ(G/H)=\Z, \quad \res^H_K=\textup{Id}, \quad \tr^H_K=[H:K].
	\]
The Mackey functor $\uZ^\ast$ is described by 
	\[
	\uZ^\ast(G/H)=\Z, \quad \tr^H_K=\textup{Id}, \quad \res^H_K=[H:K].
	\]
\end{exam}
\begin{exam}
The Mackey functor $\langle  \mathbb{Z} \rangle$ is defined by 
\[
\langle  \mathbb{Z} \rangle(G/H)=\Z \textup{~if~} H=G, \textup{~and~} 0 \textup{~if~} H\ne G.
\]
\end{exam}
The box product  $\uM\Box \uN$ of  $G$-Mackey functors  is defined as the left Kan extension 
\[
\begin{tikzcd}
	\mathcal{B}^\textup{op}_G\times \mathcal{B}^\textup{op}_G \arrow[rr, "\uM\otimes\uN"] \arrow[d, "\mu"'] &  & \mathcal{A}b \\
	\mathcal{B}^\textup{op}_G \arrow[rru, "\uM\Box\uN"']                   &  &   
\end{tikzcd}
\]
where $\mu$ denote the cartesian product of two $G$-sets and $ (\uM\otimes\uN)(X,Y)= \uM(X)\otimes \uN(Y)$. 
The universal property for the box product is described as:   a map $\uM\Box \uN\to \underline{L}$ is equivalent to specifying maps $\uM(X)\otimes\uN(Y)\to \underline{L}(X\times Y)$ for all  $X,Y\in \mathcal{B}_G^{\textup{op}}$, satisfying suitable conditions \cite{Lew81} under the restriction and transfer maps.. 

  	 An explicit formulation of the box product can be found in Lewis \cite{Lew88} for $G=C_p$ and in \cite{HM19} for $G=C_{p^n}$.
	\begin{defn}\label{defmf}
		Let $\uM, \uN\in \MM_G$. The box product  $\uM\Box \uN$ is  inductively defined as follows. Let $C_d\le G$.  For $p\mid d$, let   $\pi_p$ is the quotient map 
		$$
		\pi_p: (\uM\Box \uN)(G/C_{d/p}) \to (\uM\Box \uN)(G/C_{d/p})/ (W_{C_d}(C_{d/p})).
		$$
		Define  $\tau (M,N)(G/C_d)$ to be  the co-equalizer of 
		\[
		\begin{tikzcd}
			\bigoplus_{p,q\mid d} (\uM\Box \uN)(G/C_{d/pq}) \ar[rr,shift left=.75ex,"\pi_p\,\circ\, \tr^{C_{d/p}}_{C_{d/pq}}"]
			\ar[rr,shift right=.75ex,swap,"\pi_q\,\circ\, \tr^{C_{d/q}}_{C_{d/pq}}"]
			&&
			\bigoplus_{p\mid d}\{(\uM\Box \uN)(G/C_{d/p})/ (W_{C_d}(C_{d/p}))\}.
		\end{tikzcd}
		\]
		Then
		\begin{equation*}\label{eqbox}
			(\mbn)(G/C_d):=\Big[M(G/C_d)\otimes N(G/C_d)\oplus\tau(M,N)(G/C_{d}) \Big]/FR.
		\end{equation*}		
		The Frobenius relations are given by: 
		\begin{align*}
			x\otimes \tr_{C_{d/p}}^{C_d}(y)&=\res_{C_{d/p}}^{C_d}(x)\otimes y\quad \textup{~where~} x\in M(G/C_d), y\in N(G/C_{d/p})\\
			\tr_{C_{d/p}}^{C_d}(w)\otimes z&=w\otimes \res_{C_{d/p}}^{C_d}(z)\quad\textup{~where~} w\in M(G/C_{d/p}), z\in N(G/C_d).	
		\end{align*}
		i) 	For $p\mid d$, 
		$$
		(\uM\Box \uN)(G/C_{d/p})\to \tau(M,N)(G/C_d)\to (\uM\Box \uN)(G/C_d)
		$$
		is the transfer map.\\
		ii) 	The restriction map  is given by 
		$$
		\res^{C_d}_{C_\ell}(m\otimes n)=\res_{C_\ell}^{C_d}(m)\otm \res_{C_\ell}^{C_d}(n) \quad\textup{~ if~} m\otm n\in M(G/C_d)\otimes N(G/C_d),
		$$
		and  for $x\in (\uM\Box \uN)(G/{C_{d/p}})/(W_{C_d}(C_{d/p}))$ pick $\tilde{x}\in \underline{M}\Box\underline{N} (G/C_{d/p}) $ such that $\pi_p(\tilde{x})=x $, and  define 
		$$
		 \res^{C_d}_{C_\ell}(x)=\begin{cases}
		 	\res^{C_{d/p}}_{C_\ell}(\sum_{\gamma\in C_d/C_{d/p}} \gamma \cdot \tilde{x})\quad \textup{~if~} \ell\mid d/p\\
		 	\tr^{C_\ell}_{C_{d/p}\cap C_\ell} \res^{C_{d/p}}_{C_{d/p}\cap C_\ell}\tilde{x} \quad \textup{~if ~} \ell\nmid d/p.
		 \end{cases}
		$$
		iii) The Weyl group action on $ (\mbn)(G/C_d) $ is given by 
		$$
		\gamma\cdot(a\otm b)=\gamma \cdot a\otm \gamma \cdot b, {\;\rm where\;}  \gamma\in W_G(C_d)
		$$ 
		for $a\otimes b\in M(G/C_d)\otimes N(G/C_d)$, and on the other factors it is the induced action on the quotient group.  
	\end{defn} 
	This recovers the definition of box product for the group $G=C_{p^n}$ \cite{HM19}. One easily verifies that the formula in Definition \ref{defmf} satisfies the universal property. This is because each subgroup $C_d$ of $G=C_n$ has maximal proper subgroups given by $C_{d/p}$ for $p\mid d$, and hence it suffices to consider only the Frobenius relations arising from these subgroups.  \par
	For $G=C_n$ where $n$ is square free, we have a $\boxtimes$-construction for Mackey functors \cite{BG20}. For a divisor $d$ of $n$, note that any transitive $G$-set $G/H$ may be written as $C_d/H_1\times C_{n/d}/H_2$  where $H_1=C_d\cap H$ and $H_2=C_{n/d}\cap H$.
	\begin{defn}
	Let $d\mid n$, and  $\uM\in \MM_{C_d}$ and $\uN\in \MM_{C_n/d}$. Define $\uM\boxtimes \uN\in \MM_G$ as
	\[
	\uM\boxtimes \uN(G/H)=\uM(C_d/H_1)\otm \uN(C_{n/d}/H_2).
	\]
	For $K\le H$, $\res^H_K=\res^{H_1}_{K_1}\otimes\res^{H_2}_{K_2}$ and $\tr^H_K=\tr^{H_1}_{K_1}\otimes\tr^{H_2}_{K_2}$.
	The construction $\uM\boxtimes \uN$ may also be viewed as the left Kan extension of the tensor product along the functor $m$, 
	\begin{align*}
					m: \mcal{B}^\textup{op}_{C_{d}}\times \mcal{B}^\textup{op}_{C_{n/d}} & \longrightarrow \mcal{B}^\textup{op}_{G} \\
					(X,Y) & \longmapsto X\times Y.
				\end{align*}
	\end{defn}

\begin{exam}
One easily computes that $\uZ_d\boxtimes \uZ_{n/d}\cong \Z$ and $\uA_d\boxtimes \uA_{n/d}\cong \uA$, where $\uA_d$ (respectively $\uZ_d$) denotes the Burnside ring Mackey functor (respectively, the constant Mackey functor) for the group $C_d$ ($d\mid n$ and $n$ is square free). Writing a square free  $n=p_1\cdots p_k$, the subgroups of $C_n$ are indexed by subsets $I$ of $ \uk=\{1,\cdots, k\}  $. We define Mackey functors $\uA_I\boxtimes \uZ_{I^c}$, writing $ I^c = \uk\setminus I $, $|I|=\prod_{i\in I}p_i$, and $\uA_I=\uA_{C_{|I|}}$.  These Mackey functors fit into a string of short exact sequences for $j\in I$
\[ 0 \to \bZ_{p_j} \boxtimes \uA_{I\setminus\{j\}} \boxtimes \uZ_{I^c} \to \uA_I \boxtimes \uZ_{I^c} \to \uA_{I\setminus \{j\}} \boxtimes \uZ_{I^c \cup \{ j\}} \to 0.\]
These sequences are used in \cite{BG20} to go from computations for $\uZ$-coefficients to that for $\uA$-coefficients in the case of cyclic groups of square free order.
\end{exam}
	The following  gives a simple formula for the box product in square free cyclic case.
	\begin{prop}
		$(\underline{M}_1\boxtimes \underline{N}_1)\Box (\underline{M}_2\boxtimes \underline{N}_2)\cong (\underline{M}_1\Box \underline{M}_2)\boxtimes (\underline{N}_1\Box \underline{N}_2) .$
	\end{prop}
	\begin{proof}
		Since the box product can be described as the left Kan extension along the cartesian product functor $\mu$, we have
		\begin{myeq}\label{box-hom-Kan}
			\Hom_{[\mcal{B}^\textup{op}_{G}\to \mcal{A}b]} \big((\underline{M}_1\boxtimes \underline{N}_1)\Box (\underline{M}_2\boxtimes \underline{N}_2),\,\underline{L}\big)\cong \Hom_{[\mcal{B}^\textup{op}_{G}\times \mcal{B}^\textup{op}_{G}\to \mcal{A}b]} \big((\underline{M}_1\boxtimes \underline{N}_1){\otimes} (\underline{M}_2\boxtimes \underline{N}_2), \,\underline{L}\circ \mu\big) .
		\end{myeq}
		Also, $\uM\boxtimes \uN$ is the left Kan extension along the above functor $ m $, hence we have 
		\begin{equation*}
			\Hom_{[\mcal{B}^\textup{op}_{G}\to \mcal{A}b]} \big((\underline{M}_1\Box \underline{M}_2)\boxtimes (\underline{N}_1\Box \underline{N}_2),\,\underline{L}\big) \cong \Hom_{[\mcal{B}^\textup{op}_{C_{d}} \times  \mcal{B}^\textup{op}_{C_{n/d}}\to \mcal{A}b]} \big((\underline{M}_1\Box \underline{M}_2){\otimes} (\underline{N}_1\Box \underline{N}_2), \,\underline{L}\circ m\big)
		\end{equation*} 
		\[
		\cong \Hom_{\big[ [\mcal{B}^\textup{op}_{C_{d}}\times \mcal{B}^\textup{op}_{C_{d}} ] \times  [\mcal{B}^\textup{op}_{C_{n/d}} \times \mcal{B}^\textup{op}_{C_{n/d}}]\to \mcal{A}b\big]} \big((\underline{M}_1{\otimes} \underline{M}_2){\otimes} (\underline{N}_1{\otimes} \underline{N}_2), \,\underline{L}\circ m \circ (\mu_1\otimes \mu_2)\big).
		\]
		Further using the functor $ m $, we can rewrite   \eqref{box-hom-Kan} as 
		\[ 
			\Hom_{\big[ [\mcal{B}^\textup{op}_{C_{d}}\times  \mcal{B}^\textup{op}_{C_{n/d}} ] \times [\mcal{B}^\textup{op}_{C_{d}} \times  \mcal{B}^\textup{op}_{C_{n/d}}]\to \mcal{A}b\big]} \big((\underline{M}_1{\otimes} \underline{N}_1){\otimes} (\underline{M}_2{\otimes} \underline{N}_2), \,\underline{L}\circ \mu \circ (m{\times } m)\big).
		\]
	\end{proof}
	
	\section{The group of invertible $ \uA $-modules}\label{invertiblemod}
	Recall that $RO_0(G)$ is the kernel of the dimension map 
	\begin{align*}
		RO(G)&\xrightarrow{d} C(G),\\
		\alpha&\mapsto \dim(\alpha^H)=|\alpha^H|.
	\end{align*}
	  The Mackey functors $ \upi_\alpha^G(H\uA) $ for $ \alpha\in RO_0(G) $ happen to be invertible $ \uA $-modules \cite{tDP78}.  Let $ G=C_n $. We say  $\uM\in \MM_G  $ is invertible if there exists $ \uN\in \MM_G $ such that $ \uM\Box \uN\cong \uA $. It turns out that  $ \uM $ is invertible implies it is isomorphic to $ \uA[\tau] $ for certain function $ \tau $, see Definition \ref{tburn}. We rewrite the computation of Angeltveit \cite{Ang21} which determines the group of invertible $\uA$-modules to match the notation of the following sections. In particular, we define explicit generators $\mu_n(\alpha)$ for the invertible modules $\upi_\alpha (H\uA)$.\par
	\begin{defn}\label{tburn}
		Let $\tau $ be a  function from the set of subgroups of $ G $ to $ \Z $. 
		Thus  for each $ d\mid n $ we have $ \tau_d \in \Z$, and   $ \tau_n $ is declared to be $ 1 $.  The Mackey functor $ \uA[\tau] $   is defined in a similar way as of $ \uA $ except the restriction maps now depend on $ \tau $:
		\begin{equation*}
			\res^G_{C_d}\,\mu_n= \big(\prod_{d\mid s} \tau_{s}\big)\,\mu_d.
		\end{equation*} 
		Hence  for $ C_k\le C_d $, the restriction map is
		\begin{equation*}
			\res^{C_d}_{C_{k}} \mu_d=\big(\prod_{{k\mid s,\, d\nmid s}} \tau_{s}\big)\, \mu_{k}.
		\end{equation*}
		Therefore  $\uA[\tau ](G/C_d)\cong A(C_d)$ has a basis $\{\mu_d, \tr^{C_d}_{C_k}\mu_k \textup{~for~} k\mid d\}$. The action of $G/C_d$ on $A(C_d)$ is trivial. 
	\end{defn}
	\begin{exam}
		Let $ n=pq $. The Mackey functor structure $ \underline{A}[\tau]$ can be described in terms of $ \tau_p,\tau_q $ and $ \tau_e \in \mathbb{Z}$ as 
		\small
		\[
		\xymatrix{ & & \Z^4\ar@/_1.1pc/[dll]_{{ \left[\begin{smallmatrix}	\tau_p & q & 0 & 0 \\	0 & 0 & \tau_p\tau_e  & q	\end{smallmatrix}\right]}}   \ar@/^1.1pc/[drr]^{{ \left[\begin{smallmatrix} \tau_q & 0 & p & 0 \\ 0  & \tau_q\tau_e & 0 & p \end{smallmatrix}\right]}}  & & \\ 
			\Z^2 \ar@/_.7pc/[urr]|{{\tiny \left[\begin{smallmatrix}	 0 & 0 \\1 & 0 \\ 0 & 0\\ 0 & 1  \end{smallmatrix}\right]}} \ar@/_1.1pc/[drr]_{{ \left[\begin{smallmatrix}   \tau_q\tau_e & p \end{smallmatrix}\right]}} & & & & \Z^2 \ar@/^.7pc/[ull]|{{\tiny \left[\begin{smallmatrix} 0 & 0 \\0 & 0 \\ 1 & 0\\ 0 & 1  \end{smallmatrix}\right]}} \ar@/^1.1pc/[dll]^{{ \left[\begin{smallmatrix} \tau_p\tau_e & q \end{smallmatrix}\right]}} \\ 
			&  &  \Z \ar@/_.7pc/[ull]|{{\tiny \left[\begin{smallmatrix} 0 \\ 1 \end{smallmatrix}\right]}} \ar@/^.7pc/[urr]|{{\tiny \left[\begin{smallmatrix} 0 \\ 1 \end{smallmatrix}\right]}}&&}
		\]
		\normalsize
	\end{exam}
	Our formulation of $\uA[\tau]$ differs slightly from the description given in \cite{Ang21}. Here we have that $\res^G_{C_d}(\mu_n)=\big(\prod_{s\mid (n/d)} \tau_{s}\big)\,\mu_d$, and $\tau _e=1$. The following result classifies all invertible $\uA$-modules.  
	\begin{theorem}[\cite{Ang21}]\label{isocls} 	Let  $ \tau $ and $ \tau' $ be two functions as in Definition \ref{tburn}.
		\begin{enumerate} 
			\item
			Then 
			\[
			\uA[\tau ]\Box \uA[\tau ']\cong \uA[\tau \tau '], \qquad (\tau \tau')_d=\tau _d\cdot \tau'_d.
			\]
			\item  We have 
			$$
			\uA[\tau]\cong \uA[\tau'] \quad\textup{~if and only if ~} \quad  \tau_d\equiv  \pm\tau'_d \pmod {\frac{n}{d}}$$
			for all $ d\mid n $ such that $ \tau_d$ and $\tau'_d\not\equiv0\pmod {\frac{n}{d}} $.
			\item Let $\uM\in \mathcal{M}_G$ be invertible. Then $\uM\cong \uA[\tau ]$ where for each $d\mid n$, $\gcd(\tau _d, n/d)=1$.
		\end{enumerate}
	\end{theorem}
	As a result $ \uA[\tau ]  $ is invertible if and only if for each $d\mid n$, $\gcd(\tau _d,n/d)=1$.
	
	\begin{mysubsection}{Equivariant homotopy in  \(RO_0(G) \) grading}\label{genBurnmf}
		For the group  $C_p $,  Lewis \cite[Theorem 2.3]{Lew88} observed that 
		$ \upi_\alpha^{C_p}(H\uA) $ can be completely determined in terms of 
		$ |\alpha| $ and $ |\alpha^{C_p}|$ except in the  case  when $ \alpha\in RO_0(C_p) $, i.e., when both $ |\alpha|=|\alpha^{C_p}|=0$. In this  case,  $ \upi_\alpha^{C_p}(H\uA) \cong \uA[d]$, where $ d $ depends on $ \alpha $. However,  $ d $ can  only be determined $\mod p$ because of the fact that $ \uA[d]\cong \uA[d'] $ if $ d\equiv d' \pmod p $. To describe the relation between $\alpha  $ and $ d $, Lewis introduced the  function $d: RO_0(C_p) \to \Z $ defined as follows: suppose  $ \alpha \in RO_0(C_p)  $ is of the form $ \sum_{i}\lambda^{k_i}-\lambda^{l_i} $ where	 $ 0< l_i, k_i<p $. 
		Then $ d(\alpha)$, written as $d_\alpha$ is  
		$$
		d_\alpha= \prod_i l_ik_i^{-1}, \quad  \textup{~where~} k_i^{-1} \in  (\Z/p)^\times \textup{~be such that~}  k_ik_i^{-1}\equiv1\pmod p.
		$$
		Although the assignment $ \alpha \mapsto d_\alpha $ is not a homomorphism from $ RO_0(C_p) $ to $ \Z $, it does  become a homomorphism when viewed as a map from $RO_0(C_p)  $ to $ (\Z/p)^\times/\pm 1 $.\par
		We carry the computations further for a general cyclic group $G=C_n$, and   prove that given an $ \alpha\in RO_0({G}) $, there is a natural candidate  $\tau (\alpha)$ so that  $ \upi_\alpha^{G}(H\uA) $ is isomorphic to $ \uA[\tau(\alpha)]$ (see Definition \ref{tburn}).	For this let $ \alpha \in RO_0(G)$ 
		be a non-zero element. Observe that, up to a suitable rearrangement, such an $\alpha$ can be written as 
		\[
		\alpha=\sum_{i=1}^{m}\eta_i-\varphi_i
		\]
		such that  $ \varphi_i=\lambda^{d_ik_i} $ and $ \eta_i=\lambda^{d_il_i} $ for some divisor $ d_i$ of $ n $ with $d_i\ne n$, and  $ k_i$, $ l_i $ are positive integers  relatively prime to $ {n}/{d_i}$.  Note that in this expression, factors of $\sigma$ must cancel out. For each $ d\mid n $, we consider the set $\mathcal{S}(d):=\{i\mid d=d_i\} $ (possibly empty). In the following we define a $ G $-equivariant map $ \mu_n(\alpha) $
		\begin{equation*}
			\mu_{n}(\alpha):= \bigwedge_i \big(f_i: S^{{\eta_i}}\to S^{{\varphi_i}}\big) 
		\end{equation*}
		such that for each $d\mid n  $, the  non-equivariant  degree of 
		\begin{equation*}
			\deg \big[\bigwedge_{i\in \mathcal{S}(d)} \big(f_i :S^{\eta_i}\to S^{\varphi_i}\big)\big]=\tau_d(\alpha).
		\end{equation*}     
		For example let $ \alpha=\lambda^{dt}-\lambda^{ds} $ where  $ s$, $t$ are  relatively prime to $ {n}/{d} $. Let $ t^{-1} $ be the  integer such that   $ 1\le t^{-1}<\frac{n}{d} $
		and $ tt^{-1}\equiv 1 \pmod {\frac{n}{d}} $. Consider the map $ \C\to \C $, given by $ z\mapsto z^{s t^{-1}} $.
		Define $ \mu_n(\alpha) $ to be  the  induced  map on the  one-point compactifications. 	In this case $ \tau_d(\alpha) $ is $ s t^{-1} $.
		In the general case, we take the smash product of all such maps to obtain the map $ \mu_n(\alpha) $, 
		\begin{myeq}\label{formu}
			\mu_{n}(\alpha): S^{\sum{\eta_i}}\to S^{\sum{\varphi_i}},
		\end{myeq}
		and our required  $ \tau_d(\alpha) $ is 
		\begin{myeq}\label{eqtaual}
			\tau_d(\alpha)=\prod_{i\in \mathcal{S}(d)}k_il_i^{-1} \pmod {\frac{n}{d}}.
		\end{myeq}
		The Hurewicz image of 	$[ \mu_n(\alpha) ]\in \pi_{\alpha}^GS^0$ defines an element in 
		$ \pi_{\alpha}^G(H\uA) $,
		also denoted by 
		$ \mu_n(\alpha) $. Let $\res_{C_d}(\alpha)  $
		denote the restriction of $ \alpha $ to $ C_d $. Similarly, we may construct $\mu_d(\res_{C_d}(\alpha))$ (simply $\mu_d(\alpha)$)
		\[ 
		\mu_d(\alpha)\in {\pi}_{\alpha}^{C_d}(H\uA)
		\]
		associated to  
		$\res_{C_d}(\alpha)  $.
		Then  $ \mu_n(\alpha) $ restricts to the element $ \mu_d(\alpha) $ multiplied by
		\[
		\deg \big(\mu_{n}(\alpha)^{C_d}: S^{\sum{\eta_i}^{C_d}}\to S^{\sum{\varphi_i}^{C_d}}\big),
		\]
		the non-equivariant degree of the $ C_d $-fixed point map.
		Hence in 
		$ {\upi}_{\alpha}^{G}(H\uA) $
		\[ 
		\res^G_{C_d}\,\mu_n(\alpha)= \big[\prod_{d\mid s} \tau_{s}(\alpha)\big]\,\mu_d(\alpha).
		\]
	\end{mysubsection}
	Consider the following subgroup of $\uM(G/G)$ generated by the  elements
	$$
	\textup{~Tr~}:=\{\tr^G_{H}(\uM(G/H)) \textup{~for all~} H\lneq G\}\subseteq \uM(G/G).
	$$
	Define the $G$-Mackey functor $I(\uM)$ by
	\[
	I(\uM)(G/H)= 
	\begin{cases}
		\textup{~Tr~} & \textup{~if~} H=G\\
		\uM(G/H) & \textup{~otherwise}.
	\end{cases}
	\]
	\par
	Associated to an $\alpha\in RO_0(G)$ we have $\tau (\alpha)$ as defined in \eqref{eqtaual}. With these notations we have the following. 
	\begin{theorem}\label{mfallze}
		Let $ \alpha\in RO_0(G) $. Then $ \upi_{\alpha}^G(H\uA)\cong  \uA[\tau(\alpha)] $.
	\end{theorem}
	\begin{proof}
		First  consider the case  $\alpha=\eta-\eta'\in RO_0(G)$. From \cite[Theorem 2]{tDP78}  $\upi_{\alpha}^G(H\uA)$
		is an invertible Mackey functor, so isomorphic to  $\uA[\tau ]$ for some choice of $ \tau$ by part (3) of  Theorem \ref{isocls}. It suffices to observe that $\mu_n(\alpha)$ is a generator of  $\upi_\alpha^G(H\uA)(G/G)$. For this, proceed by induction assuming the result holds for all proper subgroups. At $G=e$, the result holds trivially. Consider the Mackey functor $\underline{I}(\tau ):=I(\uA[\tau ])$. Then we have a short exact sequence of Mackey functors 
		\[
		0\to \underline{I}(\tau )\to \uA[\tau]\xrightarrow{} \langle  \mathbb{Z} \rangle\to 0.
		\]
		Consider the following diagram
		\begin{myeq}\label{sesmf}
			\begin{tikzcd}
			& \upi^G_{\alpha}(H\underline{I}) \arrow[r] \arrow[d] & \upi^G_{\alpha}(H\uA) \arrow[r] \arrow[d, "\cong"] & \upi^G_{\alpha}(H\langle\mathbb{Z}   \rangle)  \arrow[d, "\cong"] &   \\
			0 \arrow[r] & \underline{I}(\tau ) \arrow[r]           & \uA[\tau ] \arrow[r]                    & \langle\mathbb{Z}   \rangle \arrow[r]          & 0.
		\end{tikzcd}
		\end{myeq}
		 Observe that, for any $G$-space $X$, $X\smas H\langle\mathbb{Z}   \rangle\simeq X^G\smas H\langle\mathbb{Z}   \rangle$ as both represent the same cohomology theory.
		  Therefore we get 
		\begin{equation*}
			\begin{tikzcd}
				S^{0}\smas H\langle\mathbb{Z}   \rangle \arrow[rrr, "{\mu_n(\alpha)}^G\smas H\langle\mathbb{Z}   \rangle"] \arrow[d, "\simeq"'] & & &S^{0} \smas H\langle\mathbb{Z}   \rangle\arrow[d, "\simeq"] \\
				S^{\eta}\smas H\langle\mathbb{Z}   \rangle\arrow[rrr, "\mu_n(\alpha)\smas H\langle\mathbb{Z}   \rangle"]                       &&&  S^{\eta'}\smas H\langle\mathbb{Z}   \rangle.                      
			\end{tikzcd}
		\end{equation*}
		Note that  $\mu_n(\alpha)^G=\textup{Id}\in \pi^G_0(S^0) $.  Therefore in the diagram \eqref{sesmf}, the class of  $\mu_n(\alpha)\in \upi_\alpha^G(H\uA))$ maps to $1$ in the bottom right corner $\langle\mathbb{Z}   \rangle$. Hence, $\mu_n(\alpha)$ serves as a generator of $\uA[\tau]$. Since the restrictions of $\mu_n(\alpha)$ are given by $\tau(\alpha)$, we get $\uA[\tau ]\cong \uA[\tau(\alpha)]$.\par
		For the general case, we use \eqref{smasbox}. Note that $S^{-\alpha}\smas H\uA $ is $(-1)$-connected, and $S^{-\alpha-\beta}\smas H\uA\smas H\uA\simeq S^{-\alpha}\smas H\uA\smas S^{-\beta}\smas H\uA$. So $\upi_{\alpha+\beta}(H\uA)\cong \upi_{\alpha}(H\uA)\Box \upi_{\beta}(H\uA) $.

	\end{proof}

	\section{Lifting orientation classes to $ \uA $-coefficients}\label{lifting}
	We now  construct  certain homotopy classes which provide the necessary generators of the ring $ \pi_{\bigstar_\div}^G(H\uA)  $ studied in later sections.  These classes were considered in the context of $ \uZ $-coefficients in \cite[\S 3]{HHR16}.
	\begin{defn}
		For a finite group $ \GG $, let $ V $ be a $ \GG $-representation. The inclusion $ \{0\}\subset V $ induces the map \mbox{$ S^0\to S^V $.} This gives the element $ a_V\in \pi_{-V}^\GG(S^0)\cong \pi_0^\GG(S^V) $.
		Composing it with the map $ S^0\to H\uA $ we get an element in   $  \pi_{-V}^\GG( H\uA) $, which is also denoted by $ a_V $.
	\end{defn} 
	Let $ \xi $ be a non-trivial $ 1 $-dimensional complex $ \GG $-representation coming from the group map $\xi: \GG\to S^1 $ with $ \ker(\xi)=H$. 
	We may endow  the following  $  \GG $-CW structure on $ S^{\xi} $ 
	\begin{myeq}\label{eq:cw str}
		S^0\cup ({\GG/H}_+\smas e^1) \cup ({\GG/H}_+\smas e^2).
	\end{myeq} 
	\noindent So the equivariant cellular chain complex \cite[\S 3.3]{HHR16},  
	$	A(H)\xrightarrow{0} A(H)\xrightarrow{\tr_{H}^\GG}A(\GG) $
	gives
	\begin{equation*}
		\pi_{-\xi}^\GG( H\uA)\cong A(\GG)/\tr_{H}^\GG A(H).
	\end{equation*}
	This implies that $ a_{\xi}  $ supports the relation 
	\begin{equation*}
		[\GG \times _{H} S]\cdot a_{\xi} =0 \quad \textup{~for all~} S\in A(H).
	\end{equation*}
	\begin{defn}
		Let $ I_{\tr^\GG_H} $ denote the ideal in $ A(\GG) $ generated by the elements $ \tr^\GG_H(A(H)) $.
	\end{defn}
	In terms of this notation, we have the relation
	\begin{myeq}\label{eq:au general}
		I_{\tr^\GG_{\ker(\xi)}} a_{\xi} =0.
	\end{myeq}
	For $V=\sum_i n_i\xi_i  $, we have $ a_V=\prod_ia_{\xi_i}^{n_i} $.\par
	Let $G=C_n$, $ n =p^m$. We define the classes $\alpha_{p^\ell}$ in $ A(G) $ by 
	\begin{equation*}
		\alpha_{p^{\ell}}=\big[ G/C_{n/p^{\ell}}\big]= \tr^G_{C_{n/p^{\ell}}} \res^G_{C_{n/p^{\ell}}}(1), 	
	\end{equation*}
	where  $ 1 $   refers to the element $ [G/G] $. In these terms, we may write the Burnside ring $ A(G) $  as
	\begin{myeq}\label{genpm}
		A(G)\cong \mathbb{Z}[ \alpha_{p},\cdots,\alpha_{p^{m}} ]/(\alpha_{p^{i}}\alpha_{p^j}-{p^{i}}\alpha_{p^{j}}\mid \textup{~if~} i\le j ).
	\end{myeq}
	In  case of $ G=C_n $,  $n={p_1^{s_1}\dots p_k^{s_k}} $, 
	\begin{equation*}
		A(G)\cong A(C_{p_1^{s_1}})\otimes\cdots \otimes A(C_{p_k^{s_k}}).
	\end{equation*}
	Therefore, if $ G=C_n $, $ n={p_1}\cdots {p_k} $, then 
	\begin{myeq}\label{gensq}
		A(G)\cong \mathbb{Z}[ \alpha_{p_1},\cdots,\alpha_{p_{k}} ]/(\alpha_{p_{i}}^2-{p_{i}}\alpha_{p_{i}}).
	\end{myeq}
	For example, $ A(C_p) \cong \Z[\alpha_p]/(\alpha_p^2-p\alpha_p)$,  $ \alpha_p=[C_p/e]=\tr^{C_p}_e \res^{C_p}_e(1)$.
	If $ G/H =C_{p_{j_1}^{r_1}\cdots p_{j_t}^{r_t}}$ where $ r_i> 0$ for all $ i $, then the relation $\eqref{eq:au general}  $ simplified as the collection of relations 
	\begin{myeq}\label{eq: a-relation in A-coeff}
		\prod_{i=1}^t\alpha_{p_{j_i}^{m_i}}\cdot a_{\xi} =0, \qquad \textup{ } m_i\geq r_i \textup{ for all } i.
	\end{myeq}\par
	If $ V $ is an orientable $ \GG $-representation, the orientation  class $ [V]\in H_{\dim V}(S^V;\Z) $ is preserved by the $ \GG $-action. This allows us to define $ u_V^\Z\in\pi^\GG_{\dim  V-V}(H\uZ)\cong  H^\GG_{\dim V}(S^V;\uZ) \cong \Z$ to be the class that restricts to $ [V] $ \cite{HHR16}. These classes naturally satisfy the product formula $ u_Vu_W=u_{V\oplus W} $. In what follows, we explain a construction  of these classes for $ \uA $-coefficients.\par
	Let $ \xi: \GG\to S^1 $ be a complex 1-dimensional $ G $-representation.  One has the cofibre sequences  
\[ {\GG/H}_+\xrightarrow{1-g} {\GG/H}_+\to S(\xi)_+ , \mbox{ and }  S(\xi)_+\to S^0\to S^{\xi} .\]
  Since  $ \pi_0^\GG({\GG/H}_+\smas H\uA)\cong A(H) $,  the first cofibre sequence  implies  
\[ \pi^\GG_{1}(S(\xi)_+\smas H\uA)\cong \ker(1-g: A(H)\to A(H))\cong A(H),\] 
	as $ g $ acts trivially on $ A(H) $. This combined with the second cofibre sequence gives
	\begin{myeq}\label{eq:defining u-class}
		\pi_{2-\xi}^\GG(H\uA) \xrightarrow{\cong} \pi_1^\GG(S(\xi)_+\smas H\uA)\xrightarrow{\cong} \pi_0^\GG({\GG/H}_+\smas H\uA)\xrightarrow{\cong} \pi_0^{H}(H\uA)= A(H).
	\end{myeq}
	We use this  isomorphism to define the following class. 
	\begin{defn}\label{def:u-class}
		Define the class $ u_{\xi} \in  \pi^\GG_{2-\xi}(H\uA)$ to be the element which corresponds to $1\in A(H) $ under the  isomorphism \eqref{eq:defining u-class}. 
	\end{defn}
	The class $ u_V^\Z $ is the  image of $ u_\xi^A $ in the case $ V=\xi $.	Observe that, the $ A(\GG) $-module structure on $ \pi^\GG_{2-\xi}(H\uA)\cong A(H) $ is given by the restriction map $ \res: A(\GG)\to A(H) $. Define $ I_{\res^\GG_H} $ to be the  ideal which is the kernel of the restriction map $ A(\GG)\to A(H) $.
	If the restriction map is surjective (e.g. groups of  square free order) then  
	\begin{equation*}
		\pi^\GG_{2-\xi}(H\uA)\cong A(\GG)\{u_\xi\}/I_{\res^\GG_{H}} u_{\xi}.
	\end{equation*}
	In general, this  may not be a cyclic  module as the restriction map need  not  be surjective (e.g.  groups of prime power order). 
	This leads us to define the following set of classes that together generate $A(H)$ as an $A(G)$-module.
	\begin{defn}\label{defxh}
		For a subgroup  $ H $ of $G$, and an isomorphism $\phi:M\to A(H)$ of $A(G)$-modules denote by
		\begin{meq*}\label{genH}
			x^{(H)}:=\{ [H/K]x \mid K\leq H\}\subseteq M,
		\end{meq*}
		which serve as additive generators satisfying $\phi([H/K]x)=[H/K]$.
	\end{defn}
	As an $ A(\GG) $-module, $ u_\xi^{(H)} $ generates the group $ \pi^\GG_{2-\xi}(H\uA)\cong  A(H) $.  The classes $ u_\xi^{(H)} $ satisfy two kinds of relations. One type  is $ I_{\res^\GG_H}\cdot u_\xi^{(H)}=0 $. The relations  present inside $ A(H) $ give us the second kind of relation. For groups of prime power order and square free order,  the ideal $ I_{\res^\GG_H} $ takes the following form.
	\begin{prop}\label{prop:u-rel}
		For groups of prime power order and square free order, the ideal $ I_{\res^\GG_H} $ is generated by the elements $ (\alpha_{d}-d) $, where $ d $ is a divisor of $ |G/H| $. Hence,	for every divisor $ d $ of $ |G/H| $,  $ u_{\xi} $ and  $ u_{\xi}^{(H)} $ support following relations
		\begin{myeq}\label{eq: u-relation for A-coeff}
			(\alpha_{d}-d)u_{\xi}=0 \textup{ and  } (\alpha_{d}-d)u_{\xi}^{(H)}=0.
		\end{myeq}
	\end{prop}
	Next we observe certain quadratic relations involving the classes $a_\xi$ and $u_\xi$ for two different one-dimensional complex representations $\xi$. For $ i=1,2 $, let $ \xi_i $ be a non-trivial $ 1 $-dimensional complex $ \GG $-representation coming from $\xi_i: \GG\to S^1 $ with $ \ker(\xi_i)=H_i $. 	The product of the CW structures \eqref{eq:cw str}  gives a $  \GG $-CW structure on $ S^{\xi_1+\xi_2}= S^{\xi_1}\smas S^{\xi_2} $. So the equivariant cellular chain complex is
	{\fontsize{3.34mm}{2mm}\selectfont
		\begin{myeq}\label{xi1xi2}
			\begin{split}
				A(\GG/H_1\times \GG/H_2)	\xrightarrow[]{\partial_4}
				A(\GG/H_1\times \GG/H_2)\oplus A(\GG/H_1\times \GG/H_2) \xrightarrow[]{\partial_3}A(H_1)\oplus A(H_2)\oplus A(\GG/H_1\times \GG/H_2)\\
				\xrightarrow[]{\partial_2} 	A(H_1)\oplus A(H_2)\xrightarrow[]{\partial_1}A(G)
			\end{split}
		\end{myeq}
	}
	\normalsize
	If  $ \partial $ is the  boundary map of  individual factor, then the boundary map $ \partial_r $ is determined via
	\begin{myeq}\label{eq:boundary au rel}
		\partial_r(e^i\smas e^j)= \partial(e^i)\smas e^j+(-1)^ie^i\smas \partial(e^j).
	\end{myeq} 
Note that the following two $ \GG$-sets are isomorphic
	\begin{myeq}\label{eq:isomorphic G set}
		\begin{split}
			\varphi : \GG/H_1\times \GG/H_2&\to \GG/H_1\cap H_2\times |\GG/H_1H_2|\\
			(xH_1,yH_2)&\mapsto (xH_1\cap H_2, z),
		\end{split}
	\end{myeq}
	where $ z=x^{-1} y$.
	Let $ d={|\GG/H_1H_2|}$. The boundary map $\partial_4$ is given by
	$$
	\partial_4: \bigoplus_{d} A(H_1\cap H_2)\to \bigoplus_{d} A(H_1\cap H_2)\oplus \bigoplus_{d} A(H_1\cap H_2)
	$$
	 $$ (x_1,\cdots, x_d)\mapsto (x_1-x_d,\cdots, x_d-x_{d-1})\oplus (x_1-x_2,\cdots, x_d-x_{1}).$$ 
Therefore the kernel is isomorphic to the diagonal copy of  $  A(H_1\cap H_2)$. The isomorphism $\pi^\GG_4(S^{\xi_1+\xi_2}\smas H\uA)\cong A(H_1\cap H_2)$ defines a class $u_{\xi_1+\xi_2}$ as in Definition  \ref{defxh} such that $u_{\xi_1+\xi_2}^{(H_1\cap H_2)}$ generates $\pi^\GG_{4-\xi_1-\xi_2}(H\uA)$.
	The class $ u_{\xi_1}\cdot u_{\xi_2} $ also belongs to $ \pi_4^\GG(S^{\xi_1+\xi_2}\smas H\uA) $ and equals $u_{\xi_1+\xi_2}$. This now implies 
	$$
	([H_1/K] u_{\xi_1})\cdot ([H_2/L] u_{\xi_2})=\res^{H_1}_{H_1\cap H_2}([H_1/K])\cdot \res^{H_2}_{H_1\cap H_2}([H_2/L])u_{\xi_1+\xi_2}.
	$$
	Let $ G=C_{p^m} $. Suppose $ H_1\subseteq H_2 $. The products $ u_{\xi_1}^{(H_1)}\cdot u_{\xi_2}^{(H_2)} $ in fact generate $ A(H_1\cap H_2)=A(H_1) $ (for example, when $ y=u_{\xi_2} $).  So in this case we may write $u_{\xi_1}^{(H_1)}\cdot u_{\xi_2}^{(H_2)} = u_{\xi_1+\xi_2}^{(H_1\cap H_2)}$.\par
	We also have  the following relations between  $ u_{\xi_1}a_{\xi_2} $ and $ a_{\xi_1}u_{\xi_2} $.
	\begin{thm}\label{prop:au-rel}
		The classes $ u_{\xi_1}a_{\xi_2} $ and $ a_{\xi_1}u_{\xi_2} $ in $ \pi_2^\GG(S^{\xi_1+\xi_2}\smas H\uA) $ satisfy the   relation
		\begin{myeq}\label{eq:au rel general}
			\{\Big[\frac{H_1}{H_1\cap H_2}\Big] u_{\xi_1}\}a_{\xi_2}=\{\Big[\frac{H_2}{H_1\cap H_2}\Big] u_{\xi_2}\}a_{\xi_1}.
		\end{myeq}
	\end{thm}
	\begin{proof}
		The map $ a_{\xi_2}: S^{\xi_1}\to S^{\xi_1+\xi_2}  $ gives rise to the map
		\begin{align*}
			\pi_2(S^{\xi_1}\smas H\uA)\cong A(H_1)&\xrightarrow{a_{\xi_2}} 	\pi_2(S^{\xi_1+\xi_2}\smas H\uA)
		\end{align*}
		that sends the generator $ u_{\xi_1} $ of $ A(H_1) $ to  $u_{\xi_1}\cdot a_{\xi_2}  $. Similarly we get $u_{\xi_2}\cdot a_{\xi_1}  $.\par
		Next we compute $ \partial_3 $ in the complex \eqref{xi1xi2}.
Write $\partial_3= \partial_3^1+ \partial_3^2$ where $\partial_3^i$ is the map from the $i^{th}$-factor $A(\GG/H_1\times \GG/H_2)$. We have 
		\[
		\partial_3^1(x_1,\cdots, x_d)=\tr^{H_2}_{H_1\cap H_2}(x_1)+\cdots +\tr^{H_2}_{H_1\cap H_2}(x_d)- \{(x_1-x_2,\cdots, x_d-x_{1})\}.
		\]	
and
		\[
		\partial_3^2(y_1,\cdots, y_d)=\tr^{H_1}_{H_1\cap H_2}(y_1)+\cdots +\tr^{H_1}_{H_1\cap H_2}(y_d)+ \{(y_1-y_d,\cdots, y_d-y_{d-1})\}.
		\]	
		As a result 
		\begin{equation*}
			\partial_3((0,-1,\cdots, 0),(1,0,0,\cdots,0))=\{\tr^{H_1}_{H_1\cap H_2}(1)\cdot u_{\xi_1}\}a_{\xi_2}-\{\tr^{H_2}_{H_1\cap H_2}(1)\cdot u_{\xi_2}\}a_{\xi_1}.
		\end{equation*}
		Hence we have the relation \eqref{eq:au rel general}.
	\end{proof}
	\begin{rmk}\label{elemabrel}
		Elementary Abelian groups possess  other kinds of relations also. For example,  for $ i=1,2,3 $,  let $ \lambda_i  $ be a non-trivial complex 1-dimensional  representation of  $ \HH=C_p\times C_p $.	Then in $ \pi_{4}^{\HH }(S^{\lambda_1+\lambda_2+\lambda_3}\smas H\underline{\F_p}) $
		we have a relation 
		\begin{myeq}\label{eq:auu rel a coeff}
			a_{\lambda_1} u_{\lambda_2} u_{\lambda_3} -u_{\lambda_1} a_{\lambda_2}  u_{\lambda_3} + u_{\lambda_1} u_{\lambda_2} a_{\lambda_3} =0.
		\end{myeq}
		The proof is  similar to \eqref{eq:au rel general}. The map $ a_{\lambda_3}  $ results the map $ 	\pi_4(S^{\lambda_1+\lambda_2}\smas H\underline{\F_p})\cong \F_p\xrightarrow{a_{\lambda_3}} 	\pi_4(S^{\lambda_1+\lambda_2+\lambda_3}\smas H\underline{\F_p}) $
		which sends $ {u_{\lambda_1}u_{\lambda_2}} \mapsto { u_{\lambda_1}u_{\lambda_2} a_{\lambda_3}} $. Similarly we get the classes $ a_{\lambda_1} u_{\lambda_2} u_{\lambda_3}$ and $ u_{\lambda_1} a_{\lambda_2} u_{\lambda_3} $. Suppose $ \ker(\lambda_i)=K_i $.
		The equivariant chain complex  in $ \deg 5 $ and $ \deg 4 $ looks like 
		{ \small
			\begin{align*}
				(({\F_p})^p)^3        \cong			(\underline{\F_p}(\HH/K_{1}\times \HH/K_{2} \times  \HH/K_{3}))^3	\xrightarrow[]{\partial_5}
				&(\underline{\F_p}(\HH/K_{1}\times \HH/K_{2} \times  \HH/K_{3}))^3	\oplus \\
				&	\bigoplus_{(i j)\in \{(12),(13),(23)\}}\underline{\F_p}(\HH/K_{i}\times\HH/K_{j}) \cong ((\F_p)^p)^3  \oplus (\F_p)^3
		\end{align*}}
		Iterating the identification  \eqref{eq:isomorphic G set} we get an isomorphism  of  the following two $ \HH $-set
		\begin{equation*}
			\HH/K_1\times\HH/K_2 \times \HH/K_3
			\cong \HH/e\times |\HH/K_1|.
		\end{equation*}
		Applying  this isomorphism we obtain the following description of $ \partial_5 $
		\begin{equation*}
			\begin{split}
				\partial_5\big((x_1,\cdots, x_p), (y_1,\cdots,y_p), (z_1,\cdots, z_p)\big)= (x_1-x_2,\cdots, x_p-x_1)
				+	(y_1-y_p,\cdots, y_p-y_{p-1})\\	+(z_1-z_p,\cdots, z_p-z_{p-1})+{(z_1-z_2,\cdots, z_p-z_{1})}+\Sigma_i x_i +\Sigma_i y_i+\Sigma_i z_i
			\end{split}
		\end{equation*}
		Hence we have the relation  \eqref{eq:auu rel a coeff} which arises as image of $ \partial_5 $. This   phenomenon is also present in case of Burnside ring coefficients as well as integer coefficients.
	\end{rmk}
	\normalsize
	\begin{mysubsection}{Geometric fixed points over cyclic groups}\label{geofixpt}
		We begin with the definition of the geometric fixed points spectrum.
	\end{mysubsection}
	\begin{defn}
		Let  $ \PP $  be the family of all  subgroups $ H\lneq \GG $. The classifying space for the family $ \PP $, $ E\PP $   is characterised up to homotopy equivalence by the property  that $(E\PP)^H\simeq *$ for all $ H\in \PP $ and empty when $ H=\GG $. The space $ \widetilde{E\PP} $ is  characterized by the property that $ (\widetilde{E\PP}) ^H\simeq *$  for $ H\in \PP $ and $ S^0 $ for $ H=\GG $. This fits into the cofibre sequence  $ E\PP_+\to S^0 \to \widetilde{E\PP}$. The geometric fixed point spectrum is defined as 
		$$
		\Phi^\GG(H\uA) :=(\widetilde{E\PP}\smas H\uA)^\GG.
		$$
	\end{defn} 
	For the group $ G =C_n $  
	with $n={p_1^{n_1}\dots p_k^{n_k}} $, we have the following simple description of $ \widetilde{E\PP}\smas H\uA $.	For $ i=1 $ to $  k $, let $ H_i $ be  the subgroup of $ G $ of order $ |H_i|=n/p_i $.	Let  $ \lambda^{|H_i|} $ be the  $ G $-representation whose stabilizer subgroup is $ H_i $.
	\begin{align*}
		\widetilde{E\PP}\smas H\uA &\simeq S^{\infty\lambda^{|H_1|}}\smas\cdots\smas S^{\infty\lambda^{|H_k|}}\smas H\uA\\
		&\simeq \bigwedge_{i=1}^k {\textup{~cofibre~}}(S(\infty\lambda^{|H_i|})_+\to S^0)\smas H\uA,
	\end{align*}
	where  $ S(\infty\lambda^{|H_i|}) =\lim_{n\to \infty} S(n\lambda^{|H_i|})$.	Here note that 	$ S(\infty\lambda^{|H_i|}) $ serves as a model for the classifying space $ E\FF(H_i) $, where $ \FF(H_i) $ is  the family containing all subgroups of $ H_i $.\par
	Here we explore the  homotopy groups  of the geometric fixed point spectrum with Burnside ring coefficients  $ \uA $.  
	The ring  $ \pi_*(\Phi^{C_{2^n}}(H\uZ)) $ was studied in \cite[Prop. 3.18]{HHR16} and  $\pi_* (\Phi^{(C_2)^n}(H\F_2)) $ was determined in \cite[Thm. 2]{HK17}.\par
	To determine the  homotopy groups  of the geometric fixed point spectrum, we use the following  spectral sequence that converges to $\pi_*(\Phi^G(H\uA))$ and which arises from the following  filtration 
	\[ 
	F_0=S^0\smas H\uA ,
	\]
	\begin{equation*}
		F_r = \textup{~colim}_{\{i_1,\dots,i_j\}\in T_r, j\le r} \bigwedge_{\ell=1}^j {\textup{~cofibre~}}(S(\infty\lambda^{|H_{i_\ell}|})_+\to S^0)\smas H\uA
	\end{equation*}
	\begin{myeq}\label{eq:filt-r}
		F_r/F_{r-1} = \bigvee_{\{i_1,\dots,i_r\}\in T_r} \Sigma S(\infty\lambda^{|H_{i_1}|})_+\smas  \cdots\smas\Sigma S(\infty{\lambda}^{|H_{i_r}|})_+\smas H\uA,
	\end{myeq} 
	where $ T_r $ is the set that contains all cardinality $ r $ subsets of $ \{1,2,\dots, k\} $. Thus $ E^1_{0,*} \cong \pi_*(H\uA) $, so $  E^1_{0,0} =A(G)  $ and $ E^1_{0,t}=0 {\textup{~for~}} t>0 $. 
	For an element $ s\in T_r$ of the form $ \{i_1,\dots,i_r\}$, write $ \cap H_s= H_{i_1} \cap \cdots \cap H_{i_r}$.
	We  have the equivalence
	\begin{equation*}\label{eq: univ-clssifying-id}
		S(\infty\lambda^{|H_i|})\times S(\infty\lambda^{|H_j|}) \simeq S(\infty\lambda^{|H_i\cap H_j|}).
	\end{equation*}
	Applying this we have the following simplifications 
	\begin{myeq}\label{eq:geo ss}
		E^1_{r,t} =\pi_{r+t}^G(F_r/F_{r-1})=\bigoplus_{s\in T_r}\pi_t^G(S(\infty \lambda^{|\cap H_s|} )_+\smas H\uA).
	\end{myeq}
	In each degree $ m $, $ S(\infty \lambda^{|\cap H_s|} )_+ $ has a cell of  the form $ {G/\cap H_s}_+\smas D^m $. Hence the equivariant cellular chain complex  $ C_*^{{\textup{cell}} }(S(\infty \lambda^{|\cap H_s|} )_+ , \uA)$ takes the form  
	\begin{equation*}
		\cdots\xrightarrow{} A(\cap H_s) 	\xrightarrow{0} A(\cap H_s) \xrightarrow{[G:\cap H_s]}A(\cap H_s) \xrightarrow{0}A(\cap H_s).
	\end{equation*}
	Thus 
	\begin{equation*}
		\pi_*^G(S(\infty \lambda^{|\cap H_s|} )_+\smas H\uA)=\begin{cases}
			A(\cap H_s) & \text{if $ *=0 $}\\
			A(\cap H_s)/[G: \cap H_s] &\text{if $ *=\textup{odd}$}\\
			0 & \text{otherwise}.
		\end{cases}
	\end{equation*}
	 The $ d_1 $-differentials arise by collapsing one of the factors $ S(\infty{\lambda}^{|H_i|})_+ $ to $ S^0 $.  So on the $ E^1 $ page, the $ d_1 $-differential is given by 
	\begin{equation*}
		\bigoplus_{s\in T_r}\sum_{j=1}^{r} (-1)^{j}\tr_{\cap H_s}^{\cap H_{s\setminus \{j\}}}A(\cap H_s).
	\end{equation*}
	\textit{Step 1}: To determine  the $ E^2_{r,t} $ terms, let us  first  concentrate on the $ t=0 $ level where the chain complex looks like
	\begin{myeq}\label{eq:geo complex}
		A(H_1\cap\cdots\cap H_k)	\xrightarrow{} 	\cdots	\xrightarrow{}\bigoplus_{i,j\in\{1,\cdots,k\}} A(H_i\cap H_j) \xrightarrow{}\bigoplus_{i=1}^kA(H_i) \xrightarrow{}A(G).
	\end{myeq}
	For an element $ [\cap H_s] /K\in A(\cap H_s)$, the transfer is of the form $ [\cap H_{s\setminus\{j\}}/K] $.
	Let $ \HH $ be the collection of subgroups 
	\begin{myeq}\label{defkge}
		\HH:=\{K\leq G\mid K=H_{j_1}\cap \dots\cap H_{j_\ell} {\textup{~where~}} \{j_1,\dots, j_\ell\}\in \{1,\dots, k\}\}.
	\end{myeq}
	Let $ L $ be a subgroup of $ G $. Then $ L $ is contained in some $ K\in \HH $;  let order of  $ \mathcal{K}$ be minimum of such. Suppose $ \mathcal{K} $ can be written as an intersection of  $ \ell $ number of subgroups. The complex \eqref{eq:geo complex} is a direct sum of   sub-complexes (corresponding to each such subgroup L) of length $ \ell+1 $ of the form 
	\small
	\begin{equation*}
		C_*(\mathcal{K}/L): \quad \Z\{[\mathcal{K}/L]\}\rightarrow \dots \rightarrow \bigoplus_{i,k\in \{1,\dots, \ell \}} \Z\{[H_{j_i}\cap H_{j_k}/L]\}\rightarrow\bigoplus_{i\in \{1,\dots, \ell \}} \Z\{[H_{j_i}/L]\}\rightarrow  \Z\{[G/L]\}.
	\end{equation*}
	\normalsize
	Here  $ \Z\{[H/L]\} $ means a copy of $ \Z $  corresponding to the coset $ [H/L] $.	Thus we have a chain complex 
	\begin{equation*}
		\Z\rightarrow\dots \rightarrow \bigoplus_{\ell_{C_2}}\Z\rightarrow{}  \bigoplus_{\ell_{C_1}}\Z\rightarrow{}  \Z.
	\end{equation*}
	This sub-complex is equivalent to   the simplicial  chain complex of a simplex $ \triangle^{\ell-1} $ with vertices $ \{v_0,\dots, v_{\ell-1}\} $ augmented to $ \Z $. So the homology of this complex is $ 0 $ for all $ n\ge 0 $.  Let $ C_*(G/G) $ be the complex  which at $ n=0 $ is isomorphic to $ \Z$, generated by the element $ [G/G] \in A(G)$, and $ 0 $ for $ n>0 $. Then \eqref{eq:geo complex} is isomorphic to  the direct sum of the sub-complexes $C_*(G/G) \oplus \bigoplus_{\substack{L\in G, \mathcal{K}\in \HH}}C_*(\mathcal{K}/L) $. As a result $ E^2_{0,0}\cong \Z$ and $ E^2_{*,0}=0 $ for $ *>0 $.\par
	\textit{Step 2}: To determine $ E^2_{*,t} $, for $ t>0 $ we follow a similar approach. But note that these groups are torsion groups and for each prime $p$, the torsion subgroup is a $\Z/p$-module.  We restrict to the $ p $-torsion part for each prime $ p$ that divides the order of $ G$.  Consider a subgroup $L\lneq G$, and let $p=p_i$. Suppose the set of prime factors of  $[G:L]$ are 
	$$
	\{p_{j_1},\cdots,p_{j_\ell}\}\cup \{p_i\}.
	$$
	Then define  $S_L=\{p_{j_1},\cdots,p_{j_\ell}\}$. Then $\mathcal{K}$ must be of the form $H_{j_1}\cap \cdots\cap H_{j_\ell}\cap H_i$ (see \eqref{defkge}).
	\small 
	\normalsize
	Two possibilities  may arise:	$\mathit{(i)}$ $S_L\ne \varnothing$. This implies we have a sub-complex of the form 
	\small 
	\begin{equation*}
		\overline{C}_*(\mathcal{K}/L): \quad  \Z/p\{[\mathcal{K}/L]\}\rightarrow\dots \rightarrow \bigoplus_{r=1}^\ell \Z/p\{[H_{j_r}\cap H_{i}/L]\}\rightarrow  \Z/p\{[H_{i}/L]\}.
	\end{equation*}
	This sub-complex is equivalent to the chain complex of $ \triangle^{\ell-2} $ augmented to $\Z/p$.  $\mathit{(ii)}$ $S_L=\varnothing$. That is, $p_i$ is the only prime dividing $[G:L]$. So $L\cong H_i^t:=G/({p_i})^t$ for some $1\le t\le n_i$. Each of these gives a sub-complex $ \overline{C}_*(H_i/{H_{i}^t})$ which is isomorphic to $ \Z/p_i\{[H_i/{H_{i}^t}]\} $ and $ 0 $ for $ n>0 $. This gives an element in $ E^2_{1,1} $. Consequently, $ E^2_{*,q} =0$  for $ *>1 $, and  $ E^2_{1,q} \cong \oplus_{i=1}^k(\Z/p_i)^{n_i}$ if $ q $ is odd and $ 0 $ if $ q $ is even. Hence the spectral sequence collapses to $ E^2 $-page.\par 
	Next we determine the homotopy groups of the commutative ring spectrum 	$ \Phi^G(H\uA) $. 
	For this first we determine the ring structure of $ 	\pi_{*}^G(S^{\infty \lambda^{|H_i|}}\smas H\uA) $	  using  the Tate diagram.
	\begin{lemma}\label{lem:geo Tate}
		The integer graded  ring structure of $ \pi_*({S^{\infty \lambda^{|H_i|}}}\wedge H\uA ) $ is
		\begin{align*}
			\pi_*({S^{\infty \lambda^{|H_i|}}}\wedge H\uA )&\cong \frac{A(G)}{\langle  \tr_{H_i}^GA(H_i)  \rangle}\oplus \frac{A(H_i)}{p_i}\{\frac{u_{\lambda^{|H_i|}}^{k}}{a_{\lambda^{|H_i|}}^k}\}, k>0.
		\end{align*} 
	\end{lemma}
	\begin{proof} Consider the diagram
		\begin{equation*}
			\begin{tikzcd}[scale cd=.92]
				{S(\infty \lambda^{|H_i|})}_+\wedge H\uA\arrow[r,""]\arrow["\simeq"]{d} & H\uA  \arrow[r,""]\arrow[""]{d} & {S^{\infty \lambda^{|H_i|}}}\wedge H\uA\arrow[""]{d}
				\\
				{S(\infty \lambda^{|H_i|})}_+\wedge F(S(\infty \lambda^{|H_i|})_+,H\uA)\arrow[r,""] &  F({S(\infty \lambda^{|H_i|})}_+,H\uA) \arrow[r,""] & {S^{\infty \lambda^{|H_i|}}}\wedge F({S(\infty \lambda^{|H_i|})}_+,H\uA) 
			\end{tikzcd}
		\end{equation*}
		The homotopy groups of $F({S(\infty \lambda^{|H_i|})}_+,H\uA)$  may be computed using the homotopy fixed point spectral sequence \cite[Proposition 2.8]{HM17} which  collapses at the $ E^2 $-page. As a result 
		\[
		\pi^G_{c_0+c_i\lambda^{|H_i|}}(F({S(\infty \lambda^{|H_i|})}_+,H\uA) )=A(H_i)[a_{\lambda^{|H_i|}},u_{\lambda^{|H_i|}}^{\pm}]/(p_ia_{\lambda^{|H_i|}}).
		\]
		Since $ {S^{\infty \lambda^{|H_i|}}}=\lim_{n\to \infty}S^{n{\lambda^{|H_i|}}} $, inverting the class $ a_{\lambda^{|H_i|} }$, we get
		\[
		\pi^G_{c_0+c_i\lambda^{|H_i|}}({S^{\infty \lambda^{|H_i|}}}\wedge F({S(\infty \lambda^{|H_i|})}_+,H\uA) )=\Z/p_i[a_{\lambda^{|H_i|}}^\pm,u_{\lambda^{|H_i|}}^{\pm}].
		\]
		By computing the kernel  and cokernel of the boundary  map
		\[
		\pi^G_{c_0+c_i\lambda^{|H_i|}}(F({S(\infty \lambda^{|H_i|})}_+,H\uA))\to \pi^G_{c_0+c_i\lambda^{|H_i|}}({S^{\infty \lambda^{|H_i|}}}\wedge F({S(\infty \lambda^{|H_i|})}_+,H\uA)) 
		\] 
		we see that 
		\[
		\pi^G_{c_0+c_i\lambda^{|H_i|}}({S(\infty \lambda^{|H_i|})}_+\wedge H\uA )=\oplus A(H_i){\{p_i u_{\lambda^{|H_i|}}^m\}} \oplus A(H_i)/p_i\{\Sigma^{-1}\frac{u_{\lambda^{|H_i|}}^{\ell}}{a_{\lambda^{|H_i|}}^k}\}, ~~~m, \ell\in \Z, k>0.
		\]
		In degree $ 0 $, the top cofibre sequence results the following  extension 
		\begin{equation*}
			A(H_i)\{p_i\}\to A(G)\to A(G)/\tr_{H_i}^G A(H_i).
		\end{equation*}
		Hence, the integer graded  ring structure is 
		\begin{align*}
			\pi_*({S^{\infty \lambda^{|H_i|}}}\wedge H\uA )&\cong \frac{A(G)}{\langle  \tr_{H_i}^GA(H_i)  \rangle}\oplus \frac{A(H_i)}{p_i}\{\frac{u_{\lambda^{|H_i|}}^{k}}{a_{\lambda^{|H_i|}}^k}\}, k>0.
		\end{align*} 
	\end{proof}

By Lemma \ref{lem:geo Tate} we get
	\begin{equation*}
		\pi_{*}^G(S^{\infty \lambda^{|H_i|}}\smas H\uA)\cong  \frac{A(G)}{\langle  \tr_{H_i}^G(A(H_i))  \rangle}\oplus \bigoplus_{t>0} \frac{A(H_i)}{p_i}\{(u'_i/a'_i)^t\},
	\end{equation*}
	where $ \deg  u'_i/a'_i=2 $. 
	Moreover, for each $ i $ we have the natural map of $ A
	(G) $-algebras 
	\begin{myeq}\label{eq:geo ring map}
		\psi_i:	\pi_{*}^G(S^{\infty \lambda^{|H_i|}}\smas H\uA)\to \pi_{*}^G(S^{\infty \lambda^{|H_1|}}\smas \cdots \smas S^{\infty \lambda^{|H_k|}}\smas H\uA)=\pi_*(\Phi^G(H\uA))
	\end{myeq}
	arising by inclusion. We may compare the corresponding spectral sequences of $ A(G) $-modules (as given in \eqref{eq:geo ss}) 
	to deduce that
	\begin{equation*}
		\psi_i(u'_i/a'_i)= \textup{~class represented by~} [H_i/H_i] \textup{~in~} E^2_{1,1}. 
	\end{equation*}
	We denote the image of $ \psi_i $ by $  u_{H_i}/a_{H_i}$. 
	Note that $[H_i/K]$ lies in the image of some $\tr^{H_i}_{H_i\cap H_j}$ unless $K=H_i^t$ for $[G/H_i^t]=p^t$ for some $1\le t\le n_0$.	Now using the ring map \eqref{eq:geo ring map} we get the following description of $ 	\pi_*(\Phi^{G}(H\uA)) $.  
	\begin{thm}\label{thm:geo} Let  $ G=C_n $ with $n={p_1^{n_1}\dots p_k^{n_k}} $.  For $ i=1 $ to $  k $, let $ H_i $ be  the subgroup of $ G $ of order $ |H_i|=n/p_i $. Then we have
		\begin{equation*}
			\pi_*(\Phi^{G}(H\uA))\cong \Z\{[G/G]\}\oplus \bigoplus^k_{\ell=1} \bigoplus_{t>0}\frac{A(H_\ell)}{\langle   I_{H_\ell}, p_\ell \rangle}\{(u_{H_\ell}/a_{H_\ell})^t\}
		\end{equation*}
		where $p_\ell=[G:H_\ell]$ and  the ideal $ I_{H_\ell} $  consists  the transfers
		\begin{equation*}
			\big \langle  \tr_{H_i\cap H_\ell}^{H_\ell} A(H_i\cap H_\ell) \mid 1\le i\le k, i\ne \ell\big\rangle. 
		\end{equation*}
	\end{thm}
	\begin{rmk}
		While $ \Phi^G(H\uZ) =0 $ if $ G $ is not a $ p $-group 
		(see \cite[Prop. 11]{Kri20}),   however with $ \uA $-coefficient $ \Phi^G(H\uA)  $ is non-trivial by Theorem \eqref{thm:geo}.
	\end{rmk}
	\section{Computations for the Burnside ring}\label{ringstrp}
	We show that the restriction to divisor gradings yields a simple formula in the case $G=C_p$ (compare with \cite{Lew88}) at odd primes $p$. A similar calculation may also be performed for $p=2$, but we avoid it here as in the case of $C_2$, the restriction to divisor gradings does not yield any simplification. The formula in divisor grading also help in computing the additive structure by taking the box product. 
	
	\begin{mysubsection}{Divisor gradings for $C_p$}\label{cpcase}
		We determine the ring $ \pi_{\bigstar_{\pm\textup{div}}}(H\uA) $ by using Tate square. In the case of the spectrum $ H\uZ $, this was determined for the group $ C_2 $  in \cite[Cor. 2.6]{Gre18}, and for $ C_p $ in \cite[Prop. 6.3]{Zen18}. For the spectrum $ H\underline{\F_p} $, this appears in \cite[Prop. 3.6 \& 3.8 ]{BG21a}.\par 
		The Tate diagram for $ H\uA $ is the following 
		\begin{equation*}
			\begin{tikzcd}[scale cd=.92]
				{EC_p}_+\wedge H\uA\arrow[r,""]\arrow["\simeq"]{d} & H\uA  \arrow[r,""]\arrow[""]{d} & \widetilde{EC_p}\wedge H\uA\arrow[""]{d}
				\\
				{EC_p}_+\wedge F(E{C_p}_+,H\uA)\arrow[r,""] &  F({EC_p}_+,H\uA) \arrow[r,""] & \widetilde{EC_p}\wedge F({EC_p}_+,H\uA) 
			\end{tikzcd}
		\end{equation*}
		In this case the homotopy fixed point spectral sequence \cite[Prop. 2.8]{HM17} collapses at the $ E^2 $-page and we obtain that  
		\[
		\pi_{\bigstar_{\pm \textup{div}}}(F({EC_p}_+,H\uA) )=\Z[a_\lambda,u_{\lambda}^{\pm}]/(pa_\lambda).
		\]
		Since $ \widetilde{EC_p}=\lim_{n\to \infty}S^{n\lambda} $, inverting the class $ a_\lambda $, we get
		\[
		\pi_{\bigstar_{\pm \textup{div}}}(\widetilde{EC_p}\wedge F({EC_p}_+,H\uA) )=\Z/p[a_\lambda^\pm,u_{\lambda}^{\pm}].
		\]
		By computing the kernel  and cokernel of the boundary  map
		\[ 
		\pi_{\bigstar_{\pm \textup{div}}}(F({EC_p}_+,H\uA))\to \pi_{\bigstar_{\pm \textup{div}}}(\widetilde{EC_p}\wedge F({EC_p}_+,H\uA)) 
		\] 
		we see that 
		\[ 
		\pi_{\bigstar_{\pm \textup{div}}}({EC_p}_+\wedge H\uA )=\oplus \Z{\{p u_\lambda^m\}} \oplus A/(\alpha-p,p)\{\Sigma^{-1}\frac{u_{\lambda}^{\ell}}{a_\lambda^k}\}, ~~~m, \ell\in \Z, k>0.
		\]
		Here $ \alpha=\alpha_p =[C_p/e] $. Since $ \pi_{\bigstar_{\pm \textup{div}}}(\widetilde{EC_p}_+\wedge H\uA) $ is $ a_\lambda $-periodic, its value can be determined from the integer graded homotopy groups $ \pi_n(\widetilde{EC_p}_+\wedge H\uA ) $.  Then in degree $ 0 $, we have an extension 
		$0\to \alpha\to A(C_p)\to A(C_p)/\alpha  \to 0$. As  a result
		\begin{align*}
			\pi_*(\widetilde{EC_p}_+\wedge H\uA )&= A(C_p)/\alpha\oplus A(C_p)/(\alpha-p,p)\{\frac{u_{\lambda}^k}{a_\lambda^k}\}, \textup{~where~}  k>0 \\
			&\cong A(C_p)\Big[\frac{u_{\lambda}}{a_\lambda}\Big]/(\alpha,p\frac{u_{\lambda}}{a_\lambda}).
		\end{align*} 
		Therefore, we have
		\[ 
		\pi_{\bigstar_{\pm \textup{div}}}(\widetilde{EC_p}_+\wedge H\uA)=  A(C_p)[u_{\lambda}, a_\lambda^\pm]/(\alpha,pu_{\lambda}).
		\]
		In order to determine $ \pi_{\bigstar_{\pm \textup{div}}}(H\uA)  $, consider the connecting homomorphism
		\[ 
		\pi_{\bigstar_{\pm \textup{div}}}(\widetilde{EC_p}_+\wedge H\uA) \to \pi_{\bigstar_{\pm \textup{div}}-1}({EC_p}_+\wedge H\uA).
		\]
		We see for $ j,k>0 $, the elements 
		$ A(C_p)/(\alpha,p)\{\Sigma^{-1}\frac{1}{u_\lambda^ja_\lambda^k}\} $ 
		lies in the cokernel. Further the elements $ A(C_p)/{(\alpha,pu_{\lambda})}\{u_\lambda^s\mspace{1mu} a_\lambda^t\} $ for $ s,t\ge0 $ and $  A(C_p)/\alpha\, \{pa_\lambda^{-k}\}, k>0$ lies in kernel.
		If $ m> 0 $, we  have the following non-trivial extension
		\[ 
		0\to \Z\{pu_\lambda^m\}\to A(C_p)/(\alpha-p)\{u_\lambda^m\} \to A(C_p)/(\alpha,p)\{u_\lambda^m\}\to 0 
		\]
		and if $ m<0$ then we have the following
		\[ 
		\Z\{pu_\lambda^m\}\xrightarrow{~\cong~} A(C_p)/(\alpha-p)\{{\alpha u_\lambda^m}\}.  
		\]
		\[ 
		\pi_{\bigstar_{\pm \textup{div}}}(H\uA) =  A(C_p)[u_{\lambda}, a_\lambda]/\big(\alpha a_\lambda,(\alpha-p)u_{\lambda}\big)\oplus A(C_p)/(\alpha-p)\{{\alpha u_\lambda^{-l}}\}
		\]
		\[
		\hspace{2cm}\oplus A(C_p)/\alpha\{p\mspace{1mu} a_\lambda^{-k}\}\oplus A(C_p)/(\alpha,p)\{\Sigma^{-1}\frac{1}{u_\lambda^ja_\lambda^k}\}.
		\]

	\end{mysubsection}
	
	\begin{mysubsection}{The additive structure in mostly zero cases for $C_{pq}$}\label{mostlyzerocase}
		The Mackey functors $ \underline{\pi}_\alpha^{C_p}(H\uA)$ depend only on the fixed point dimensions  except in the cases when both $|\alpha| = |\alpha^{C_p}|=0$ \cite{Lew88}. Using the box product with invertible Mackey functors $A[\tau ] $ of \S \ref{invertiblemod}, one may determine the additive structure of $\upi^G_\alpha(H\uA)$ 	if the value of $\upi_\beta^G(H\uA)$ is known for $\beta\in \bigstar_{\pm \textup{div}}$ and $|\alpha^H|=|\beta^H|$ for all subgroups $H$. We use this idea to demonstrate the additive structure for $G=C_{pq}$ in the cases left unresolved in \cite{BG19}.
	\end{mysubsection}
	\begin{defn}
		An $ \alpha\in RO({C_n}) $  is said to be \textit{non-zero} if all the fixed point dimensions  $ |\alpha^{C_d}| $ are non-zero, and \textit{mostly non-zero} if  $ |\alpha^{C_d}|=0 $ implies $ |\alpha^{C_{dp}}|\neq 0 $ for all $ p $ dividing $ n $. We say $ \alpha $ has \textit{many-zeros} if it is not any of the above. 
	\end{defn}
	For square free $n$, the Mackey functor valued cohomology groups   $ \underline{\pi}_\alpha^{C_n}(H\uA) $ in the case $ \alpha $  non-zero or mostly non-zero,  and the  group structure  $ \pi_\alpha^{C_n}(H\uA) $ for $\alpha$ with many-zeros   were computed in \cite[Theorem 6.19 and Theorem 6.26]{BG20}.\par
	We show how to determine the Mackey functors using the box product, when $n=pq$ with both $p,q$ odd primes. Up to symmetry  between $p$ and $q$, $\alpha\in RO(C_{pq})$  with many-zeros fit into one of the following cases
	\begin{enumerate}
		\item[1.] $|\alpha^H|=0$ for all $H\le C_{pq}$. 
		\item [2.] $|\alpha^H|= 0$ for $H\in \{C_p,C_q,C_{pq}\}$, and  $|\alpha|\ne 0$.
		\item [3.] $|\alpha^H|= 0$ for $H\in\{e,C_q,C_{pq}\}$, and  $|\alpha^{C_p}|\ne 0$.
		\item[4.] $|\alpha^H|= 0$ for $H\in\{e, C_p,C_q\}$, and  $|\alpha^{C_{pq}}|\ne 0$. 
		\item[5.] $|\alpha^H|= 0$ for $H\in \{C_q, C_{pq}\}$, and  $|\alpha|\ne 0, |\alpha^{C_p}|\ne 0$.  
		\item [6.] $|\alpha^H|= 0$ for $H\in \{e, C_q\}$, and  $|\alpha^{C_p}|\ne 0, |\alpha^{C_{pq}}|\ne 0$.
	\end{enumerate}
	The ingredients used for our computations are as follows. Suppose $\uM=\upi^{C_{pq}}_\alpha(H\uA)$. The information of the groups  $\underline{M}(C_{pq}/C_p), \underline{M}(C_{pq}/e) $ along with $\res^{C_p}_e$, $\tr^{C_p}_e$  maps between them can be derived from   $\downarrow^{C_{pq}}_{C_p}\upi^{C_{pq}}_{\alpha}(H\uA)\cong \upi^{C_p}_{\alpha}(H\uA)$. The structure of the latter Mackey functor is known by \cite{Lew88} (see \S  \ref{genBurnmf}).  Similarly, $\downarrow^{C_{pq}}_{C_q}\upi^{C_{pq}}_{\alpha}(H\uA)\cong \upi^{C_q}_{\alpha}(H\uA)$ determines  $\underline{M}(C_{pq}/C_q), \underline{M}(C_{pq}/e) $ together with $\res^{C_q}_e$, $\tr^{C_q}_e$. Here for $\uM\in \mathcal{M}_G$,  $\downarrow^G_H \uM(S):=\uM(G\times_H S )$, where $H\le G$ and $S$ is an $H$-set. We also use \cite{BG19}.\par
	Next we write down the additive structure in each case.\\
	\textbf{Case 1}: This is already covered by Theorem \ref{mfallze}. Following \S \ref{genBurnmf} we see that $\tau(\alpha)$ is determined by integers $\tau _p, \tau _q, \mbox{ and } \tau _e$. Hence $ \upi_\alpha(H\uA) $ is
	\small
	\[
	\xymatrix{ & & \Z^4\ar@/_1.1pc/[dll]_{{ \left[\begin{smallmatrix}	\tau_p & q & 0 & 0 \\	0 & 0 & \tau_p\tau_e  & q	\end{smallmatrix}\right]}}   \ar@/^1.1pc/[drr]^{{ \left[\begin{smallmatrix} \tau_q & 0 & p & 0 \\ 0  & \tau_q\tau_e & 0 & p \end{smallmatrix}\right]}}  & & \\ 
		\Z^2 \ar@/_.7pc/[urr]|{{\tiny \left[\begin{smallmatrix}	 0 & 0 \\1 & 0 \\ 0 & 0\\ 0 & 1  \end{smallmatrix}\right]}} \ar@/_1.1pc/[drr]_{{ \left[\begin{smallmatrix}   \tau_q\tau_e & p \end{smallmatrix}\right]}} & & & & \Z^2 \ar@/^.7pc/[ull]|{{\tiny \left[\begin{smallmatrix} 0 & 0 \\0 & 0 \\ 1 & 0\\ 0 & 1  \end{smallmatrix}\right]}} \ar@/^1.1pc/[dll]^{{ \left[\begin{smallmatrix} \tau_p\tau_e & q \end{smallmatrix}\right]}} \\ 
		&  &  \Z \ar@/_.7pc/[ull]|{{\tiny \left[\begin{smallmatrix} 0 \\ 1 \end{smallmatrix}\right]}} \ar@/^.7pc/[urr]|{{\tiny \left[\begin{smallmatrix} 0 \\ 1 \end{smallmatrix}\right]}}&&}
	\]
	\normalsize
	\textbf{Case 2}: In this case, there exists a $k\ne 0$ so that  $\alpha-k\lambda\in RO_0{(C_{pq})}$, and thus,  
	\begin{myeq}\label{boxuni}
		\upi_\alpha (H\uA)\cong \upi_{k \lambda}H\uA\Box\upi_{\alpha-k \lambda}H\uA.
	\end{myeq}
	Note that for $k>0$, $\upi_{-k\lambda}(H\uA)\cong \upi_{-\lambda}(H\uA) =\textup{Coker}(\uZ^\ast\to \uA):=\underline{C}$ and  $\upi_{k\lambda}(H\uA)\cong \textup{Ker}(\uA\to \uZ):=\underline{K}$. These works out to be 
	\small
	\[
	\xymatrix{\\\underline{C}:}
	\xymatrix{ & &\Z^3\ar@/_1pc/[dll]_{{ \left[\begin{smallmatrix} 1 &q &0 \end{smallmatrix}\right]}}   \ar@/^1pc/[drr]^{{ \left[\begin{smallmatrix} 1 & 0 &p   \end{smallmatrix}\right]}}  & &\\ 
		\Z \ar@/_.7pc/[urr]|{{\tiny \left[\begin{smallmatrix} 0  \\ 1 \\ 0  \end{smallmatrix}\right]}} \ar@/_1pc/[drr] & & & &\Z \ar@/^.7pc/[ull]|{{\tiny \left[\begin{smallmatrix} 0 \\ 0 \\ 1  \end{smallmatrix}\right]}} \ar@/^1pc/[dll] \\ 
		&  &  0 \ar@/_.7pc/[ull] \ar@/^.7pc/[urr]}
	\hspace{1.5cm}
	\xymatrix{\\\underline{K}:}
	\xymatrix{ & &\Z^3\ar@/_1pc/[dll]_{{ \left[\begin{smallmatrix} \gamma_p&q &0 \end{smallmatrix}\right]}}   \ar@/^1pc/[drr]^{{ \left[\begin{smallmatrix} \gamma_q & 0 &p   \end{smallmatrix}\right]}}  & &\\ 
		\Z \ar@/_.7pc/[urr]|{{\tiny \left[\begin{smallmatrix} 0  \\ 1 \\ 0  \end{smallmatrix}\right]}} \ar@/_1pc/[drr] & & & &\Z \ar@/^.7pc/[ull]|{{\tiny \left[\begin{smallmatrix} 0 \\ 0 \\ 1  \end{smallmatrix}\right]}} \ar@/^1pc/[dll] \\ 
		&  &  0 \ar@/_.7pc/[ull] \ar@/^.7pc/[urr]}
	\]
	\normalsize
	where  $\gamma_p$ and $\gamma_q$ satisfies $\gamma_p p+\gamma_qq=1$.
 By taking box product with $\underline{C}$ or with $\underline{K}$, we find $\upi_\alpha(H\uA)$ to be one of the following 
	\[
	\xymatrix{ & &\Z^3\ar@/_1pc/[dll]_{{ \left[\begin{smallmatrix} \tau _p &q &0 \end{smallmatrix}\right]}}   \ar@/^1pc/[drr]^{{ \left[\begin{smallmatrix} \tau _q & 0 &p   \end{smallmatrix}\right]}}  & &\\ 
		\Z \ar@/_.7pc/[urr]|{{\tiny \left[\begin{smallmatrix} 0  \\ 1 \\ 0  \end{smallmatrix}\right]}} \ar@/_1pc/[drr] & & & &\Z \ar@/^.7pc/[ull]|{{\tiny \left[\begin{smallmatrix} 0 \\ 0 \\ 1  \end{smallmatrix}\right]}} \ar@/^1pc/[dll] \\ 
		&  &  0 \ar@/_.7pc/[ull] \ar@/^.7pc/[urr]}
	\hspace{2cm}
	\xymatrix{ & &\Z^3\ar@/_1pc/[dll]_{{ \left[\begin{smallmatrix} \tau _p\gamma_p&q &0 \end{smallmatrix}\right]}}   \ar@/^1pc/[drr]^{{ \left[\begin{smallmatrix}  \tau _q\gamma_q & 0 &p   \end{smallmatrix}\right]}}  & &\\ 
		\Z \ar@/_.7pc/[urr]|{{\tiny \left[\begin{smallmatrix} 0  \\ 1 \\ 0  \end{smallmatrix}\right]}} \ar@/_1pc/[drr] & & & &\Z \ar@/^.7pc/[ull]|{{\tiny \left[\begin{smallmatrix} 0 \\ 0 \\ 1  \end{smallmatrix}\right]}} \ar@/^1pc/[dll] \\ 
		&  &  0 \ar@/_.7pc/[ull] \ar@/^.7pc/[urr]}
	\]
depending on whether $|\alpha|<0$ or $>0$. Here $\tau_p, \tau_q, \tau_e$ are integers that determine $\tau(\alpha - k \lambda)$. \\
	\textbf{Case 3}: Again  there exists $k\ne 0$ so that $\alpha-k(\lambda-\lambda^p)\in RO_0({C_{pq}})$. For $k>0$, let $\underline{L}^{(k)}_1:=\upi_{k\lambda^p-k\lambda}(H\uA)$  and  $\underline{L}^{(k)}_2:=\upi_{k\lambda-k\lambda^p}(H\uA)$. This fits into
	\begin{equation*}
		{%
			\xymatrix@R=0.5cm@C=0.5cm{
				0\ar[r]&	\upi_{k\lambda^p-(k-1)\lambda}(H\uA)\ar[r]^-{\varphi} \ar[d]_-{\cong}
				& \upi_{k\lambda^p-k\lambda}(H\uA) \ar[r]^-{} \ar[d]_-{\cong}&
				\upi_{k\lambda^p-(k-1)\lambda-1}(S(\lambda)_+\smas H\uA)\ar[r]^-{}\ar[d]_-{\cong}&0\\ 
				&	\uA_p\boxtimes \bZ_q	\ar[r] &\underline{L}^{(k)}_1\ar[r]&\uZ
		}}
	\end{equation*}
	Note that $\downarrow^{C_{pq}}_{C_q}\upi^{C_{pq}}_{k\lambda^p-k\lambda}(H\uA)\cong \upi^{C_q}_{k\lambda^p-k\lambda}(H\uA)\cong \uA[\gamma_p^k]$ where  $ p \gamma_p\equiv1\pmod q $ by \eqref{eqtaual}.  Let  $\underline{L}^{(k)}_1(C_{pq}/C_p)\cong \Z\{x_p\} $,  $\underline{L}^{(k)}_1(C_{pq}/e)\cong \Z\{x_e\} $,   $\underline{L}^{(k)}_1(C_{pq}/C_q)\cong \Z\{x_q, \tr^{C_q}_e x_e\} $, and $\uA_p \boxtimes \bZ_q({C_{pq}}/C_{pq})\cong \Z\{y_{pq}, \tr^{C_{pq}}_{C_q}(y_q)\}$. Then we may choose the generators for $\underline{L}^{(k)}_1(C_{pq}/C_{pq}) $ as $\{\varphi(y_{pq})+\gamma _p^k \cdot \tr^{C_{pq}}_{C_p}(x_p), \tr^{C_{pq}}_{C_q}(x_q), \tr^{C_{pq}}_{C_p}(x_p)\}$. A similar computation determines $\underline{L}^{(k)}_2$.  
	\[
	\xymatrix{\\\underline{L}^{(k)}_1:}
	\xymatrix{ & &\Z^3\ar@/_1.1pc/[dll]_{{ \left[\begin{smallmatrix} \gamma_p^kq &\gamma_p^kp &q  \end{smallmatrix}\right]}}   \ar@/^1.1pc/[drr]^{{ \left[\begin{smallmatrix} q & p &0  \\ 0  &0 &1 \end{smallmatrix}\right]}}  & &\\ 
		\Z \ar@/_.7pc/[urr]|{{\tiny \left[\begin{smallmatrix} 0  \\0 \\ 1   \end{smallmatrix}\right]}} \ar@/_1.1pc/[drr]_1 & & & &\Z^2 \ar@/^.7pc/[ull]|{{\tiny \left[\begin{smallmatrix} 0 &0 \\1 &0 \\ 0 & p  \end{smallmatrix}\right]}} \ar@/^1.1pc/[dll]^{{ \left[\begin{smallmatrix} \gamma_p^k &q \end{smallmatrix}\right]}} \\ 
		&  &  \Z \ar@/_.7pc/[ull]_ p \ar@/^.7pc/[urr]|{{\tiny \left[\begin{smallmatrix} 0 \\ 1 \end{smallmatrix}\right]}}}
	\hspace{1.5cm}
	\xymatrix{\\\underline{L}^{(k)}_2:}
	\xymatrix{ & &\Z^3\ar@/_1.1pc/[dll]_{{ \left[\begin{smallmatrix} p^{k-1} &p^k &q  \end{smallmatrix}\right]}}   \ar@/^1.1pc/[drr]^{{ \left[\begin{smallmatrix} 1 & p &0  \\ 0 &0 &p \end{smallmatrix}\right]}}  & &\\ 
		\Z \ar@/_.7pc/[urr]|{{\tiny \left[\begin{smallmatrix} 0  \\0 \\ 1   \end{smallmatrix}\right]}} \ar@/_1.1pc/[drr]_p & & & &\Z^2 \ar@/^.7pc/[ull]|{{\tiny \left[\begin{smallmatrix} 0 &0 \\1 &0 \\ 0 & 1  \end{smallmatrix}\right]}} \ar@/^1.1pc/[dll]^{{ \left[\begin{smallmatrix} p^k &q \end{smallmatrix}\right]}} \\ 
		&  &  \Z \ar@/_.7pc/[ull]_ 1 \ar@/^.7pc/[urr]|{{\tiny \left[\begin{smallmatrix} 0 \\ 1 \end{smallmatrix}\right]}}}
	\]
	After taking box product as in \eqref{boxuni},	$\upi_\alpha(H\uA)$  works out to be
	\[
	\xymatrix{ & &\Z^3\ar@/_1.1pc/[dll]_{{ \left[\begin{smallmatrix} \gamma_p^kq \tau _p\tau _q\tau _e &\gamma_p^kp\tau _p\tau _e &q  \end{smallmatrix}\right]}}   \ar@/^1.1pc/[drr]^{{ \left[\begin{smallmatrix} q\tau _q & p &0  \\ 0  &0 &1 \end{smallmatrix}\right]}}  & &\\ 
		\Z \ar@/_.7pc/[urr]|{{\tiny \left[\begin{smallmatrix} 0  \\0 \\ 1   \end{smallmatrix}\right]}} \ar@/_1.1pc/[drr]_1 & & & &\Z^2 \ar@/^.7pc/[ull]|{{\tiny \left[\begin{smallmatrix} 0 &0 \\1 &0 \\ 0 & p  \end{smallmatrix}\right]}} \ar@/^1.1pc/[dll]^{{ \left[\begin{smallmatrix} \gamma_p^k\tau _p\tau _e &q \end{smallmatrix}\right]}} \\ 
		&  &  \Z \ar@/_.7pc/[ull]_ p \ar@/^.7pc/[urr]|{{\tiny \left[\begin{smallmatrix} 0 \\ 1 \end{smallmatrix}\right]}}}
	\hspace{1cm}
	\xymatrix{ & &\Z^3\ar@/_1.1pc/[dll]_{{ \left[\begin{smallmatrix} p^{k-1}\tau _p\tau _q\tau _e &p^k\tau _p\tau _e &q  \end{smallmatrix}\right]}}   \ar@/^1.1pc/[drr]^{{ \left[\begin{smallmatrix} \tau _q & p &0  \\ 0 &0 &p \end{smallmatrix}\right]}}  & &\\ 
		\Z \ar@/_.7pc/[urr]|{{\tiny \left[\begin{smallmatrix} 0  \\0 \\ 1   \end{smallmatrix}\right]}} \ar@/_1.1pc/[drr]_p & & & &\Z^2 \ar@/^.7pc/[ull]|{{\tiny \left[\begin{smallmatrix} 0 &0 \\1 &0 \\ 0 & 1  \end{smallmatrix}\right]}} \ar@/^1.1pc/[dll]^{{ \left[\begin{smallmatrix} p^k\tau _p\tau _e &q \end{smallmatrix}\right]}} \\ 
		&  &  \Z \ar@/_.7pc/[ull]_ 1 \ar@/^.7pc/[urr]|{{\tiny \left[\begin{smallmatrix} 0 \\ 1 \end{smallmatrix}\right]}}}
	\]
The integers $\tau_p, \tau_q, \tau_e$ are obtained from $\tau(\alpha-k(\lambda-\lambda^p))$.\\
	\textbf{Case 4}: Let $\beta= \lambda-\lambda^p-\lambda^q+2$. There exists $k\ne 0$  so that  $\alpha-k\beta\in RO_0({C_{pq}})$. For $k>0$, let $\underline{N}^{(k)}_1:=\upi_{k\beta}(H\uA)$  and  $\underline{N}^{(k)}_2:=\upi_{-k\beta}(H\uA)$. This fits into
	\begin{equation*}
		{%
			\xymatrix@R=0.5cm@C=0.5cm{
					\upi_{k\beta}(S(\lambda)_+\smas H\uA)\cong \uZ^*\ar[r]^-{} 
				& \upi_{k\beta}(H\uA) \ar[r]^-{\varphi} 
				&
				\upi_{k\beta-\lambda}( H\uA)\ar[r]^-{}
				&0.
		}}
	\end{equation*}
	We determine 
	\[
	\xymatrix{\\\underline{N}^{(k)}_1:}
	\xymatrix{ & &\Z^3\ar@/_1pc/[dll]_{{ \left[\begin{smallmatrix} 1 &0 &0 \\  0& p^{k-1}  &q \end{smallmatrix}\right]}}   \ar@/^1pc/[drr]^{{ \left[\begin{smallmatrix}  0 &1 &0 \\ q^{k-1} &0 &p \end{smallmatrix}\right]}}  & &\\ 
		\Z^2 \ar@/_.7pc/[urr]|{{\tiny \left[\begin{smallmatrix} q &0 \\0 &0 \\ 0 & 1  \end{smallmatrix}\right]}} \ar@/_1pc/[drr]_{{ \left[\begin{smallmatrix} q^k & p \end{smallmatrix}\right]}} & & & &\Z^2 \ar@/^.7pc/[ull]|{{\tiny \left[\begin{smallmatrix} 0 &0 \\p &0 \\ 0 &1  \end{smallmatrix}\right]}} \ar@/^1pc/[dll]^{{ \left[\begin{smallmatrix} p^k &q \end{smallmatrix}\right]}} \\ 
		&  &  \Z \ar@/_.7pc/[ull]|{{\tiny \left[\begin{smallmatrix} 0 \\ 1 \end{smallmatrix}\right]}} \ar@/^.7pc/[urr]|{{\tiny \left[\begin{smallmatrix} 0 \\ 1 \end{smallmatrix}\right]}}}
	\hspace{.9cm}
	\xymatrix{\\\underline{N}^{(k)}_2:}
	\xymatrix{ & &\Z^3\ar@/_1pc/[dll]_{{ \left[\begin{smallmatrix} q &0 &0 \\  0& \gamma_p^k  &q \end{smallmatrix}\right]}}   \ar@/^1pc/[drr]^{{ \left[\begin{smallmatrix}  0 &p &0 \\ \gamma_q^k &0 &p \end{smallmatrix}\right]}}  & &\\ 
		\Z^2 \ar@/_.7pc/[urr]|{{\tiny \left[\begin{smallmatrix} 1 &0 \\0 &0 \\ 0 & 1  \end{smallmatrix}\right]}} \ar@/_1pc/[drr]_{{ \left[\begin{smallmatrix} \gamma_q^k & p \end{smallmatrix}\right]}} & & & &\Z^2 \ar@/^.7pc/[ull]|{{\tiny \left[\begin{smallmatrix} 0 &0 \\1 &0 \\ 0 &1  \end{smallmatrix}\right]}} \ar@/^1pc/[dll]^{{ \left[\begin{smallmatrix} \gamma_p^k &q \end{smallmatrix}\right]}} \\ 
		&  &  \Z \ar@/_.7pc/[ull]|{{\tiny \left[\begin{smallmatrix} 0 \\ 1 \end{smallmatrix}\right]}} \ar@/^.7pc/[urr]|{{\tiny \left[\begin{smallmatrix} 0 \\ 1 \end{smallmatrix}\right]}}}
	\]
	where   $q\gamma_q\equiv 1\pmod p$ and   $p\gamma_p\equiv 1\pmod q$.
	After taking box product as in \eqref{boxuni},	$\upi_\alpha(H\uA)$  turns out to be
	\[
	\xymatrix{ & &\Z^3\ar@/_1pc/[dll]_{{ \left[\begin{smallmatrix} 1 &0 &0 \\  0& p^{k-1}\tau _p\tau _e  &q \end{smallmatrix}\right]}}   \ar@/^1pc/[drr]^{{ \left[\begin{smallmatrix}  0 &1 &0 \\ q^{k-1}\tau _q\tau _e &0 &p \end{smallmatrix}\right]}}  & &\\ 
		\Z^2 \ar@/_.7pc/[urr]|{{\tiny \left[\begin{smallmatrix} q &0 \\0 &0 \\ 0 & 1  \end{smallmatrix}\right]}} \ar@/_1pc/[drr]_{{ \left[\begin{smallmatrix} q^k\tau _q\tau _e & p \end{smallmatrix}\right]}} & & & &\Z^2 \ar@/^.7pc/[ull]|{{\tiny \left[\begin{smallmatrix} 0 &0 \\p &0 \\ 0 &1  \end{smallmatrix}\right]}} \ar@/^1pc/[dll]^{{ \left[\begin{smallmatrix} p^k\tau _p\tau _e &q \end{smallmatrix}\right]}} \\ 
		&  &  \Z \ar@/_.7pc/[ull]|{{\tiny \left[\begin{smallmatrix} 0 \\ 1 \end{smallmatrix}\right]}} \ar@/^.7pc/[urr]|{{\tiny \left[\begin{smallmatrix} 0 \\ 1 \end{smallmatrix}\right]}}}
	\hspace{1.5cm}
	\xymatrix{ & &\Z^3\ar@/_1pc/[dll]_{{ \left[\begin{smallmatrix} q &0 &0 \\  0& \gamma_p^k  \tau _p\tau _e &q \end{smallmatrix}\right]}}   \ar@/^1pc/[drr]^{{ \left[\begin{smallmatrix}  0 &p &0 \\ \gamma_q^k\tau _q\tau _e &0 &p \end{smallmatrix}\right]}}  & &\\ 
		\Z^2 \ar@/_.7pc/[urr]|{{\tiny \left[\begin{smallmatrix} 1 &0 \\0 &0 \\ 0 & 1  \end{smallmatrix}\right]}} \ar@/_1pc/[drr]_{{ \left[\begin{smallmatrix} \gamma_q^k \tau _q\tau _e& p \end{smallmatrix}\right]}} & & & &\Z^2 \ar@/^.7pc/[ull]|{{\tiny \left[\begin{smallmatrix} 0 &0 \\1 &0 \\ 0 &1  \end{smallmatrix}\right]}} \ar@/^1pc/[dll]^{{ \left[\begin{smallmatrix} \gamma_p^k \tau _p\tau _e &q \end{smallmatrix}\right]}} \\ 
		&  &  \Z \ar@/_.7pc/[ull]|{{\tiny \left[\begin{smallmatrix} 0 \\ 1 \end{smallmatrix}\right]}} \ar@/^.7pc/[urr]|{{\tiny \left[\begin{smallmatrix} 0 \\ 1 \end{smallmatrix}\right]}}}
	\]
where $\tau_p$, $\tau_q$, and $\tau_e$ are obtained as before.	\\
	\textbf{Case 5}:  There exists $k\ne 0$ with $k\ne -\ell$ such that $\alpha-(k\lambda^p+\ell\lambda)\in RO_0(C_{pq})$. Let $\beta=k\lambda^p+\ell\lambda$. If   either ($k>0$ and $\ell> -k$) or ($k<0$ and $\ell< -k$), then $\upi_\beta(H\uA)\cong \upi_{\lambda^p}(H\uA)\cong \uA_p \boxtimes \bZ_q $. After taking box product we get $\upi_\alpha(H\uA)\cong \uA_p[\tau _q] \boxtimes \bZ_q $. Next suppose $k<0$ and $\ell>-k$. Then  
	$$
	\upi_{\beta}(H\uA)\cong \upi_{k\lambda^p+(-k+1)\lambda}(H\uA)=\textup{Ker}(\upi_{k\lambda^p-k\lambda}(H\uA)\to \uZ) \cong  \uA_p[\gamma_q] \boxtimes \bZ_q
	$$
	 where $q\gamma_q+p\gamma_p=1$. Hence $\upi_\alpha(H\uA)\cong \uA_p[\gamma_q\tau _q] \boxtimes \bZ_q $. Finally, let $k>0$ and $\ell<-k$. Then $\upi_{\beta}(H\uA)\cong \textup{Coker}(\uZ^*\to \upi_{k \lambda^p-k \lambda}(H\uA))\cong \langle  \Z/p \rangle_p \boxtimes  \uZ_q^*\oplus  \uA_p[q]\boxtimes \bZ_q$. Hence $\upi_\alpha(H\uA)\cong \langle  \Z/p \rangle_p \boxtimes  \uZ_q^*\oplus\uA_p[q\tau _q] \boxtimes \bZ_q $, where $\tau_q$ is defined as before.\\
	\textbf{Case 6}: There exists $k,\ell,\ne0$ such that $\alpha-k(\lambda^p-\lambda)-\ell(2-\lambda^q)\in RO_{C_{pq}}$. Let $\beta=k(\lambda^p-\lambda)+\ell(2-\lambda^q)$. This fits into 
	\begin{myeq}\label{cses}
		\upi_{\beta}(S(\lambda)_+\smas H\uA)\to
		\upi_{\beta}(H\uA) \xrightarrow{\varphi} 
		\upi_{\beta-\lambda}( H\uA)\to
		0
	\end{myeq}
	where the first term is isomorphic to $ \uZ^*$. Four cases may arise:  (a) $2k+2\ell<0$ and $2\ell<0$. (b) $2k+2\ell<0$ and $2\ell>0$. (c)  $2k+2\ell>0$ and $2\ell<0$. (d) $2k+2\ell>0$ and $2\ell>0$. Associated Mackey functors are denoted by $ {U}^{(k)}_-$, ${U}^{(k)}_+$,  ${V}^{(k)}_-$, ${V}^{(k)}_+$.
	Using \eqref{cses} we deduce the following: $U_-^{(k)}\cong \uZ^*_p\boxtimes \uA_q[\gamma_p^k] $, $V_+^{(k)}\cong \uZ_p\boxtimes \uA_q[\gamma_p^k]$, and
	\[
	\xymatrix{\\{U}^{(k)}_+:}
	\xymatrix{ & &\Z^2\oplus\Z/q\ar@/_1.1pc/[dll]_{{ \left[\begin{smallmatrix} p^{k-1} &q&0  \end{smallmatrix}\right]}}   \ar@/^1.1pc/[drr]^{{ \left[\begin{smallmatrix}  1 &0 &0\\0&p&0 \end{smallmatrix}\right]}}  & &\\ 
		\Z \ar@/_.7pc/[urr]|{{\tiny \left[\begin{smallmatrix} 0  \\  1\\0  \end{smallmatrix}\right]}} \ar@/_1.1pc/[drr]_p & & & &\Z^2 \ar@/^.7pc/[ull]|{{\tiny \left[\begin{smallmatrix} p &0   \\0 &1\\0&0 \end{smallmatrix}\right]}} \ar@/^1.1pc/[dll]^{{\tiny \left[\begin{smallmatrix} p^k &q  \end{smallmatrix}\right]}} \\ 
		&  &  \Z \ar@/_.7pc/[ull]_1 \ar@/^.7pc/[urr]|{{\tiny \left[\begin{smallmatrix} 0 \\1  \end{smallmatrix}\right]}}}
		\hspace{.4cm}
	\xymatrix{\\{V}^{(k)}_-:}
	\xymatrix{ & &\Z^2\ar@/_1pc/[dll]_{{ \left[\begin{smallmatrix} p\gamma_p^k &q  \end{smallmatrix}\right]}}   \ar@/^1pc/[drr]^{{ \left[\begin{smallmatrix}  p&0\\0&1 \end{smallmatrix}\right]}}  & &\\ 
		\Z \ar@/_.7pc/[urr]|{{\tiny \left[\begin{smallmatrix} 0 \\1   \end{smallmatrix}\right]}} \ar@/_1pc/[drr]_1 & & & &\Z^2 \ar@/^.7pc/[ull]|{{\tiny \left[\begin{smallmatrix} 1&0 \\0&p  \end{smallmatrix}\right]}} \ar@/^1pc/[dll]^{{ \left[\begin{smallmatrix} \gamma_p^k &q \end{smallmatrix}\right]}} \\ 
		&  &  \Z \ar@/_.7pc/[ull]_p \ar@/^.7pc/[urr]|{{\tiny \left[\begin{smallmatrix} 0 \\ 1 \end{smallmatrix}\right]}}}
	\]
	Here $ p\gamma_p\equiv 1 \pmod q $. In the case of ${U}_-^{(k)}$ and ${V}_+^{(k)}$, we get $p^k$ in place of $\gamma^k_p$ if $k<0$. After taking box product as in \eqref{boxuni},	$\upi_\alpha(H\uA)$  turns out to be respectively  $\uZ^*_p\boxtimes \uA_q[\gamma_p^k \tau _p\tau _e]$,  $\uZ_p\boxtimes \uA_q[\gamma_p^k \tau _p\tau _e]$, 
		\[
	\xymatrix{ & &\Z^2\oplus\Z/q\ar@/_1.1pc/[dll]_{{ \left[\begin{smallmatrix} p^{k-1}\tau _p\tau _e &q&0  \end{smallmatrix}\right]}}   \ar@/^1.1pc/[drr]^{{ \left[\begin{smallmatrix}  1 &0 &0\\0&p&0 \end{smallmatrix}\right]}}  & &\\ 
		\Z \ar@/_.7pc/[urr]|{{\tiny \left[\begin{smallmatrix} 0  \\  1\\0  \end{smallmatrix}\right]}} \ar@/_1.1pc/[drr]_p & & & &\Z^2 \ar@/^.7pc/[ull]|{{\tiny \left[\begin{smallmatrix} p &0   \\0 &1\\0&0 \end{smallmatrix}\right]}} \ar@/^1.1pc/[dll]^{{\tiny \left[\begin{smallmatrix} p^k\tau _p\tau _e &q  \end{smallmatrix}\right]}} \\ 
		&  &  \Z \ar@/_.7pc/[ull]_1 \ar@/^.7pc/[urr]|{{\tiny \left[\begin{smallmatrix} 0 \\1  \end{smallmatrix}\right]}}}
		\hspace{.5cm}
	\xymatrix{ & &\Z^2\ar@/_1pc/[dll]_{{ \left[\begin{smallmatrix} p\gamma_p^k\tau _p\tau _e &q  \end{smallmatrix}\right]}}   \ar@/^1pc/[drr]^{{ \left[\begin{smallmatrix}  p&0\\0&1 \end{smallmatrix}\right]}}  & &\\ 
		\Z \ar@/_.7pc/[urr]|{{\tiny \left[\begin{smallmatrix} 0 \\1   \end{smallmatrix}\right]}} \ar@/_1pc/[drr]_1 & & & &\Z^2 \ar@/^.7pc/[ull]|{{\tiny \left[\begin{smallmatrix} 1&0 \\0&p  \end{smallmatrix}\right]}} \ar@/^1pc/[dll]^{{ \left[\begin{smallmatrix} \gamma_p^k\tau _p\tau _e &q \end{smallmatrix}\right]}} \\ 
		&  &  \Z \ar@/_.7pc/[ull]_p \ar@/^.7pc/[urr]|{{\tiny \left[\begin{smallmatrix} 0 \\ 1 \end{smallmatrix}\right]}}}
	\]
where $\tau_p$, $\tau_q$, and $\tau_e$ are obtained analogously as before.

	\section{The positive cone of $ \pi^G_\bigstar(H\uA) $}\label{secrepA}
	The objective of this section is to describe  $\pi_{\bigstar_\div}^{G} (H\uA) $ in two cases: $ (a) $ $ |G| $ is a prime power $ = p^m $. $ (b)$ $|G| $ is square free  $= p_1\cdots p_k $. We use different methods in the cases $ (a) $ and $ (b) $. Recall that  $ \bigstar^e\subseteq \bigstar_\div $   consists of  linear combinations of the form  $ \ell -(\sum_{d_i\mid n}~~b_i\mspace{2mu} \lambda^{d_i}) $  where $  \ell  \in \Z$ and $ b_i \ge 0$.
	\begin{mysubsection}{The prime power  case}\label{subsecprimeA}
		We determine the coefficient ring $\pi_{*}^{\G} (\Sigma^VH\uA) $ via an explicit cellular decomposition of $ S^V $, where $ V=\sum_{i=0}^{m-1} b_i\lambda^{p^i} $ is a representation of $\G= C_{p^m} $.  In this method, we first determine  $\pi_{\bigstar^e}^{G} (H\uA)  $, and then by adding a copy of $ \sigma $ complete the picture of $\pi_{\bigstar_\div}^{G} (H\uA) $. The following  is a prototype of the calculation.
		\begin{exam}\label{ex:cell decom Cp^n}
			Let $ V=\lambda + \lambda^p $. 	Let $ g\in \G $ be the fixed generator, and the  equivariant cell decomposition of $ S^V $ is given by
			\begin{align*}
				S^{\lambda + \lambda^{p}} &=S^{\lambda^{p}} \smas[\G/\G_+ \smas e^0 \cup \G/e_+ \smas e^1\underset{1-g}{\cup}\G/e_+ \smas e^2]\\
				&=	 \{\G/\G_+ \smas e^0\cup \G/C_{p+}\smas e^1 \underset{1-g}{\cup} \G/C_{p+} \smas e^2\} \cup \G/e_+ \smas e^3 \underset{1-g}{\cup} \G/e_+ \smas e^4.
			\end{align*}
			In the above the $ 2 $-skeleton is $ S^{\lambda^p} $, and the $ 3 $-cell and the $ 4 $-cell come from the identification $ S^{\lambda^p}\smas {G/e}_+\smas e^n\simeq  {G/e}_+\smas e^{n+2} $.	So the equivariant cellular chain complex becomes
			\begin{equation*}
				A(e)\xrightarrow{0} A(e)\xrightarrow{\partial_3} A(C_p)\xrightarrow{0} A(C_p) \xrightarrow{\tr_{C_p}^\G} A(\G).
			\end{equation*}
			Observe that the map $ \partial_3 $ arises via the composite
			\begin{equation*}
				{\G/e}_+ \smas S^{2}\simeq 	{\G/e}_+ \smas S^{\lambda^p} \longrightarrow  \G/\G_+ \smas S^{\lambda^p} \longrightarrow {\G/C_p}_+ \smas S^{\lambda^p},
			\end{equation*}
			where by collapsing the orbit $ \G/e $ to a point  we get the second map and the third one is the Pontryagin-Thom collapse map. Hence the map $ \partial_3 $ is given by 
			\begin{myeq}\label{tr form}
				\partial_3(1)=\res_{C_p}^\G  \tr_e^\G(1) = p^{m-1}[C_p/e], 
			\end{myeq}
			so $ \partial_3 $  is injective. Thus the homotopy groups in odd degree vanish and 
			\[
			\pi_i^\G(S^{\lambda + \lambda^p} \smas H\uA)=\begin{cases}
				A(\G)/\tr_{C_p}^\G(A(C_p)), & \text{if $i=$ 0}\\
				A(C_p)/(p^{m-1}[C_p/e]), & \text{if $i=$ 2}\\
				A(e) & \text{if $i=$ 4}.
			\end{cases}
			\]
		\end{exam}
		\vspace{.2cm}
		For a representation $ V= k_{m-1}\lambda^{p^{m-1}} + \cdots + k_0\lambda $ of $\G $,  the cellular decomposition  of $ S^V $ is given by \cite[Proposition 2.3]{HHR17}
		\begin{myeq}\label{celldecom}
			S^V= {\G/\G}_+ \smas e^0 \cup {\G/C_{p^{m-1}}}_+  \smas e^1\underset{1-g}{\cup}  {\G/C_{p^{m-1}}}_+  \smas e^2 \cup \cdots  \underset{1-g}{\cup} {\G/C_{p^{m-1}}}_+  \smas e^{2k_{m-1}}
		\end{myeq}
		\begin{equation*}
			{\cup} ~ {\G}/C_{p^{m-2}} \smas  e^{{2k_{m-1}} +1} \cup \cdots \underset{1-g}{\cup} \G/e_+  \smas e^{\dim(V)}.
		\end{equation*}
		
		The following  describes the coefficient ring $ \pi_{*}^{\G}(\Sigma^V H \uA) $. We write $ \lambda_i $ to mean $ \lambda^{p^i} $.
		
		\begin{theorem}\label{thm A cpm}
			Let $ \G=C_{p^m} $. The ring $ \pi^\G_{\bigstar^e}(H\uA) $ is generated over $ A(G) $ by the classes 
			$$
			a_{\lambda_0}, \cdots, a_{\lambda_{m-1}}, u_{\lambda_0}^{(e)}, \cdots, u_{\lambda_{m-1}}^{(C_{p^{m-1}})} 
			$$
			modulo the  relations \eqref{eq: a-relation in A-coeff}, \eqref{eq: u-relation for A-coeff} and \eqref{eq:au rel general}.
		\end{theorem} 
		
		\begin{proof}
			Assume by induction that the statement holds for any subgroup of $ \G $. Fix a representation $ V= \sum_{i=1}^s b_i\lambda_i \in \bigstar^e$ with $ \lambda_i\ne 1$. Then the cellular decomposition of $ S^V $ implies $ \pi_*^\G(S^V\smas H\uA) $ is non-zero only for $ 0\le *\le \dim V $. Let  $ \lambda_r \in V $ have the smallest stabilizer.
			We apply induction on the number of representations, that is, we assume the result for   $ W=V - \lambda_r $ and prove it for $ V $.
			We  observe that  in degrees $ *< |W| $  the  cellular decomposition of $ S^V $ coincides with the   cellular decomposition of $ S^W $. Hence  it suffices to establish this in degrees $ *=|W|, |W| +1  $ and $  |W|  +2 $. \par
			Let $ *=  \lvert  W \rvert +2$. We claim that $ \prod_{\lambda_i\in V} (u_{\lambda_i}^{(C_{p^i})} )^{b_i}$ generates $ \pi_{ |V|}^\G(S^V \smas H\uA )$. We know that $\pi_{ 2}^\G(S^{\lambda^k} \smas H\uA ) \cong A(C_{p^k}) $ is generated by $ u_{\lambda_k}^{(C_{p^k})} $. Suppose by induction $ \phi_{W}= \prod_{\lambda_i\in W} (u_{\lambda_i}^{(C_{p^i})})^{b_i}$ generates $ \pi_{ |W|}^\G(S^W \smas H\uA )$.
			The claim follows once we establish $ \phi_{W} \cdot  u_{\lambda_r}^{(C_{p^r})} $ generates $ \pi_{ |W| +2}^\G(S^V \smas H\uA) \cong A(C_{p^r}) $.
			For this consider the following  map
			\[
			\pi_{ |W|}^{\G}(S^{W}\smas H\uA)\otimes \pi_2^{\G}(S^{\lambda^r}\smas H\uA)\xrightarrow{\prod_{\lambda_i \in W} u_{\lambda_i}\cdot u_{\lambda_r}} \pi_{ |W|+2}^{\G}(S^{W+\lambda^r}\smas H\uA).
			\]
			Restricting this map to the orbit $ [\G/C_{p^r}] $, we see that the product of generators $ \pi_{ |W|}^{C_{p^r}}(S^{|W|}\smas H\uA) $ and $ \pi_2^{C_{p^r}}(S^2\smas H\uA) $ maps to  $1\in \pi^{C_{p^r}}_{|W| +2}(S^{|W|+2}\smas H\uA) $. Moreover, from the cellular structure the restriction map 
			\begin{equation*}
				\res^\G_{C_{p^r}} : \pi^\G_{|W|+2}
				(S^V \smas H\uA) \rightarrow \pi_{ |W| +2}^{C_{p^r}}(S^{|W| +2} \smas H\uA)
			\end{equation*}
			is an isomorphism. 
			So the image of $ \prod_{\lambda_i \in W} u_{\lambda_i} \cdot u_{\lambda_r} $ is   $1\in  \pi_{ |W| +2}^\G(S^V \smas H \uA) \cong A(C_{p^r})$. Proceeding similarly we obtain that $ \prod_{\lambda_i \in W} u_{\lambda_i} \cdot u_{\lambda_r}^{(C_{p^r})} $ generates $ \pi_{ |W| +2}^\G(S^V \smas H \uA) $, so does $ \phi_{W} \cdot  u_{\lambda_r}^{(C_{p^r})} $. \par
			
			Next, consider the following long exact sequence associated to the cofibre sequence
			\begin{equation*} 
				S(\lambda_r)_+ \rightarrow S^0 \rightarrow S^{\lambda_r} ,
			\end{equation*}
			\begin{myeq}\label{eq:les for Cp^n case}
				\cdots \rightarrow \pi_{*}^\G({S(\lambda_r)}_+ \smas S^W \smas H\uA) \xrightarrow{\psi_*} \pi_{*}^\G(S^W \smas H\uA) \xrightarrow {a_{\lambda_r}} \pi_{*}^{\G}(S^V \smas H \uA) \rightarrow \pi_{*-1}^\G(S{(\lambda_r)}_+ \smas S^W \smas H \uA) \rightarrow \cdots 
			\end{myeq}
			Let $ *= |W|+1 $ and let $ \lambda_l\in W  $ have the smallest stabilizer. The map
			\[
			\psi_{|W|}: A(C_{p^r})\cong \pi_{|W|}^\G({\G/C_{p^r}}_+\smas S^{|W|}\smas H\uA)\to \pi_{ |W|}^\G(S^W\smas H\uA)\cong A(C_{p^l})
			\]
			is given by $ \res_{C_{p^l}}^\G \tr_{C_{p^r}}^\G(1)= p^{m-l}[C_{p^l}/C_{p^r}]$ as observed in Example \ref{ex:cell decom Cp^n}. So injective. Hence $ \pi_{ |W| +1}^\G(S^V\smas H\uA) $ is zero.\par
			When $ *=|W| $, the last  term  in \eqref{eq:les for Cp^n case} is zero because of the following (this follows analogously as \cite[(4.4)]{BG20})
\[\upi^{C_{p^r}}_\alpha (H\uA) = 0, \mbox{ and } \upi_{\alpha-1}^{C_{p^r}}(H\uA) = 0 \implies \upi_\alpha^{C_{p^r}}(S(\lambda_r)_+ \wedge H\uA)=0.\]
Hence $ a_{\lambda_r} $ is surjective.\par
			Finally, to deduce that there are no further relations, first we observe  that the image of $ \psi_{|W|} $ is a product of $u$-classes and elements of $I_{\tr^G_{\ker(\lambda_r)}}$. Moreover, we may find an explicit representative for  each element of the  group $ \pi_{|W|}^G(S^W\smas H\uA)/\Im(\psi_{|W|})\cong \pi_{|W|}^\G(S^V\smas H\uA) $ with the help of relation \eqref{prop:au-rel} as follows: any combination of elements of the form $ \displaystyle \prod_{\lambda_i \in V\setminus \lambda_j } u_{\lambda_i}a_{\lambda_j}\in \pi_{|W|}^\G(S^V\smas H\uA) $ can be rewritten (using \eqref{eq:au rel general})  as 
			$ a_{\lambda_r}([C_{p^j}/C_{p^r}]\cdot \displaystyle \prod_{\lambda_i \in V\setminus \lambda_r } u_{\lambda_i}) $, which can be simplified further as
			\[
			a_{\lambda_r}([C_{p^j}/C_{p^r}]\cdot  \prod_{\lambda_i \in V\setminus \lambda_r } u_{\lambda_i})=a_{\lambda_r}(\res_{C_{p^l}}^{C_{p^j}}[C_{p^j}/C_{p^r}]  \prod_{\lambda_i \in V\setminus \lambda_r } u_{\lambda_i})=p^{j-l}a_{\lambda_r}\Big([C_{p^l}/C_{p^r}]  \prod_{\lambda_i \in V\setminus \lambda_r } u_{\lambda_i}\Big).
			\]
			This completes the proof.
		\end{proof}
		The following  describes $ \pi^\G_{\div}(H\uA) $.  We only need to consider the case $ p=2 $ as Theorem \ref{thm A cpm} covers the $ p $ odd case. 
		\begin{theorem}\label{thm A cpm div}
			Let $ \G=C_{2^m} $. The ring $ \pi^\G_{\bigstar_\div}(H\uA) $ is generated over $ A(G) $ by the classes 
			$$
			a_{\sigma}, a_{\lambda_0}, \cdots, a_{\lambda_{m-1}}, u_{\lambda_0}^{(e)}, \cdots, u_{\lambda_{m-1}}^{(C_{2^{m-1}})} 
			$$
			modulo the  relations \eqref{eq: a-relation in A-coeff}, \eqref{eq: u-relation for A-coeff} and \eqref{eq:au rel general}.
		\end{theorem} 
		\begin{proof}
			Fix $ V=\sum_{i=1}^s b_i\lambda_i+\sigma\in \bigstar_{\div}$, $ \lambda_i\ne 1 $. Suppose $ \lambda_r\in V $ have smallest stabilizer. 	As in Theorem \ref{thm A cpm}, we proceed  by induction on the number of representations assuming   the result for $ W=V-\lambda_r $. So it suffices to investigate $ \pi_*^\G(S^V\smas H\uA)  $ in degrees $ *= |W|, |W+1|,$ and $ |W|+2 $. Let $ \lambda_\ell\in W $ have the smallest stabilizer, and $ \sigma:\G/H\to \{\pm 1\} $ where $ [\G:H]=2 $. Since $ \tr^G_H $ is injective, the equivariant cellular chain complex yields  $ \pi_{1-\sigma}^\G( H\uA) =0$ and  $ \pi_{-\sigma}^\G( H\uA) =A(G)/A(H)\cong \Z $.\par  
			We claim $ \pi_{|V|}^\G(S^V\smas H\uA) =0$ if $ \sigma\in V $. Following \eqref{celldecom}, we get the equivariant cell complex 
			\[
			A(C_{2^r})\xrightarrow{\zeta_{|V|}}A(C_{2^r})\xrightarrow{\partial_{|W|+1}} A(C_{2^\ell})\xrightarrow{\zeta_{|W|}}A(C_{2^\ell})\xrightarrow{}\cdots \xrightarrow{\partial_2}A(H)\xrightarrow{\tr^G_H}A(G),
			\] 			
			where the map $\zeta_{|V|}  $ is induced by the map $ 1-g:\G/C_{2^r}\to \G/C_{2^r} $.
			\begin{equation*}
				\begin{tikzcd}
					\G/C_{2^r}\smas S^W \arrow[rr, "(1-g)\wedge \id"] \arrow[d, "\simeq"'] & &\G/C_{2^r}\smas S^W \arrow[d, "\simeq"] \\
					\G/C_{2^r}\smas S^{\dim W} \arrow[rr, "(1-g)\wedge \id"]                    &  & \G/C_{2^r}\smas S^{\dim W}.                    
				\end{tikzcd}
			\end{equation*}
			Since $ g $ acts on $ S^{\dim W} $ via a map of degree $ -1$, the bottom map of the above diagram becomes $ 1+g $. Hence $ \zeta_{|V|}  $ is the  multiplication by $ 2 $ map, which is injective. Hence our claim follows. This implies $  \pi_{|W|+1}^{\G}(S^V \smas H \uA)\cong A(C_{2^r})/2$.\par
			Let $ *=|W| $. We have the long exact sequence 
			\begin{equation*}
				\cdots \rightarrow  \pi_{|W|}^\G(S^W \smas H\uA) \xrightarrow {a_{\lambda_r}} \pi_{|W|}^{\G}(S^V \smas H \uA) \rightarrow \pi_{|W|-1}^\G(S{(\lambda_r)}_+ \smas S^W \smas H \uA) \rightarrow \cdots 
			\end{equation*} 		
			The first term is zero by the above claim, and the last term vanishes as $ \pi_{-i}^\G({\G/C_{2^r}}_+\smas H \uA) =0$ for $ i=1,2 $. Thus $ \pi_{|W|}^{\G}(S^V \smas H \uA)=0 $.\par
			Consider the long exact sequence 
			\begin{equation*}
				\xymatrix@R=0.5cm@C=0.5cm{
					0\to	\pi_{|W|+1}^H(S^{V-\sigma} \smas H\uA)
					\ar[d]^{\cong}\ar[r]^{\cdot 2} 
					& \pi_{|W|+1}^{\G}(S^{V-\sigma} \smas H \uA)
					\ar[d]^{\cong}\ar[r]^{a_\sigma} 
					& \pi_{|W|+1}^\G( S^V \smas H \uA)
					\ar[r] 
					& \pi_{|W|}^H(S^{V-\sigma} \smas H\uA)\to \cdots
					\\ 
					A(C_{2^r}) & A(C_{2^r}) & &
				}
			\end{equation*}
			The last term is zero by Theorem \ref{thm A cpm} as $ |W| $ is odd and $V-\sigma\in \bigstar^e  $. Hence by adjoining the generator $ a_\sigma $ to the generators list of $ \pi^\G_{\bigstar^e}(H\uA) $, we obtain all  generators for $ \pi^\G_{\bigstar_\div}(H\uA) $.
		\end{proof}
		
	\end{mysubsection}
	
	\begin{mysubsection}{The square free case}\label{subsecsqA}
		Let $ G=C_n=C_{p_1}\cdots C_{p_k} $ where $ p_i $ are distinct primes. The reason that the computation is relatively easier in this case than general cyclic groups is that for each subgroup $ H $, (which is also consequently of square free order) $ A(G)\xrightarrow{\res^G_H} A(H) $ is surjective. As a consequence, for a representation $ \lambda^d $, the notation $ u_{\lambda^d}^{(C_d)} $ may be replaced with our usual $ u_{\lambda^d} $ which stands for a single class satisfying the relation $ I_{\res^G_{C_d}}\cdot u_{\lambda^d}=0 $.\par
		Throughout this section,  for $I\subseteq \uk  $, $ I^c $  denote the complement $ \uk\setminus I $. Suppose $ |I|= \prod_{i\in I}p_i$. We write $C_I  $ to denote the group $ C_{|I|} $ and  $ \lambda^I $ means the representation $ \lambda^{|I|} $.\par
		Suppose $ n $ is even. Let $ \mcal{R}_n $ be the ring generated over the Burnside ring $ A(G) $ by the elements $ a_\sigma, u_{\lambda^I}, a_{\lambda^I}$ in $ \pi_{\bigstar_\div}^G(H\uA) $ satisfying the  relations 
		\[
		a_\sigma^2=a_{\lambda^{n/2}}, \quad			I_{\tr^G_{C_{n/2}}}a_{\sigma}=0, \quad 	I_{\tr^G_{C_I}}a_{\lambda^I}=0, \quad I_{\res^G_{C_I}}u_{\lambda^I}=0, 
		\]
		\[
		\Big[\frac{G}{C_{I^c}\cup (I\cap L)}\Big] u_{\lambda^I}a_{\lambda^L}=\Big[\frac{G}{C_{L^c}\cup (I\cap L)}\Big] u_{\lambda^L}a_{\lambda^I},
		\]
		for $ I,L\subseteq \uk $ and $ |I|,|L|\ne n $. In case  $n  $ is odd, then define $ \RR_n $ to be the same as above except that the class $ a_\sigma $ and the relation  $I_{\tr^G_{C_{n/2}}}a_{\sigma}=0  $ does  not appear. The main goal of this section is to show that the ring $ \pi_{\bigstar_\div}^G(H\uA)$ is  isomorphic to $ \RR_n $.\par
		For a $ G $ Mackey functor $ \uM $,  $ \uM_p =\res^G_{C_p}\uM$. 
		We briefly describe our approach.
		The unit map $ \eta_{\Z}: \uA_p\to \uZ_p $ for the commutative monoid  $ \uZ_p $,  gives rise to the short exact sequence of Mackey functors 
		\begin{myeq}\label{eq: ses of C_p mfunctors }
			0\to \bZ_p\xrightarrow{\kappa} \uA_p\xrightarrow{\eta_{\Z}} \uZ_p\to 0,
		\end{myeq}
		where  the map $ \kappa$ is given by $1\mapsto  p\cdot[C_p/C_p]-[C_p/e] $ at the orbit $  C_p/C_p$.
		The  map $ \uA_p\to \bZ_p $ classifying $ 1\in \Z=\Hom_{\MM_{C_p}}(\uA_p,\bZ_{p}) $, fits into  the short exact sequence 
		\begin{equation*}
			0\to \uZ^*_p\to \uA_p \xrightarrow{} \bZ_p\to 0.
		\end{equation*}
		The composite $ \bZ_p\to \uA_p\to \bZ_p $ is multiplication by $ p $.
		By \eqref{eq: ses of C_p mfunctors }, we have the following short exact sequence of Mackey functors
		\begin{myeq}\label{eq: ses of G=C_n mfunctors}
			0\to \bZ_{p_i}\boxtimes\uA_{I\setminus\{i\}}\boxtimes \uZ_{{I^c}}\xrightarrow{\kappa} \uA_{I}\boxtimes \uZ_{{I^c}} \xrightarrow{\eta_{\Z}} \uA_{I\setminus\{i\}}\boxtimes \uZ_{{I^c}\cup \{i\}}\to 0.
		\end{myeq}
		Now observe that  $ \uA_I\boxtimes\uZ_{I^c} $ is a commutative Green functor. Hence the map $ \RR_n\to \pi_{\bigstar_\div}^G(H\uA) $ induces 
		\begin{myeq}\label{phimap}
			\Phi: \RR_n\otimes_{A(G)} A(C_I)\to \pi^G_{\bigstar_\div}(H(\uA_I\boxtimes\uZ_{I^c})),
		\end{myeq}
		via the composition $ H\uA\to H(A_I\boxtimes\Z_{I^c}) $, where the right-hand side is an $ \RR_n $-module via the unit map $ \uA\to  \uA_I\boxtimes \uZ_{I^c}$, and $ A(G) $ acts on $ A(C_I) $ via the restriction map. We show that $ \Phi $ is an isomorphism.\par
		Define $ u_{\lambda^d}^{A_I\boxtimes\Z_{I^c}} $ to be the image of $  u_{\lambda^d} $ under the ring homomorphism   $ \pi^G_{2-\lambda^d}(H\uA)\to \pi^G_{2-\lambda^d}(H(\uA_I\boxtimes\uZ_{I^c})) $   induced by the unit map. This definition also applies to any commutative monoid $ \uM $.\par 
		The restriction map, $ \res: A(G)\to A(C_I) $, sends 
		\begin{myeq}\label{eq: restriction map A(G) to A(C_I)}
			\res(\alpha_{p_\ell})=\begin{cases}
				\alpha_{p_\ell} & \text{if $\ell\in I$}\\
				p_\ell & \text{if $\ell\not\in I$},
			\end{cases}
		\end{myeq}
		where $ \alpha_{p_\ell} $ is as defined  in \eqref{gensq}.	The relation   $ I_{\res^G_{C_d}}u_{\lambda^d}=0 $ implies 
		$$
		(\alpha_{p_i}-p_i)\,u^{A_I\boxtimes \Z_{I^c}}_{\lambda^d}=0, \quad     \textup{~for each~} i\in I {\textup{~such that~}}\ p_i\nmid d. 
		$$
		We simply write the elements $ u_{\lambda^d}^M $ or $ a_{\lambda^d}^M $ in $ \pi_{\bigstar_\div}(H\uM) $ as  $ u_{\lambda^d} $ or $ a_{\lambda^d} $ when there is no confusion  likely about which coefficient we are referring to. The following result is a consequence of \cite[Proposition 6.14]{BG20}.
		\begin{prop}\label{prop: ses of R_n modules}
			Let $I$ be a subset $\underline{k}$ and $i \in I$. For $\alpha \in RO(G)$, we have a short exact sequence of graded $ \RR_n $-modules
			\begin{equation*}
				0 \to \pi^G_{\bigstar_\div}(H(\bZ_{p_i}\boxtimes \uA_{I \setminus \{ i\}}\boxtimes \uZ_{I^c})) \xrightarrow{\kappa_*} \pi^G_{\bigstar_\div}(H(\uA_I\boxtimes\uZ_{I^c})) \xrightarrow{{\eta_{\Z}}_*} \pi^G_{\bigstar_\div}(H(\uA_{I\setminus \{i\}}\boxtimes\uZ_{{I^c}\cup \{i\}})) \to 0. 
			\end{equation*}
		\end{prop}
		We know that the ring $\pi_{\bigstar_\div}^G(H\uZ) $ is generated by  the classes $ u_{\lambda^d}^\Z, a_{\lambda^d}^\Z $ modulo the relations $\frac{n}{d} a_{\lambda^d}=0$, and $\frac{d}{\gcd(d,s)}a_{\lambda^s}u_{\lambda^d}= \frac{s}{\gcd(d,s)}a_{\lambda^d}u_{\lambda^s}$ by  \cite[Theorem 5.25]{BG20}, (also see \S \ref{secpositiveHZ}).
		\begin{prop}\label{lem: Cohomology of point with Z-coefficient}
			As an $ \mcal{R}_n$-module
			\[
			\pi_{\bigstar_\div}^G(H\uZ) \cong \mcal{R}_n{\otimes}_{A(G)} \Z.
			\]
		\end{prop}
		Let $ \alpha\in RO(G) $ and let $ \uM $ be a  $ G/C_p $-Mackey functor. We note from \cite[Proposition 6.12]{BG20}, that 
		\begin{myeq}\label{eq: iso of coh gr for <Z>_p coeff}
			\pi_{\alpha}^G(H(\bZ_p\boxtimes \uM))\cong \pi^{G/C_p}_{\alpha^{C_p}}({H\uM}).
		\end{myeq}
		We show that the ring $\pi_{\bigstar_\div}^G(H(\bZ_p\boxtimes \uM)) $ can be determined from $ \pi^{G/C_p}_{\bigstar_\div}(H{\uM}) $ by adjoining certain generators.
		\begin{prop}\label{prop: coh ring of bZ_pBoxM}
			Let $ \uM $ be a commutative  $ G/C_p $-Green functor.  We have
			\[
			\pi_{\bigstar_\div}^G(H(\bZ_p\boxtimes \uM))\cong \pi^{G/C_p}_{(\bigstar_\div)^{C_p}}({H\uM})\big[a_{\lambda^{j}}^{\bZ_p\msp\boxtimes\msp \uM}\mid p\nmid j\big].
			\]
			Under this isomorphism the elements $ u_{\lambda^d}^M $ and  $ a_{\lambda^d}^M $ correspond to $ u_{\lambda^{pd}}^{\bZ_p\msp\boxtimes\msp \uM}$ and $ a_{\lambda^{pd}}^{\bZ_p\msp\boxtimes\msp \uM}$, where $ p\nmid d $.
		\end{prop}
		
		\begin{proof} 
			By \eqref{eq: iso of coh gr for <Z>_p coeff}, if $ p\nmid d $ we have an isomorphism
			\[
			\pi^{G}_{\alpha+\lambda^{d}}({H(\bZpboxM)})\xrightarrow[\cong]{~a_{\lambda^{d}}~}\pi^{G}_{\alpha}({H(\bZpboxM)}).
			\]
			In this way for $ \alpha\in \bigstar_{\div} $, we may remove all the multiples of $ \lambda^d $ for $ p\nmid d $, and are  finally left with the $ C_p $-fixed part $ (\bigstar_\div)^{C_p} $. Thus,  we can  identify the rings
			\[ 
			\pi_{\bigstar_\div}^G(H(\bZ_p\boxtimes \uM))\cong \pi^{G}_{(\bigstar_\div)^{C_p}}(H({\bZpboxM}))\big[a_{\lambda^d}\mid p\nmid d\big].
			\]
			As the fixed point spectrum functor is lax monoidal \cite[Proposition V.3.8]{MM02},
			we conclude from \eqref{eq: iso of coh gr for <Z>_p coeff} that the following   rings are  isomorphic
			\begin{equation*}
				\pi^{G}_{(\bigstar_\div)^{C_p}}(H({\bZpboxM}))\cong \pi^{G/C_p}_{(\bigstar_\div)^{C_p}}(H{\uM}).
			\end{equation*}
		\end{proof}
		We now prove the main theorem.
		\begin{theorem}\label{Thm: Ring structure main Thm}
			The ring  $ \pi_{\bigstar_\div}^G(H\uA)\cong \mathcal{R}_n$,  and as an $ \mcal{R}_n$-module
			\[
			\pi^G_{\bigstar_\div}(H(\uA_I\boxtimes\uZ_{I^c}))\cong \mcal{R}_n{\otimes}_{A(G)} A(C_I).
			\]
			Furthermore,
			\[
			\pi^G_{\bigstar_\div}(H(\bZ_{p_i}\boxtimes \uA_{I \setminus \{ i\}}\boxtimes \uZ_{I^c}))\cong (p_i-\alpha_{p_i})\big(\mcal{R}_n{\otimes}_{A(G)} A(C_I)\big)
			\]
			as an ideal of $ \pi^G_{\bigstar_\div}(H(\uA_I\boxtimes\uZ_{I^c})) $. 
		\end{theorem}
		\begin{proof}
			We induct over the cardinality of $ I $, the case $ I=\emptyset $ follows by Proposition  \ref{lem: Cohomology of point with Z-coefficient}. Assume that the statement holds for $ I\setminus \{i\} $. We show that the middle horizontal map $ \Phi $ is an isomorphism in the following commutative diagram in which columns are exact.
			\begin{equation*}
				\begin{tikzcd}[scale cd=.90]
					0\ar[d]  & 0\ar[d]\\
					\II \ar[d]\arrow[r,"\psi"]&\pi^G_{\bigstar_\div}(H(\bZ_{p_i}\boxtimes \uA_{I \setminus \{ i\}}\boxtimes \uZ_{I^c}))\arrow[d,"\kappa_*"] \\  \RR_n\otimes_{A(G)} A(C_I)\arrow[r,"\Phi"]\ar[d]&\pi^G_{\bigstar_\div}(H(\uA_I\boxtimes\uZ_{I^c})) \arrow[d,"{\eta_{\Z}}_*"] \\ \RR_n\otimes_{A(G)} A(C_{I\setminus \{i\}}) \arrow[r,"\cong"]\ar[d]&\pi^G_{\bigstar_\div}(H(\uA_{I\setminus \{i\}}\boxtimes\uZ_{{I^c}\cup \{i\}})) \arrow[d,""] \\
					0&0
				\end{tikzcd}
			\end{equation*}	
			The induction hypothesis tells us that the bottom row is an isomorphism. The restriction map $ A(C_I)\to A(C_{I\setminus \{i\}})$ induces the ring
			homomorphism
			\[
			\RR_n{\otimes}_{A(G)} A(C_I)\to \RR_n{\otimes}_{A(G)} A(C_{I\setminus \{i\}}),
			\]
			and let  $ \II$ denote the kernel. One may use  an equation similar to \eqref{eq: restriction map A(G) to A(C_I)}  to see that  the ideal $ \II $ is generated by $ (p_i-\alpha_{p_i}) $. Restricting $ \Phi $ to $ \II $, we obtain the map $ \psi $. Next, we show that there exists a map in the reverse direction
			\begin{myeq}\label{eq: reverse map in main thm of ring structure}
				\varphi: \pi^G_{\bigstar_\div}(H(\bZ_{p_i}\boxtimes \uA_{I \setminus \{ i\}}\boxtimes \uZ_{I^c}))\to \II,
			\end{myeq}
			such that $ \varphi\circ \psi $ is identity. By Proposition \ref{prop: coh ring of bZ_pBoxM}, the left-hand side of \eqref{eq: reverse map in main thm of ring structure} reduces to a smaller subgroup, and we assume the result for it.  In order to construct the map $ \varphi $, we first identify $  \pi^G_{\bigstar_\div}(H(\bZ_{p_i}\boxtimes \uA_{I \setminus \{ i\}}\boxtimes \uZ_{I^c})) $ as a quotient of $ \RR_n\otimes_{A(G)} A(C_I) $.  The unit map $  \uA_{p_i}\to \bZ_{p_i} $  induces an $ \RR_n $-module homomorphism 
			from $ \pi^G_{\bigstar_\div}(H(\uA_I\boxtimes\uZ_{I^c}))  $ to $  \pi^G_{\bigstar_\div}(H(\bZ_{p_i}\boxtimes \uA_{I \setminus \{ i\}}\boxtimes \uZ_{I^c}))  $.  In particular, this gives an $\RR_n$-module map
			\[
			\gamma: \RR_n\otimes_{A(G)} A(C_I)\to \pi^G_{\bigstar_\div}(H(\bZ_{p_i}\boxtimes \uA_{I \setminus \{ i\}}\boxtimes \uZ_{I^c})),
			\]
			that sends  $ \alpha_{p_j}\otimes 1$ to $ \alpha_{p_j} $ if $ j\ne i $, $a_{\lambda^s}{}\otimes 1$ to $ a_{\lambda^s} $ for all $ s $, and $ u_{\lambda^d}\otimes 1$ to $ u_{\lambda^d} $ if $ p_i\mid d$, with kernel generated by the elements $u_{\lambda^d}\otimes 1$ such that $p_i\nmid d$ and $\alpha_{p_i}\otimes 1 $. This shows that  the map $ \gamma $ is surjective.
			Next, consider the diagram
			\begin{center}
				\begin{tikzcd}
					\RR_n\otimes_{A(G)} A(C_I)	 \arrow[rr, "\gamma"] \arrow[rd, "f"] &  & \pi^G_{\bigstar_\div}(H(\bZ_{p_i}\boxtimes \uA_{I \setminus \{ i\}}\boxtimes \uZ_{I^c}))  \arrow[ld, "\varphi"', dashed]
					\\
					& \II  &  
				\end{tikzcd}
			\end{center}
			Here the map $ f $ is defined by  $ 1\mapsto  (p_i-\alpha_{p_i}) $, and therefore it takes  the element $ \alpha_{p_i}\otimes 1$ to 0  by \eqref{genpm}. Also, because of the relations in \eqref{eq: u-relation for A-coeff}, it sends  $ u_{\lambda^d}\otimes 1$  to 0  when $ p_i\nmid d $.  Consequently, there exists a $ \RR_n $-module map $ \varphi $ such that $ f =\varphi\circ \gamma$. Since $ \gamma $ is also a ring map, we get $ \varphi(1) =(p_i-\alpha_{p_i})$.  In order to determine the map $ \psi $, we observe that the graded module map
			\[
			\kappa_*: \pi^{G}_{0}(H(\bZ_{p_i}\boxtimes \uA_{I\setminus\{i\}}\boxtimes \uZ_{{I^c}}) ) \to   \pi^{G}_{0}(H(\uA_I\boxtimes \uZ_{I^c})),
			\]
			sends $ 1$ to $(p_i-\alpha_{p_i}) $ (see \eqref{eq: ses of C_p mfunctors }), and $ \Phi(p_i-\alpha_{p_i})=(p_i-\alpha_{p_i}) $. These together with the fact that  $ \kappa_* $ is injective by Proposition \ref{prop: ses of R_n modules}, implies $  \psi(p_i-\alpha_{p_i})=1$. Hence, the map $ \psi  $ is an isomorphism, and applying the five-lemma, we conclude that the map  $ \Phi $ is a ring isomorphism. 
		\end{proof}
	\end{mysubsection}
	\section{The positive cone of $ \pi^G_\bigstar(H\uZ) $}\label{secpositiveHZ}
	We now  describe the positive cone of $ \pi^G_\bigstar(H\uZ) $, denoted by  $ \bigstar_\textup{div} $, which consists of  linear combinations of the form  $ \ell -\epsilon\sigma -(\sum_{d_i\mid n}~~b_i\mspace{2mu} \lambda^{d_i}) $  where $  \ell  \in \Z, b_i\ge 0$ and $ \epsilon\in \{0,1\} $. Let us begin with the part which does not contain the representation $ \sigma $,  that is  $ \pi_{\bigstar^e}^G(H\uZ) $.\par
	Let $ \hat{\RR}$ be the graded ring generated by the classes $ a_{\lambda^d}, u_{\lambda^d} $, where $ d $ is  a  divisor of $ n $, $d \ne n $ with relations
	\begin{myeq}\label{eq:a rel Z}
		\frac{n}{d}a_{\lambda^d} =0 \quad \textup{~and~} \quad 
		\frac{d}{\gcd(d,s)}a_{\lambda^s} u_{\lambda^d}=\frac{s}{\gcd(d,s)} a_{\lambda^d}u_{\lambda^s}.
	\end{myeq}
	That is, 
	\[
	\hat{\RR}=\Z[a_{\lambda^d},u_{\lambda^d}\mid  d \textup{~divides~} n, d\ne n]/(\frac{n}{d}a_{\lambda^d}, \frac{d}{\gcd(d,s)}a_{\lambda^s} u_{\lambda^d}-\frac{s}{\gcd(d,s)} a_{\lambda^d}u_{\lambda^s}).
	\]
	For $ V= \sum_{i=1}^sm_i\lambda^{d_i}$ with $ m_i\ge0 $  and $ \hm=2r \le \dim V$, define
	\begin{myeq}\label{ivnot}
		\begin{aligned}
			&
			I_{\hm-V}:=\{\textup{the set of all~}  L:=(l_1,\cdots,l_s )  \textup{~such that~}  0\le l_i\le m_i   \textup{~and ~}  \sum_{i=1}^s l_i=r\},\\
			&\textup{for~} L\in I_{\hm-V},\quad \textup{~let~}  J_{\hm-V}(L):=    \{i  \mid  m_i-l_i>0\}.
		\end{aligned}
	\end{myeq}
	A direct calculation gives us the following.
	\begin{lemma}\label{lcmgcd}
		With the notation \eqref{ivnot}, the group  $ \hat{\RR}_{\hm-V} $ is isomorphic to $ \Z $ if $ \hm=\dim V $, and if $ \hm<\dim V $ then it is cyclic  of order
		\[
		\lcm\big[k_{L}\mid L\in I_{\textup{\tiny $ \hm-V $}}\big] \quad \textup{~where~}\quad k_L:=\gcd\big[\frac{n}{d_i}\mid i\in J_{\textup{\tiny $ \hm-V $}}({L})\big].
		\]
	\end{lemma}
	\begin{proof}
		If $ \hm=\dim V $, the only monomials in the $ a_{\lambda^d} $, $ u_{\lambda^d} $ in degree $ \hm-V $ is the product $ \prod_{i=1}^{s} u_{\lambda^{d_i}}^{m_i}$. This implies $ \hat{\RR}_{\hm-V}\cong \Z $. If $ \hm<\dim V $, the monomials are indexed by $ I_{\hm-V} $ via $ L\in I_{\hm-V} $ corresponds to $ \prod_{i\in J_{\hm-V}(L)}a_{\lambda^{d_i}}^{m_i-l_i}\prod_{i=1}^su_{\lambda^{d_i}}^{l_i} $. The last relation of \eqref{eq:a rel Z} implies that this is cyclic of the required order.
	\end{proof}
	We elaborate the above calculation in the following example. 
	\begin{exam}
		Let $ G=C_{p^2q^2} $ and $ V=\lambda^p+\lambda^q+\lambda^{p^2}$ and $ \hm=2 $. (In terms of the notation \eqref{ivnot}, $ d_1=p, d_2=q, d_3=p^2 $. So 
		$$
		I_{\hm-V}=\{L_1=(1,0,0), L_2=(0,1,0), L_3=(0,0,1)\},
		$$
		\[
		J_{\hm-V}(L_1)=\{2,3\}, \quad J_{\hm-V}(L_2)=\{1,3\},  \quad J_{\hm-V}(L_3)=\{1,2\}.  
		\]
		The corresponding monomials in $ \hat{\RR}_{\hm-V} $ are $ a_{\lambda^p}a_{\lambda^q}u_{\lambda^{p^2}} $, $u_{\lambda^{p}} a_{\lambda^q}a_{\lambda^{p^2}}$, and  $a_{\lambda^p}u_{\lambda^{q}}a_{\lambda^{p^2}} $ of order $ pq, q, q^2 $, respectively.  The relations  \eqref{eq:a rel Z} say that  
		\[
		pa_{\lambda^p}a_{\lambda^q}u_{\lambda^{p^2}}=u_{\lambda^{p}} a_{\lambda^q}a_{\lambda^{p^2}}  \quad \textup{~and~} \quad p^2a_{\lambda^p}a_{\lambda^q}u_{\lambda^{p^2}}= qa_{\lambda^p}u_{\lambda^{q}}a_{\lambda^{p^2}}.
		\]
		This means that $ \hat{\RR}_{\hm-V} $ is cyclic of order   $ pq^2 $.
	\end{exam}
	The argument of Theorem \ref{thm A cpm} for $ \pi_{\bigstar^e}^G(H\uA) $ follows verbatim to prove the next result.
	\begin{thm}\label{thm Z cpm}
		Let $ G=C_{p^m} $. The ring $ \pi_{\bigstar^e}^G(H\uZ) $ is isomorphic to $ \hat{\RR} $.
	\end{thm}
	It is worthwhile to note that the proof of  Theorem \ref{thm Z cpm} is much easier than that of Theorem \ref{thm A cpm}. We demonstrate the inductive step in the example below.
	\begin{exam}
		Let  $\hm-V\in \bigstar^e $ for the group $ G=C_{p^m}$. 
		Write $ V=W+\lambda^{p^r} $ where $ \lambda^{p^r}\in V $ has smallest stabilizer. As the odd-dimensional groups are zero, we have  the long exact sequence 
		\begin{equation*}
			\cdots\to \pi_{\hm-W}^G(S(\lambda^{p^r})_+\smas H\uZ)\xrightarrow{\varphi} \pi_{\hm-W}^G(H\uZ)\xrightarrow{} \pi_{\hm-V}^G(H\uZ)\xrightarrow{} \pi_{\hm-W-1}^G(S(\lambda^{p^r})_+\smas H\uZ)\xrightarrow{}\cdots 
		\end{equation*}
		The first and last term may be identified to $ \pi_{\hm-|W|}^{C_{p^r}}(H\uZ) $ and $ \pi_{\hm-|W|-2}^{C_{p^r}}(H\uZ) $ because $ \res_{C_{p^r}}V=\dim (V) $ by our choice of $ \lambda^{p^r} $.
		If $ \hm\ne |W|, |W|+2 $,  the middle map, $ a_{\lambda^{p^r}} $ is an isomorphism.\par
		If $ \hm=|W| $, then the  last term is zero. The map $ \varphi $  is  multiplication by $p^{m-r} $.
		The inductive hypothesis implies $ \pi_{\hm-W}(H\uZ)\cong \Z $.  As a result $ \pi_{\hm-V}(H\uZ)\cong \Z/p^{m-r} $. This is consistent with  Lemma \ref{lcmgcd}. If $ \hm=|W|+2 $, then the second term is zero, and the last term is $ \Z $. This completes the proof.
	\end{exam}
	We may generalize this result further to the group $ G=C_n $.
	\begin{thm}\label{thm ring z ev}
		The sub-ring  $ \pi_{\bigstar^e}^G(H\uZ) $ of $ \pi_\bigstar^G(H\uZ) $ is isomorphic to $ \hat{\RR} $.
	\end{thm}
	\begin{proof}
		We observe that the inclusion map $\iota: \hat{\RR}\to  \pi_{\bigstar^e}^G(H\uZ)\subseteq \pi_{\bigstar}^G(H\uZ)$ is an isomorphism.
		Assume by induction that   the statement holds  for any subgroup of $ G $. Fix a representation $ V=\sum_{i=1}^s m_i\lambda^i\in \bigstar^e $ with $ \lambda^i\ne 1 $. Then the cellular decomposition of $ S^V $ yields $ \pi_*^G(S^V\smas H\uZ) $ is non-zero for $ 0\le *\le \dim V $. We first show $ \pi_{*-V}^G(H\uZ) =0$ if $ * $ is odd.
		Theorem \ref{thm Z cpm}  implies that  the odd-dimensional groups are zero for $ G=C_{p^m} $. Now suppose   $ |G| $ is divisible by two distinct primes $ p $ and $ q $. 
		Let $  x\in \pi_{*-V}^G(H\uZ)$. By the induction hypothesis   the groups $ \pi_{*-V}^{C_{n/p}}(H\uZ) $, $ \pi_{*-V}^{C_{n/q}}(H\uZ) $ are zero.  So $ \tr^G_{C_p}\res^G_{C_p}(x)=px=0 $ and $ \tr^G_{C_q}\res^G_{C_q}(x)=qx=0 $. Hence $ x=0 $.\par
		Next we show $ \iota $ is surjective. This is proved by induction on the number of representations. That is, we assume the result for $ W=V-\lambda^k $, and prove it for $ V $. Consider the following   long exact sequence  for $ \hm $ even
		\begin{equation*}
			0\to \pi_{\hm-W}(S(\lambda^k)_+\smas H\uZ)\to \pi_{\hm-W}(H\uZ)\xrightarrow{a_{\lambda^k}} \pi_{\hm-V}(H\uZ)\xrightarrow{\partial} \pi_{\hm-W-1}(S(\lambda^k)_+\smas H\uZ)\to 0.
		\end{equation*}
		We may identify $ \pi_{\hm-W-1}^G(S(\lambda^k)_+\smas H\uZ)$ with  $\pi_{\hm-W-2}^{C_k}( H\uZ) $.
		Under this identification, the map $  \pi^G_{\hm-V}(H\uZ)\to \pi_{\hm-W-2}^{C_k}( H\uZ) $ is  $ \res^G_{C_k} $.
		By induction, we know that $ \pi_{\hm-W-2}^{C_k}( H\uZ) $ is generated by the classes $ a_{\lambda^j}, u_{\lambda^j} $ where $ \lambda^j$ are sub-representations of $\res_{C_k} (W)$. We may write them  as $ \res^G_{C_k}(a_{\lambda^j}),  \res^G_{C_k}(u_{\lambda^j})$ where $ \lambda^j\in W $. Further, for an element  $ \psi \in  \pi^G_{\hm-W-2}( H\uZ) $, the $ \res^G_{C_k}(\psi \cdot u_{\lambda^k})= \res^G_{C_k}(\psi)$. 
		Moreover, for an element $ x \in \pi^G_{\hm-W}(H\uZ) $, $ x\cdot a_{\lambda^k} $ represents the image of $ \pi^G_{\hm-W}(H\uZ) $ inside $ \pi^G_{\hm-V}(H\uZ) $. Hence the classes $ a_{\lambda^i}, u_{\lambda^i} $, $ \lambda^i\in V $ generate $  \pi^G_{*}(S^V\smas H\uZ)$. Therefore, the map $ \iota $ is surjective.\par
		Finally we show that  order of $ \pi_{\hm-V}^G(H\uZ) $ is  same as $ \mathcal{R}_{\hm-V} $. It is enough to establish this for groups of prime power order. This is because  if $H=C_{{p}^{k}}$ and $\gcd(p,[G:H])=1 $, then $ \tr^G_H\res^G_H=\res^G_H\tr^G_H=[G:H] $, which is co-prime to $ p $. Hence to determine the $ p $-torsion in $ \pi_{\hm-V}^G(H\uZ)  $, it is enough to look for $ p $-torsion in $ \pi_{\hm-V}^H(H\uZ)  $. For groups of prime power order, this follows from Theorem \ref{thm Z cpm}.
	\end{proof}
	Next we determine the ring  $ \pi_{\bigstar_{\textup{div}}}^G(H\uZ) $. We may assume $ G =C_n$, $ n $ is even, else we are in the situation of Theorem \ref{thm ring z ev}.  Let $ \mathcal{R}$ denote the graded ring generated by the classes $ a_{\sigma}, a_{\lambda^d}, u_{\lambda^d} $, where $ d $ is  a  divisor of $ n $, $d \ne n $ with relations
	\begin{myeq}\label{eq:a rel z div}
		2a_\sigma=0,~~~  a_\sigma^2=a_{\lambda^{n/2}},~~~	\frac{n}{d}a_{\lambda^d} =0 \quad \textup{~and~} \quad 
		\frac{d}{\gcd(d,s)}a_{\lambda^s} u_{\lambda^d}=\frac{s}{\gcd(d,s)} a_{\lambda^d}u_{\lambda^s}.
	\end{myeq}
	That is, 
	\[
	\RR=\Z[a_{\sigma}, a_{\lambda^d},u_{\lambda^d}]/(2a_\sigma,~~~  a_\sigma^2-a_{\lambda^{n/2}},\frac{n}{d}a_{\lambda^d}, \frac{d}{\gcd(d,s)}a_{\lambda^s} u_{\lambda^d}-\frac{s}{\gcd(d,s)} a_{\lambda^d}u_{\lambda^s}).
	\]
	We extend Theorem \ref{thm ring z ev} to the general case in the theorem below.
	\begin{thm}\label{hzthm}
		The ring  $ \pi_{\bigstar_{\textup{div}}}^G(H\uZ) $  is isomorphic to $ \RR$.
	\end{thm}
	\begin{proof}
		We show that the  map $\iota: {\RR}\to  \pi_{\bigstar_\div}^G(H\uZ)$ is an isomorphism. Theorem \ref{thm ring z ev} already implies that $ \iota $ is  an isomorphism in gradings belonging to $ \bigstar^e $. It suffices to prove the theorem in grading $ V+\sigma $ for $ V\in \bigstar^e $.  Let $ H $ be the stabilizer of the representation $ \sigma $, namely the subgroup of  index $ 2 $.  Consider the long exact sequence 
		\begin{equation*}
			\cdots\to \pi_{\hm-V}^G(H\uZ)\xrightarrow{a_\sigma} \pi_{\hm-V-\sigma}^G(H\uZ)\xrightarrow{} \pi_{\hm-V-1}^H(H\uZ) \xrightarrow{\tr^G_H} \pi_{\hm-V-1}^G(H\uZ)\to \cdots
		\end{equation*}
		Let $ \hm $ be even. Then  $ \pi_{\hm-V-1}^H(H\uZ)=0 $ by Theorem \ref{thm ring z ev}. Hence the map $ a_\sigma $ is surjective.\par
		We now observe that  $ \pi_{\hm-V-\sigma}^G(H\uZ)=0 $ for $ \hm $ odd. This follows from the description of the order given in \eqref{ivnot} and knowing the transfer map $ \tr^G_H $. The order of $ \pi_{\hm-V-1}^G(H\uZ) $ is twice the order of $ \pi_{\hm-V-1}^H(H\uZ) $.
		As the $ \res^G_H:  \pi_{\hm-V-1}^G(H\uZ)\to \pi^H_{\hm-V-1}(H\uZ)$ is surjective,  an element of  $  \pi_{\hm-V-1}^H(H\uZ)$ which is  the form  $ u_{\lambda^{i_1}}\cdots u_{\lambda^{i_s}} a_{\lambda^{i_{s+1}}}\cdots  a_{\lambda^{i_{t}}} $, can be rewritten using the restriction as $ \res^G_H(u_{\lambda^{{i_1}'}}\cdots u_{\lambda^{{i_s}'}} a_{\lambda^{{i_{s+1}}'}}\cdots  a_{\lambda^{{i_t}'}}) $, where $\res_H (\lambda^{{i_{j}}'})= \lambda^{i_j}$. Since $ \tr^G_H\res^G_H $ is multiplication by $ 2 $, we get $ \tr^G_H(u_{\lambda^{i_1}}\cdots u_{\lambda^{i_s}} a_{\lambda^{j_{1}}}\cdots  a_{\lambda^{j_{t}}})=2 u_{\lambda^{{i_1}'}}\cdots u_{\lambda^{{i_s}'}} a_{\lambda^{{j_1}'}}\cdots  a_{\lambda^{{j_t}'}}$. 
		Hence the last map, $ \tr^G_H $ is injective. This shows 
		$ \pi_{\hm-V-\sigma}^G(H\uZ)=0 $ when $ \hm $ is odd.\par
		Finally, we observe that for $ \hm $ even  the group  $ \pi_{\hm-V-\sigma}^G(H\uZ) $ is cyclic of order $ 2 $ which matches  with the order of $ \RR_{\hm-V-\sigma} $.
		This completes the proof.
	\end{proof}
	\section{The stable homotopy in negative grading}\label{secneg}
	In this section we  determine the ring $\pi^G_{\bigstar_{\le 0}}(H\uZ)$, where $G=C_{p^m}$  and  $\bigstar_{\le 0}$ consists of those $\alpha\in RO(G)$ for which $|\alpha^H|\le 0$ for all $H\le G$.  Here, we  use the fact that $S^0\to H\uA$  is an equivalence in degree $\le0$, so that  $\pi_\alpha^G(S^0) \to \pi_\alpha^G(H\uA)$ is an isomorphism for $\alpha\in \bigstar_{\le0}$ \cite[Theorem 2.1]{Lew92}. Further, if we write
	\[
	C(G):=\text{Functions (conjugacy classes of subgroups of $G$, $\Z$)},
	\]
	the fixed point map 
	$$
	\mathcal{F}^G: \pi_\alpha^G(S^0)\to C(G), \quad  \mathcal{F}^G(\nu)(H)= \deg (\nu^H).
	$$
	is injective for $\alpha\in \bigstar_{\le 0}$ \cite[Proposition 1]{tDP78}.\par
	We first consider the case $G=C_{p^m}$, $p$ odd.	
	Denote the representation $\lambda^{p^i}$ by $\lambda_i$  and the subgroup  $C_{p^i}$ by   $H_i$ where $0\le i\le m$.  For computations with $H\uZ$, we use $\bigstar_{\pm \textup{div},\le 0}$ (written simply $\bigstar_{\le0}$) for linear combination of $\lambda_i$ such that the fixed points are of dimension $\le 0$. Any  $\alpha\in \bigstar_{\le 0}$ can be expressed as $\sum_{i=0}^m{\frac{-|\alpha^{H_i}|}{2}}(\lambda_{i-1}-\lambda_i)$.  The map $S^{\lambda_{i-1}}\to S^{\lambda_i}$ given by $z\mapsto z^p$, defines an element $ a_{\lambda_i/\lambda_{i-1}}\in \pi^G_{\lambda_{i-1}-\lambda_i}(H\uZ)$. Observe that 
	\begin{myeq}\label{fixa}
		\mathcal{F}^G(a_{\lambda_{i}/\lambda_{i-1}})_j=
		\begin{cases}
			p & \text{if $0\le j\le i-1$}\\
			0 & \text{if $ j=i$}\\
			1 & \text{if $ j> i$}.
		\end{cases}
	\end{myeq}
	\begin{defn}\label{defc}
		For $\alpha\in {\bigstar_{\le 0}}$, define   $\chi_\alpha= \prod_{i=0}^{m}(a_{\lambda_i/\lambda_{i-1}})^{-|\alpha^{H_i}|/2}\in \pi^G_\alpha(H\uZ)$. Here we identify the symbol $a_{\lambda_0/\lambda_{-1}}$ with $a_{\lambda_0}$.	
	\end{defn}
	\begin{defn}
		Let $\alpha\in \bigstar_{\le 0}$ be such that $|\alpha|=0$.
		Define  $\chi^G_{\alpha,i}\in \pi^G_\alpha(H\uZ)$  (or simply $\chi_{\alpha,i}$)
		\begin{myeq}\label{chidef}
			\chi_{\alpha,i}:=\tr^G_{H_i}(\chi_{\res_{H_i}(\alpha)}).
		\end{myeq}
	\end{defn}
	Iteratively using  \eqref{fixa}, we obtain
	\begin{myeq}\label{fix1}
		\mathcal{F}^{G}( \chi_{\alpha,{i}})_j=
		\begin{cases}
			p^{m-i}\cdot	p^{-\sum_{k=j}^{m} |\alpha^{H_k}|/2}& \text{if $j\le i $ and $|\alpha^{H_j}|=0$}\\
			0 & \text{otherwise}.
		\end{cases}
	\end{myeq}
	Our goal is to show that all $\chi_{\alpha,i}$ along with $a_{\lambda_0}$ generate $\pi^G_{\bigstar_{\le 0}}(H\uZ)$. Note that  $a_{\lambda_i}=(\prod_{j=1}^{i}a_{\lambda_j/\lambda_{j-1}}) \cdot a_{\lambda_0}$, and $\chi_{\lambda_{j-1}-\lambda_j,m}=a_{\lambda_j/\lambda_{j-1}}$. 
	First we investigate relations involving $\chi_{\alpha,i}$ and $\chi_{\alpha,j}$.
	For fixed $\alpha$ such that $|\alpha|=0$, we have the classes $\chi_{\alpha,0},\cdots, \chi_{\alpha,m}$.\par
	Applying the formula \eqref{fix1} for the group $H_i$, we obtain
	\begin{myeq}\label{rel11}
		\begin{aligned}
			|\alpha^{H_i}|=0 &\implies  \chi_{\alpha,{i-1}}=p\cdot\chi_{\alpha,i},\\
			|\alpha^{H_i}|=-2r< 0 &\implies \chi_{\alpha,{i}}=p^{r-1}\cdot\chi_{\alpha,i-1}.
		\end{aligned}
	\end{myeq}
	\begin{notation}\label{defi}
		We use the following notation.
		\begin{itemize}
			\item Let $I(\alpha)=(i_1,\cdots, i_r) $ where $0\le i_1<\cdots<i_r\le m$ be such that $|\alpha^{H_{i_j}}|=0$ but $|\alpha^{H_{i_j+1}}|\ne 0$.
			\item Let $S(\alpha)=(s_1,\cdots, s_r)$ where $i_{j-1}<s_j\le i_{j}$ be so that $|\alpha^{H_{s_j}}|=0$ but $|\alpha^{H_{s_j-1}}|\ne 0$. Sometimes we write $S(\alpha)=(s_1(\alpha),\cdots, s_r(\alpha))$ to avoid any confusion.
			\item Let 
			$$
			q_{k,l}(\alpha)=\sum_{j=i_l+1}^{i_k}-|\alpha^{H_j}|/2.
			$$
			For clarity,  we sometimes write it as $q_{i_k,i_l}(\alpha)$.
		\end{itemize}
	\end{notation}
	The equation \eqref{rel11} implies that all the $\chi_{\alpha,i}$, $0\le i\le m$ may be expressed as a combination of $\chi_{\alpha,i_1},\cdots, \chi_{\alpha,i_k}$ where $I(\alpha)=(i_1,\cdots,i_k)$.
	Applying \eqref{fix1} for the group $H_{i_k}$, we obtain
	\begin{myeq}\label{rel9}
		p^{i_{k-1}+q_{k,k-1}(\alpha)-s_{k}+1}\chi_{\alpha,{i_{k-1}}}-p^{i_k-s_{k}+1}\chi_{\alpha,{i_k}}=0.
	\end{myeq}
	\begin{exam}
		Let $G=C_{p^3}$ and $\alpha= 2\lambda_0-2\lambda_1$. So fixed points of $\alpha$ are $(0,-4,0,0)$. The  generators $\chi_{\alpha,i}$, $1\le i\le 4$ satisfy the relations $\chi_{\alpha,1}=p\chi_{\alpha,0}=p\chi_{\alpha,2}$ and $\chi_{\alpha,2}=p\chi_{\alpha,3}$ by \eqref{rel11}. Further, by \eqref{rel9}, $p(\chi_{\alpha,0}-p\chi_{\alpha,3})=0 $. We have from \cite[Exmaple 5.7]{BG21b}
		\[
		\pi_\alpha^{C_{p^3}}(H\uZ)=\Z\{\chi_{\alpha,3}\}\oplus \Z/p\{\chi_{\alpha,0}-p\chi_{\alpha,3}\}.
		\]
	\end{exam}
	\medskip
	The multiplicative relations between the $\chi_{\alpha,i}$ are easily obtained.  Let $s=\min (i,s)$. Then 
	\begin{myeq}\label{rel18}
		\chi_{\alpha,i}\cdot \chi_{\beta,s}= \tr^G_{H_i}\big(p^{m-i}\chi_\alpha\cdot  \tr^{H_i}_{H_s}(\chi_\beta)\big)= p^{m-i} \chi_{\alpha+\beta,s}.
	\end{myeq}
	In particular, we get $ \chi_{\alpha,m}\chi_{\beta,s}=\chi_{\alpha+\beta,s}$.
	In $\pi^G_{\bigstar_{\le 0}}(H\uZ)$, we also have the following relations between $a_{\lambda_0}$ and $\chi_{\alpha,{s}}$
	\begin{myeq}
		\begin{aligned}\label{rel3}
			p^{s} a_{\lambda_0}\cdot\chi_{\alpha,{s}}=\tr^G_{H_{s}}(p^{s}a_{\lambda_0}\cdot \chi_\alpha)=0 \quad  (\text{as~} \textup{for the group $G $,~}  p^m a_{\lambda_0}=0  \textup{~by~} \eqref{eq:a rel Z}).
		\end{aligned}
	\end{myeq}
	Let $\rho$ denote the set of relations of the form \eqref{rel9}, \eqref{rel18} and \eqref{rel3}. Define  
	\[
	\Gamma=\Z[a_{\lambda_0}, \chi_{\alpha,s}\mid \alpha\in \bigstar_{\le 0}, |\alpha|=0, 0\le s\le m],
	\]			 			 
	and 
	$$
	\mathcal{R}_{\le 0}=\Gamma/(\rho) .
	$$
	Here we can reduce the generating set of $\Gamma$ by considering only those $\chi_{\alpha,s}$ such that $s\in I(\alpha)$.
	Let $\mathcal{R}^0$ be the sub-ring of $\mathcal{R}_{\le 0}$  generated by the classes $\{\chi_{\alpha,s}\mid \alpha\in \bigstar_{\le 0}, |\alpha|=0, 0\le s\le m\} $. This is the part of $\mathcal{R}_{\le 0}$  in gradings $\alpha\in {\bigstar_{\le 0}}$ with $|\alpha|=0$. So $\mathcal{R}_{<0}= \mathcal{R}^0\{a_{\lambda_0}^k\mid k\ge 1\}$ is the subset of $\mathcal{R}_{\le 0}$ in gradings $\alpha\in {\bigstar_{\le 0}}$ with $|\alpha|<0$.
	\begin{thm}\label{neodd}
		The ring $\pi^G_{\bigstar_{\le 0}}(H\uZ)$ is isomorphic to $\mathcal{R}_{\le 0}$.
	\end{thm}
	\begin{proof}
		We proceed by induction  firstly on the order of the group $G$, that is, in the induction step we assume that the result holds for all proper subgroups of $G$. 
		Suppose   $\alpha=c_0\lambda_0+\cdots+c_{m-1}\lambda_{m-1}+\ell$ is a $G$-representation in $\bigstar_{\le 0}$. We perform an inner induction on $s$ such that $\lambda_s$ is minimal  positive representation in $\alpha$ (i.e., $c_i\le 0$ for $i\le s-1$ and $c_s>0$) and assuming the map $\mathcal{R}_{\le 0, \alpha-\lambda_s}\to \pi^G_{\alpha-\lambda_s}(H\uZ)$ is an isomorphism, we prove that  the map $\mathcal{R}_{\le 0, \alpha}\to \pi^G_{\alpha}(H\uZ)$ is an isomorphism.\par
		The inner induction begins with the case that  no positive multiples of $\lambda_i$ occur in $\alpha$, that is, $\alpha=c_0\lambda_0+\cdots+c_{m-1}\lambda_{m-1}+\ell$ with  $c_i\le 0$ for all $i$ and $\ell\le 0$. Such an $\alpha$ lies in the intersection of  $\bigstar_{\le 0}$ and  $\bigstar_{\textup{div}}$. By Theorem \ref{hzthm}, we know that the group  $\pi^G_{\alpha}(H\uZ)$  is  generated by $  \langle  a_{\lambda_0}^{-c_0}\cdots a_{\lambda_{m-1}}^{-c_{m-1}} \rangle$ if $\ell=0$, and is $0$ if $\ell<0$. Note we may write  $\mathcal{R}_\alpha=\mathcal{R}_\beta\cdot a_{\lambda_0}^{-\sum_{j=0}^{m-1}c_j}$,  where $\beta=c_0\lambda_0+\cdots+c_{m-1}\lambda_{m-1}-(\sum_{j=0}^{m-1}c_j)\cdot\lambda_0$.
		Note either $\beta=0$ or $I(\beta)=(0,m)$.
		If  $\beta=0$, then  $\chi_{\beta,m}=1$, and $\mathcal{R}_\alpha\cong \Z/{p^m}\{a_{\lambda_0}^{-c_0}\}\cong \pi^G_\alpha(H\uZ)$. In  the other case, $ a_{\lambda_0}\cdot \chi_{\beta,0}=0$ by \eqref{rel3}.  Now write $\beta=\sum_{s=1}^{m-1}(\sum_{i=s}^{m-1}c_i)(\lambda_s-\lambda_{s-1})$ so that $\chi_{\beta,m}= \prod_{s=1}^{m-1}\chi_{\lambda_{s-1}-\lambda_s,m}^{-\sum_{i=s}^{m-1}c_i}$ by \eqref{rel18}. So
		$$
		\chi_{\beta,m}\cdot a_{\lambda_0}^{-\sum_{j=0}^{m-1}c_j}
		= a_{\lambda_0}^{-c_0}\cdots a_{\lambda_{m-1}}^{-c_{m-1}}
		$$
		which implies  $\mathcal{R}_\alpha\cong \pi^G_{\alpha}(H\uZ)$.\par
		For the induction step, suppose $\beta=\alpha^{H_s}$= $c_s\lambda_s+\cdots+c_{m-1}\lambda_{m-1}+\ell$. We prove
		\begin{enumerate}
			\item [(A)] \namelabel{A} $\mathcal{R}_{\le 0, \alpha}\xrightarrow{a_{\lambda_0}\cdot \chi_{\lambda_0-\lambda_s,m}}\mathcal{R}_{\le 0, \alpha-\lambda_s}$ is surjective.
			\item [(B)] \namelabel{B} $K:=\text{Ker}(\mathcal{R}_{\le 0, \alpha}\xrightarrow{a_{\lambda_0}\cdot \chi_{\lambda_0-\lambda_s,m}}\mathcal{R}_{\le 0, \alpha-\lambda_s})=\begin{cases}
				0, & \text{if $|\beta|<0$}\\
				\langle   a_{\lambda_0}^{d}\cdot \chi_{\lambda_0-\lambda_1,m}^{-c_1}\cdots \chi_{\lambda_0-\lambda_{s-1},m}^{-c_{s-1}}\cdot \chi_{\beta,s} \rangle, &\text{if $|\beta|=0$},
			\end{cases}$
		\end{enumerate}
		where $|\alpha|=2(c_0+\cdots+c_{s-1})+|\beta|=-2d $.
		Once we establish (A) and (B), the theorem can be deduced as follows.  Consider  the diagram
		\begin{myeq}\label{diag1}
			\begin{tikzcd}
				0 \arrow[r]  &  K\arrow[r] \arrow[d] & \mathcal{R}_{\le 0, \alpha} \arrow[rr,"a_{\lambda_0}\cdot \chi_{\lambda_0-\lambda_s,m}"] \arrow[d] && \mathcal{R}_{\le 0, \alpha-\lambda_s} \arrow[r] \arrow[d] &0\\
				0 \arrow[r]           & \pi_{\alpha}^{H_s}( H\uZ) \arrow[r, "\tr^G_{H_s}"]           & \pi_{\alpha}^G(H\uZ) \arrow[rr, "a_{\lambda_s}"]          & & \pi_{\alpha-\lambda_s}^G( H\uZ)\arrow[r] & 0      
			\end{tikzcd}
		\end{myeq}
		The top row is a short exact sequence by \eqref{A} and \eqref{B}. The bottom row comes from the exact sequence 
		\begin{equation*}
			\cdots \to\pi_{\alpha-\lambda_s+1}^G( H\uZ)  \to \pi_{\alpha}^{G}(S(\lambda_s)_+\smas H\uZ) \xrightarrow{\varphi} \pi_{\alpha}^G(H\uZ) \xrightarrow{a_{\lambda_s}}        \pi_{\alpha-\lambda_s}^G( H\uZ)\to \pi_{\alpha-1}^{H_s}( H\uZ) \to \cdots 
		\end{equation*}
		Identifying $\pi_{\alpha}^{G}(S(\lambda_s)_+\smas H\uZ)\cong \pi_{\alpha}^{H_s}(H\uZ)$ (this follows from the same method as \cite[(4.4)]{BG20}), we see that the map $\varphi$ is $\tr^G_{H_s}$.  The group  $\pi_{\alpha-1}^{H_s}( H\uZ)$ is zero as all the fixed point dimensions of $\res_{H_s}(\alpha-1)$ are negative. Thus the bottom row of \eqref{diag1} is exact once we prove\par
		\begin{description}
			\item[(C)] \namelabel{C}  $\tr^G_{H_s}: \pi_{\alpha}^{H_s}( H\uZ)\to \pi_{\alpha}^G(H\uZ) $ is injective.
		\end{description}
		Note that $\res_{H_s}(\alpha)\in \bigstar_{\textup{div},\le 0}^{H_s}$. So the group $\pi_{\alpha}^{H_s}( H\uZ)$ is non-zero only when $|\beta|=0$, and in this case generated by the element $  a_{\lambda_0}^{-c_0}\cdots a_{\lambda_{s-1}}^{-c_{s-1}}$. Its image in $\pi_{\alpha}^G(H\uZ)$ is $\tr^G_{H_s}(a_{\lambda_0}^{-c_0}\cdots a_{\lambda_{s-1}}^{-c_{s-1}})=a_{\lambda_0}^{-c_0}\cdots a_{\lambda_{s-1}}^{-c_{s-1}}\tr^G_{H_s}(1)$, which is same as $a_{\lambda_0}^{-c_0}\cdots a_{\lambda_{s-1}}^{-c_{s-1}}\cdot \chi_{\beta,s} $.
		Therefore the theorem follows by the induction hypothesis and the five-lemma.\par
		{\textbf{Proof of \eqref{C}}}: Note that if $|\alpha^{H_s}|<0 $, then $\pi_{\alpha}^{H_s}( H\uZ)$=0. So we may assume $|\alpha^{H_s}|=0$. Let $\underline{c}=(c_0,\cdots ,c_{s-1})$. If $\underline{c}=0$, i.e., $c_0=\cdots =c_{s-1}=0$, then the composition
		\[
		\Z\cong \pi_\alpha^{H_s}(H\uZ)\xrightarrow{\tr^G_{H_s}}\pi_\alpha^G(H\uZ)\xrightarrow{\res^G_{H_s}}\pi_\alpha^{H_s}(H\uZ)\cong \Z
		\]
		is multiplication by $p^{m-s}$, hence $\tr^G_{H_s}$ is injective if $\underline{c}=0$. In the induction step, we assume  the result holds for $H_i$, $i\le s-1$, and all $ \underline{c}' $ such that  $\underline{c}\le \underline{c}'\le 0$,  (i.e., $c_i\le c_i'\le 0$ for all $i\le s-1$), and we prove it for $\underline{c}$.  Let $i$ be such that $c_i<0$. We have the following diagram
		\begin{equation*}
			\begin{tikzcd}
				\pi^{H_i}_{\alpha+\lambda_i}(H\uZ)  \arrow[d,"="'] \arrow[r, "\tr^{H_s}_{H_i}"] & \pi^{H_s}_{\alpha+\lambda_i}(H\uZ) \arrow[d,"\tr_{H_s}^{G}"] \arrow[r] & \pi^{H_s}_{\alpha}(H\uZ) \arrow[d,"\tr_{H_s}^{G}"] \arrow[r] & 0 \\
				\pi^{H_i}_{\alpha+\lambda_i}(H\uZ) \arrow[r,"\tr^{G}_{H_i}"]           & \pi^{G}_{\alpha+\lambda_i}(H\uZ) \arrow[r]           & \pi^{G}_{\alpha}(H\uZ) \arrow[r]           & 0
			\end{tikzcd}
		\end{equation*}
		By the induction hypothesis, two horizontals and the middle vertical transfer maps are injective, hence it follows that the right vertical $\tr^G_{H_s}$ map is injective.
		\par
		{\textbf{Proof of \eqref{A}}}: We have $\mathcal{R}_{\alpha-\lambda_s}=\mathcal{R}_\beta\cdot a_{\lambda_0}^{1-\frac{|\alpha|}{2}}$, where $\beta=\alpha-\lambda_s+(1-\frac{|\alpha|}{2})\lambda_0$. Then $I(\beta)=(0, i_2,\cdots, i_r )$ where $i_2>s$. So $\mathcal{R}_\beta$ is generated by the elements $\{\chi_{\beta,0}, \chi_{\beta,i_2}, \cdots, \chi_{\beta,i_r} \}$.
		We note from \eqref{rel3} and \eqref{rel18} that $ \chi_{\beta,0}\cdot a_{\lambda_0}=0$ and 
		\begin{equation*}
			a_{\lambda_0}^{1-\frac{|\alpha|}{2}}\cdot  \chi_{\beta,i_k}=a_{\lambda_0}\cdot \chi_{\lambda_0-\lambda_s,m}\cdot a_{\lambda_0}^{-\frac{|\alpha|}{2}}\cdot  \chi_{\beta+\lambda_s-\lambda_0,i_k}= a_{\lambda_0}\cdot \chi_{\lambda_0-\lambda_s}\cdot y
		\end{equation*}
		where $y=a_{\lambda_0}^{-\frac{|\alpha|}{2}}\cdot  \chi_{\alpha-\frac{|\alpha|}{2}\lambda_0,i_k}\in \mathcal{R}_\alpha$. Therefore
		the map in \eqref{A} is surjective.\par
		{\textbf{Proof of \eqref{B}}}: Suppose $|\beta|<0$. 
		As $c_i\le 0$ for $i< s$, we have that 
		$|\alpha|=-2d$ for some $d>0$, and $|\alpha^{H_i}|<0$ for $i\le s$. This implies $I(\alpha)=(\tilde{i}_1,\cdots , \tilde{i}_r)$ where $\tilde{i}_1>s$. First suppose $s>0$. Let $\hat{\alpha}=\alpha+d\lambda_0$ and $\gamma=\alpha-\lambda_s+(d+1)\lambda_0$.   This implies, $I(\hat{\alpha})=I(\gamma)=(i_1,\cdots,i_r)$ where $i_1=0$ and $i_l=\tilde{i}_{l-1}$. 
		We know that $a_{\lambda_0}\cdot \chi_{{\hat{\alpha}},0}=a_{\lambda_0}\cdot \chi_{{{\gamma}},0}=0$, and for $j\ge 2$,
		\begin{equation*}
			a_{\lambda_0}\cdot \chi_{\lambda_0-\lambda_s,m}\cdot (\chi_{\hat{\alpha},i_j}\cdot a_{\lambda_0}^d)=
			a_{\lambda_0}^{d+1}\cdot \chi_{\gamma,i_j}.
		\end{equation*}
		Also,  $q_{k,{k-1}}(\hat{\alpha})=q_{k,{k-1}}(\gamma)$ for $2<k$.
		Although $q_{2,1}(\hat{\alpha})$ and $q_{2,1}(\gamma)$ may be different,  equation \eqref{rel9} and \eqref{rel3} gives us
		\[
		p^{i_2-s_2+1}\chi_{\hat{\alpha},i_2}\cdot a_{\lambda_0}=p^{i_2-s_2+1}\chi_{\gamma,i_2}\cdot a_{\lambda_0}=0.
		\]
		Since for both $\alpha$ and $\alpha-\lambda_s$, all the relations involving the generators $\chi_{\hat{\alpha},i_k}\cdot a_{\lambda_0}^d$ and $\chi_{\gamma, i_k}\cdot a_{\lambda_0}^{d+1}$
		are exactly same, we conclude the case  $|\beta|<0$ and $s>0$. The case  $|\beta|<0$ and $s=0$ follows because multiplication by $a_{\lambda_0}$ is isomorphism if degree is negative.\par
		Now consider the case  $|\beta|=0$. Let $|\alpha|=-2d$. Suppose  $\hat{\alpha}=\alpha+d\lambda_0$ and $\gamma=\alpha-\lambda_s+(d+1)\lambda_0$. Let $S(\alpha)=(s_1,s_2,\cdots,s_r )$. Note that $I(\alpha)$ is of the form $(s,i_2,\cdots,i_r )$. Then $I(\gamma)=(0,i_2,\cdots,i_r)$, and 
		$$
		I(\hat{\alpha})=
		\begin{cases}
			(s,i_2,\cdots, i_r), & \text{if $s_1$ is 0 or 1}\\
			(0,s,i_2,\cdots,i_r), & \text{for $1<s_1\le s$}.
		\end{cases}
		$$
		Here we consider the second choice of $I({\hat{\alpha}})$, the first one is similar. Thus, $ \mathcal{R}_{\hat{\alpha}}$ is generated by  $\{\chi_{\hat{\alpha},0}, \chi_{\hat{\alpha},s}, \chi_{\hat{\alpha},i_2},\cdots, \chi_{\hat{\alpha},i_r}\}$ and $ \mathcal{R}_{\gamma}$ is generated by $\{\chi_{\gamma,0}, \chi_{\gamma,i_2},\cdots,\chi_{\gamma,i_r}\}$. 
		For $j\ge 2$, we obtain
		\[
		a_{\lambda_0}\cdot \chi_{\lambda_0-\lambda_s,m}\cdot (\chi_{\hat{\alpha},i_j}\cdot a_{\lambda_0}^d)=
		a_{\lambda_0}^{d+1}\cdot \chi_{\gamma,i_j}.
		\]
		Moreover, for $k\ge 3$, $q_{i_k,i_{k-1}}$ is  same for $\hat{\alpha}$ and $\gamma$. Note that, $a_{\lambda_0}\cdot \chi_{\lambda_0-\lambda_s,m}\cdot \chi_{\hat{\alpha},s}=0$, and
		\begin{equation*}
			\begin{aligned}
				a_{\lambda_0}^{d}\cdot \chi_{\lambda_0-\lambda_1,m}^{-c_1}\cdots \chi_{\lambda_0-\lambda_{s-1},m}^{-c_{s-1}}\cdot \chi_{\beta,s}=	a_{\lambda_0}^d\cdot \chi_{-c_1(\lambda_0-\lambda_1)-\cdots-c_{s-1}(\lambda_0-\lambda_{s-1}),m }\cdot \chi_{\beta,s}=a_{\lambda_0}^d\cdot \chi_{\hat{\alpha},s}.
			\end{aligned}
		\end{equation*}
		Next we compare $\chi_{\gamma,i_2}$ with $\chi_{\gamma,0}$. By equation \eqref{rel9} and \eqref{rel3}, we get 
		\begin{myeq}\label{wqz}
			p^{i_2-s_2(\gamma)+1}(\chi_{\gamma,i_2}\cdot a_{\lambda_0}^{d+1})=0.
		\end{myeq}
		We have the following relation involving  $\chi_{{\hat{\alpha}},s}$ and $\chi_{{\hat{\alpha}},i_2}$ by \eqref{rel9}
		\begin{myeq}\label{wre}
			p^{s+q_{i_2,s}(\hat{\alpha})-s_3({\hat{\alpha}})+1}\chi_{\hat{\alpha},s}\cdot a_{\lambda_0}^d-p^{i_2-s_3({\hat{\alpha}})+1}\chi_{\hat{\alpha},{i_2}}\cdot a_{\lambda_0}^d=0.
		\end{myeq}
		Observe that $s<s_3({\hat{\alpha}})\le i_2$ and  $s_3({\hat{\alpha}})=s_2(\gamma)$. The first term in \eqref{wre} lies in the kernel $K$. Hence the map $a_{\lambda_0}\cdot \chi_{\lambda_0-\lambda_s,m}$ sends the relation \eqref{wre} to the relation  \eqref{wqz}. Thus the kernel of $a_{\lambda_0}\cdot \chi_{\lambda_0-\lambda_s,m}$ is spanned by $\chi_{\hat{\alpha},s}$, which completes the proof of \eqref{B}.
	\end{proof}
	Next we consider  $G=C_{2^m} $ case. We may define the classes $\chi_{\alpha,s}$ in an analogous way as in \eqref{chidef} and Definition \ref{defi}. Along with $a_{\lambda_0}$, they  satisfy relations  \eqref{rel9}, \eqref{rel18} and \eqref{rel3}. In this case,   $a_\sigma\in \pi_{-\sigma}^G(H\uZ)$ is not a multiple of $a_{\lambda_0}$ and  satisfies 
	\begin{myeq}\label{rel27}
		2a_\sigma=0, \qquad 	a_{\sigma}^2=a_{\lambda_0}\chi_{\lambda_0-\lambda_{m-1},m}, \quad\textup{~and~}\quad a_\sigma\chi_{\alpha,s}=0 \textup{~for~} s\le m-1.
	\end{myeq}
	Let $ \bigstar^e_{\le 0}$ comprise those elements of  ${\bigstar_{\le 0}}$ which does not involve $\sigma$, i.e, of the form $\sum_{i=0}^{m-1}c_i\lambda_i+r $.
	\begin{exam}
		Note that $\pi^G_{\sigma-1}(H\uZ)=0$. This is because $S^{-\sigma}\simeq\textup{~Fib~}(S^0\to {G/{H_{m-1}}}_+)$. 
		This gives the long exact sequence 
		\[
		\cdots \to\pi_0^G( H\uZ)  \xrightarrow{\res^G_{H_{m-1}}} \pi_{0}^{{H_{m-1}}}( H\uZ) \xrightarrow{} \pi_{-1}^G(S^{-\sigma}\smas H\uZ) \xrightarrow{}        \pi_{-1}^G( H\uZ)\to  \cdots 
		\]
		Since the map $\res^G_{H_{m-1}}$ is identity, we have $\pi^G_{\sigma-1}(H\uZ)=0$. This implies $\chi_\alpha$ in Definition \ref{defc} works only for $\alpha\in {\bigstar_{\le 0}^e}$.
	\end{exam}
	We see some further relations in this case.
	\begin{prop}\label{propev}
		Let $\beta=\sum_{i=0}^{m-1}c_i\lambda_i+r +\sigma\in {\bigstar_{\le 0}}$ be such that $|\beta|=0$. Consider $\chi_{\beta,s}\in \pi_\beta^G(H\uZ)$ with $s<m$. Then
	\end{prop}
	\begin{myeq}\label{chitwo}
		(1) \,	2\chi_{\beta,s}=0, \quad \textup{~~}  \quad (2)\,\chi_{\beta,m-1}=0 \textup{~if~} |\beta^G|=-1.
	\end{myeq}
	\begin{proof}
		For (1), let $g$ be a generator of ${G/{H_{m-1}}}_+ $, which acts on $\upi^G_{\alpha}(H\uZ)(G/{H_{m-1}})$.  Let  $x\in \pi_\beta^{H_{m-1}}(H\uZ) $. Then 
		$$
		\res^G_{H_{m-1}}\tr^G_{H_{m-1}}(x)=x+g(x)=0,
		$$  
		as $g$ acts by $-1$, since $\beta$ contains a copy of $\sigma$. This implies,
		\[
		2\chi_{\beta,s}=\tr^G_{H_{m-1}}\res^G_{H_{m-1}}\chi_{\beta,s}= \tr^G_{H_{m-1}}\res^G_{H_{m-1}}\tr^G_{H_{m-1}}\chi_{\beta,s}^{H_{m-1}}=0.
		\]
		For (2), consider the  diagram
		\begin{myeq}\label{diagev}
			\xymatrix@R=0.5cm@C=0.5cm{
				\cdots\ar[r]&	\pi_{\beta-\sigma+1}^{G}(H\uZ)\ar[rr]^-{\res^G_{H_{m-1}}} \ar[d]_-{\res^G_{H_{m-1}}}
				&	& \pi_\beta^{H_{m-1}}(H\uZ) \ar[rr]^-{\tr^G_{H_{m-1}}} &
				& \pi_\beta^G(H\uZ)\ar[r]^-{a_\sigma}&\pi_{\beta-\sigma}^{G}(H\uZ)\ar[d]^-{\cong}\\ 
				&	\pi_{\beta-\sigma+1}^{H_{m-1}}(H\uZ)\ar[rru]_-{\cong}		 && &&& 0
			}
		\end{myeq}
		If $|\beta^G|=-1$, i.e, $|\beta-\sigma+1^G|=0$, then 
		$\res^G_{H_{m-1}}(\chi_{\beta-\sigma+1,m})=\chi_{\beta-\sigma+1,m-1}^{H_{m-1}}$. This implies (2).
	\end{proof}
	Let $\hat{\rho}$ denote the set of relations  \eqref{rel9}, \eqref{rel18},  \eqref{rel3},  \eqref{rel27}, and \eqref{chitwo}. 
	Define
	\[
	\hat{\Gamma }=\Z[a_\sigma, a_{\lambda_0}, \chi_{\alpha,s}\mid \alpha\in {\bigstar_{\le 0}}, |\alpha|=0, 0\le s\le m \textup{~with~} s\ne m \textup{~if~} \alpha\not\in {\bigstar_{\le 0}^e}]
	\]
	and 
	\[
	\hat{\mathcal{R}}_{\le 0}=\hat{\Gamma}/(\hat{\rho}).
	\]
	\begin{thm}\label{negen}
		The ring $\pi^G_{\bigstar_{\le 0}}(H\uZ)$ is isomorphic to $\hat{\mathcal{R}}_{\le0}$.
	\end{thm}
	\begin{proof}
		We note that 
		\[
		\hat{\mathcal{R}}_{\bigstar_{\le 0}^e}=\Z[ a_{\lambda_0}, \chi_{\alpha,s}\mid \alpha\in {\bigstar_{\le 0}^e}, |\alpha|=0, 0\le s\le m]/\langle  \rho \rangle 
		\]
		where $\rho$ are the relations generated by \eqref{rel9}, \eqref{rel18}, and  \eqref{rel3}. An entirely analogous argument as in Theorem \ref{neodd} implies $\hat{\mathcal{R}}_{{\bigstar_{\le 0}^e}}\cong \pi_{{\bigstar_{\le 0}^e}}^G(H\uZ)$.
		Now let $\alpha\in \bigstar_{\le 0}\setminus {{\bigstar_{\le 0}^e}}$  be of the form $\alpha=\sum_{i=0}^{m-1}c_i\lambda_i+r +\sigma$.\par 
		\textbf{Case 1}: $|\alpha|$ is odd. So $|\alpha^G|=r$ is even.
		Consider the diagram
		\begin{equation*}
			\begin{tikzcd}
				& & \hat{\mathcal{R}}_{\alpha+\sigma} \arrow[r,"a_\sigma"] \arrow[d,"\cong"] & \hat{\mathcal{R}}_{\alpha} \arrow[r] \arrow[d] &0\\
				0 \arrow[r]           & \pi_{{\alpha}+\sigma}^{H_{m-1}}( H\uZ) \arrow[r, "\tr^G_{H_{m-1}}"]           & \pi_{\alpha+\sigma}^G(H\uZ) \arrow[r, "a_{\sigma}"]          & \pi_{\alpha}^G( H\uZ)\arrow[r] & 0      
			\end{tikzcd}
		\end{equation*}
		Note that multiplication by $a_\sigma$ is surjective in the the top row. The bottom row arises from the exact sequence 
		\begin{equation*}
			\cdots \to\pi_{\alpha+1}^G( H\uZ)  \to \pi_{\alpha+\sigma}^{G}({G/H_{m-1}}_+\smas H\uZ) \xrightarrow{\varphi} \pi_{\alpha+\sigma}^G(H\uZ) \xrightarrow{a_{\sigma}}        \pi_{\alpha}^G( H\uZ)\to \pi_{\res{(\alpha)}}^{H_{m-1}}(  H\uZ) \to \cdots 
		\end{equation*}
		By the induction hypothesis, the group $\pi_{\res{(\alpha)}}^{H_{m-1}}(  H\uZ)=0$ as $\res(\alpha)\in {\bigstar^e_{\le 0}}$ and $|\res(\alpha)^{H_i}|$ is odd for all $i$.  An analogous argument as in \eqref{C} shows that the map $\varphi=\tr^G_{H_{m-1}}$ is injective. Hence the bottom row is exact. 
		Let $|\alpha+\sigma|=-2d\le 0$, and $\beta=\alpha+\sigma+d\lambda_0$. So $|\beta|=0$. Let $I(\beta)=(i_1,\cdots, i_r)$. Since  $\alpha+\sigma\in {\bigstar^e_{\le 0}}$, $\pi^G_{\alpha+\sigma}(H\uZ)\cong \hat{\mathcal{R}}_{\alpha+\sigma}$, and generated by the classes ${\{\chi_{\beta,i_1}\cdot a_{\lambda_0}^d, \cdots, \chi_{\beta,i_r}\cdot a_{\lambda_0}^d\}}$. If $i_r\ne m$, then all of them map to zero along the top row by \eqref{rel27}, also along the bottom row as they lie in the image of $\tr^G_{H_{m-1}}$. If $i_r=m$, then we get $\hat{\mathcal{R}}_\alpha\cong \Z/2\{\chi_{\beta,m}\hspace{.02cm}a_\sigma \hspace{.02cm} a_{\lambda_0}^d\}$ by \eqref{rel27}. 
		Also, we see that 
		$$
		2\chi_{\beta,m}=\tr^G_{H_{m-1}}\res^G_{H_{m-1}}\chi_{\beta,m}=\tr^G_{H_{m-1}}(\chi_{\beta,m-1}^{H_{m-1}})
		$$
		as $\res^G_{H_{m-1}}(\chi_{\beta,m})=\chi^{H_{m-1}}_{\beta,m-1}$. So $\pi_\alpha^G(H\uZ)\cong \hat{\mathcal{R}}_\alpha\cong0$.\par
		\textbf{Case 2}:  $|\alpha|$ is even.   This implies 
		$|\alpha^G|=r$ is odd, so $<0$. Let $|\alpha|=-2d\le 0$, and $\beta=\alpha+d\lambda_0$. So $|\beta|=0$. Consider the  diagram as in \eqref{diagev}.
		We see that for $s\le m-1$, 
		\[
		\res^G_{H_{m-1}}(\chi_{\beta-\sigma+1,s})=\res^G_{H_{m-1}}\tr^G_{H_{m-1}}(\chi_{\beta-\sigma+1,s}^{H_{m-1}})=(1+g)(\chi_{\beta-\sigma+1,s}^{H_{m-1}})=2\chi_{\beta-\sigma+1,s}^{H_{m-1}},
		\]
		here $g$ acts trivially because $\beta-\sigma+1\in {\bigstar_{\le 0}^e}$. This together with part (2) of Proposition \ref{propev} completes the proof.
	\end{proof}

\end{document}